\documentclass[a4paper,11pt]{amsart}
\usepackage[colorlinks,linkcolor=blue,citecolor=blue]{hyperref}
\usepackage{latexsym, amssymb, amsmath, amsthm, mathrsfs, bbm, amscd}
\usepackage{extarrows}
\usepackage[all, knot]{xy}
\xyoption{arc}

\usepackage{anysize}\marginsize{30mm}{30mm}{30mm}{30mm}
\def \tr{\triangleright}

\def \End{\operatorname{End}}

\def \Aut{\operatorname{Aut}}
\def \rank{\operatorname{rank}}
\def \Irr{\operatorname{Irr}}
\def \id{\operatorname{id}}

\def \rad{\operatorname{rad}}
\def \soc{\operatorname{soc}}

\def \GL{\operatorname{GL}}
\def \Gr{\operatorname{Gr}}
\def \Asym{\operatorname{AS}}
\def \diag{\operatorname{diag}}

\def \Z{\mathbb{Z}}

\def \k{\mathbbm{k}}
\def \B{\mathcal{B}}
\def \D{\mathcal{D}}
\def \M{\mathcal{M}}

\numberwithin{equation}{section}

\newtheorem{theorem}{Theorem}[section]
\newtheorem{lemma}[theorem]{Lemma}
\newtheorem{proposition}[theorem]{Proposition}
\newtheorem{corollary}[theorem]{Corollary}
\newtheorem{definition}[theorem]{Definition}
\newtheorem{lemma-definition}[theorem]{Lemma-Definition}
\newtheorem{example}[theorem]{Example}
\newtheorem{remark}[theorem]{Remark}

\begin{document}

\title[Extension of elementary $p$-groups and its application]{Extension of elementary $p$-groups and its application in classification of groups of prime exponent}

\subjclass[]{} \keywords{}

\author{Zheyan Wan, Yu Ye and Chi Zhang}

\address{(Wan)\quad Yau Mathematical Sciences Center, Tsinghua University, Beijing 100084, China} \email{wanzheyan@mail.tsinghua.edu.cn}

\address{(Ye)\quad School of Mathematical Sciences, University of Science and Technology of China, Hefei 230026, China; Wu Wen-Tsun Key Laboratory of Mathematics, USTC, Chinese Academy of Sciences, Hefei 230026, China} \email{yeyu@ustc.edu.cn}

\address{(Zhang)\quad Department of Mathematics, China University of Mining and Technology, Xuzhou 221116, China} \email{zclqq32@cumt.edu.cn}

\date{}
\maketitle

\begin{abstract} Let $p$ be a prime number and $\Z_p=\Z/p\Z$.
We study finite groups with abelian derived subgroup and exponent $p$ in terms of group extension data and their matrix presentations. We show a one-to-one correspondence between the following two sets: (i) the isoclasses of
class 2 groups of exponent $p$ and order $p^{m+n}$ and with derived subgroup $\Z_p^n$, and (ii)
the set $\Gr(n,\Asym_m(\Z_p))/\GL_m(\Z_p)$ of orbits of $\Gr(n,\Asym_m(\Z_p))$ under the congruence action by $\GL_m(\Z_p)$, where $\Gr(n,\Asym_m(\Z_p))$ is the set of $n$-dimensional subspaces of anti-symmetric matrices of order $m$ over $\Z_p$.
We give a description of the orbit spaces $\Gr(2, \Asym_m(\Z_p))/\GL_m(\Z_p)$ for all $m$ and $p$ by applying the theory of pencils of anti-symmetric matrices. Based on this, we show complete sets of representatives of orbits of $\Gr(3,\Asym_4(\Z_3))/\GL_4(\Z_3)$, $\Gr(4, \Asym_4(\Z_3))/\GL_4(\Z_3)$ and $\Gr(3, \Asym_5(\Z_3))/\GL_5(\Z_3)$. As a consequence, we obtain a classification of corresponding class 2 groups of exponent $p$. In particular, we recover the classification of groups with exponent 3 and order $\le 3^8$.
\end{abstract}

\let\thefootnoteorig\thefootnote
\renewcommand{\thefootnote}{\empty}


 \let\thefootnote\thefootnoteorig


\section{Introduction}
Throughout all groups considered are finite. $G$ always denotes a group and its derived subgroup is denoted by $G'$; $p$ is a prime number, and $\mathbb{Z}_p=\Z/p\Z$ is a cyclic group of order $p$ as well as a finite field with $p$-elements.

To find all the groups of a given order is one of the oldest and core problems in the theory of groups which dates back to late nineteenth century. Cayley (1878) call it the \emph{general problem} of groups. Amounts of work has been done on this general problem since, and many problems in the theory of (finite) groups are closely related to it, among which the classification of finite simple groups
and the Burnside problem and its variants, if not the only, are the most influential ones.

The Burnside problem, originally posed by William Burnside in 1902, asks given positive integer $m, n$, whether and when the free Burnside group $B(m,n)$ is finite, where $B(m,n)$ is the group generated by $m$ generators $x_1, \cdots, x_m$ subject to relations $x^n$ for each word $x$ in $x_1, \cdots, x_m$. The full answer to this problem is not known yet, although P. Novikov, S. Adian,  S. Ivanov and I. Lys\"{e}nko have shown the infiniteness of $B(m,n)$ for $m>1$ and sufficiently large $n$. Another invariant of the Burnside problem, namely the restricted Burnside problem was posed in 1930s. It asks whether there exists an upper bound for the order of finite groups with $m$ generators and exponent $n$. Recall that a group $G$ is said to have exponent $n$ if $n$ is the least positive number such that $g^n=1$ for all $g\in G$. The answer to the restricted Burnside problem is affirmative. In 1958 A. Kostrikin gave a solution in case $n$ is a prime number, and later in 1989, E. Zelmanov for arbitrary $n$.

The general problem can be deduced to the one to find all groups of given order and given exponent. Although it is known that there are only finitely many isoclasses of groups with given number of generators and exponent, it is still far from having a full answer to the general problem. The simplest case of the general problem is to determine all groups of prime-power order, or more restricted, the ones with prime exponent. A group of order $p^k$ for some prime $p$ and $k\le 1$ is usually called a $p$-group. A finite group of exponent $p$ is clearly a $p$-group. We are interested in the following basic problem in the theory of $p$-groups.

\noindent {\bf Problem.}
Determine all non-isomorphic groups of exponent $p$.

The classification of $p$-groups has been studied by many authors, see \cite{new,vish,wil,lee1,lee2,lee3,lee4} and the references therein. Newman gave a good survey to early development of this problem, and he also showed several methods to produce $p$-groups from the ones of lower order \cite{new}. Vishnevetskii obtained a classification of class 2 groups of exponent $p$ with derived groups of order $p^2$ \cite{vish}. In fact, he completely classified the groups of this form which cannot be expressed as a central product of two proper subgroups. Wilkinson gave a list of the groups of exponent $p$ and order $p^7$ for all $p$ \cite{wil}. Recently, Vaughan-Lee \cite{lee2}  studied the groups of order $p^8$  and exponent $p$ based on the $p$-group generation algorithm as introduced in \cite{ob} and give the formula to  calculate the number of the groups of order $p^8$  and exponent $p$ $(p > 7)$.

All known methods to construct $p$-groups rely heavily on the fact that a $p$-group is nilpotent, so that it can be written as iterated extension by elementary $p$-groups or more specifically by $\Z_p$'s, the cyclic group with $p$-elements.

Let $H$ and $K$ be groups. Recall that an extension of $H$ by $K$ is by definition an exact sequence $1\to K\to G\to H\to 1$, and $K$ is called the kernel of the extension. To find all extensions of $H$ by $K$ for given $H$ and $K$ is called the extension problem. With the classification of finite simple groups, a full solution to the extension problem will give a full solution to the general problem of groups for every finite group has a composite series.

The extension problem is very hard in general. However, if the extension has abelian kernel, say $K$ is an abelian group, then the conjugate action of $G$ on $K$ induces an action of $H$ on $K$, and
the famous Schreier theorem says that the equivalence classes of extensions of $H$ by $K$ is in one-to-one correspondence with $H^2(H, K)$, the second cohomology group of $H$ with coefficients in $K$.
Thus to find all extensions of $H$ by $K$, it is equivalent to find all possible actions $\rho\colon H\to \Aut(K)$ and to calculate $H^2(H,K)$ for each $\rho$.

Now let $0\to A\to G\to H\to 1$ be an extension with abelian kernel, where $A$ is viewed as an additive group with the identity element denoted by 0. Let $\rho\colon H\to \Aut(A)$ be the homomorphism induced by the $H$-action on $A$, and $\varphi\in H^2(H, A)$ the corresponding cohomology class. Such a quadruple $(H, A, \rho, \varphi)$ is called a \emph{group extension datum}, and the extension $0\to A\to G\to H\to 1$ is called a \emph{realization} of the datum. Any two realizations of a given datum are isomorphic. Moreover, given a datum $(H, A, \rho, \varphi)$, there exists a canonical realization $0\to A\to A\rtimes_{\rho,\varphi}H\to H\to 1$, where the middle group $G(H, A, \rho, \varphi)=A\rtimes_{\rho,\varphi}H$ is a generalization of semiproduct and called the \emph{realizing group} of the datum.
Recall that by definition, two extensions with abelian kernel are equivalent if and only if the corresponding extension data are isomorphic. We refer to Section 3.2 for more detail.

Recall that $G$ is solvable if and only if it has a subnormal series with abelian factor groups, i.e.,  $$G\cong A_r\rtimes_{\rho_r,\varphi_r}(A_{r-1}\rtimes_{\rho_{r-1},\varphi_{r-1}}(\cdots (A_1\rtimes_{\rho_1,\varphi_1} A_0)\cdots))$$ for some abelian (or even cyclic) groups $A_0,  \cdots, A_r$, and some actions $\rho_1, \cdots, \rho_r$ and 2-cocycles $\varphi_1, \cdots, \varphi_r$.

It is highly nontrivial to classify all modules for a given group in general, even for abelian groups. Fortunately, all groups concerned in this work are nilpotent and we need only to consider trivial actions. In fact, $G$ is nilpotent if and only if  it is obtained by iterated central extensions, or equivalently, $$G\cong Z_s\rtimes_{\varphi_s}(Z_{s-1}\rtimes_{\varphi_{s-1}}(\cdots (Z_1\rtimes_{\varphi_1} Z_0)\cdots))$$ for some abelian (or even cyclic) groups $Z_0, \cdots, Z_{s}$ with trivial actions and some 2-cocycles $\varphi_1, \cdots\varphi_s$. Here we simply write $\rtimes_{1,\varphi}$ as $\rtimes_{\varphi}$. The price is that the length $s$ may be much bigger than the one $r$ if we allow nontrivial actions as above.

The calculation of the cohomology group $H^2(H, A)$ is another big problem. The cohomology group of finite abelian groups can be calculated by using the K\"unneth formula or the well known Lyndon-Hochschild-Serre spectral sequence \cite{lyn, hs}, while for nonabelian groups quite a few is known yet. Moreover, for our purpose we need to understand the structure (the group table, group presentation ect.) of the realizing group. Then only knowing the group structure of $H^2(H, A)$ is not enough, one needs an explicit formula for the 2-cocylces. To find the cocycle formula is far more difficult than the calculation of the cohomology group. In fact, for a finite abelian group $H$, the cohomology group $H^*(H, A)$ has been long known, while the cocycle formula just appeared very recently, see for instance \cite{hlyy,hwy}.

Another problem should be taken into account is that nonequivalent extensions can give isomorphic middle group, c.f. \cite[Exercise 7.40]{rot}. Thus we also need to answer the question when the middle groups of two nonequivalent extensions are isomorphic.

In this paper, we mainly deal with extensions of the form $0\to A\to G\to H\to 1$ such that $A$ and $H$ are both abelian groups and $G$ has prime exponent $p$. Note that $H$ is abelian if and only if $A\ge G'$, the derived subgroup of $G$. Clearly $G$ has exponent $p$ implies that both $A$ and $H$ have exponent $p$, and hence $A\cong \Z_p^n$ and $H\cong \Z_p^m$ for some $m, n\ge 0$. In this case, $A$ can be viewed as an $n$-dimensional representation of $H$ over $\Z_p$, or equivalently an $n$-dimensional module of the algebra $\Z_p[T_1, \cdots, T_m]/\langle T_1^p, \cdots, T_m^p\rangle$, and hence the representation theory of finite dimensional algebra may apply. Moreover, since $H$ is a finite abelian group, the cocycle formula obtained in \cite{hlyy} enables the calculation of the structure of the middle group $G$.

First we consider the problem when two group extension datum have isomorphic realizing group. By introducing an equivalence relation on the set of group extension data, we can give an answer to this problem: under certain mild condition, two group extension data are equivalent if and only if their realizing groups are isomorphic.

\begin{proposition}[Proposition \ref{prop-equidata-isogp} and Corollary \ref{cor-equidata-isogp}] Let $\D = (H, A, \rho, \varphi)$ and $\tilde{D}=(\tilde H,\tilde A, \tilde\rho, \tilde\varphi)$ be two data, and let $G=A\rtimes_{\rho, \varphi}H$ and $\tilde G=\tilde A\rtimes_{\tilde \rho, \tilde\varphi} \tilde H$. Assume that one of the following holds:
\begin{enumerate} \item $A=G'$ and $\tilde A=\tilde G'$;
\item $A= Z(G)$ and $\tilde A= Z(\tilde G)$.
\end{enumerate}
Then $G$ and $\tilde{G}$ are isomorphic if and only if $\mathcal D$ and $\tilde {\mathcal D}$ are equivalent.
\end{proposition}

We describe the derived subgroup $(A\rtimes_{\rho,\varphi}H)'$, and consequently we obtain an equivalent condition for a group extension datum satisfying $A=(A\rtimes_{\rho,\varphi}H)'$, cf. Proposition \ref{prop-der-abext} and Corollary \ref{cor-der-abext}. Moreover, by using the cocycle formula for abelian groups, we give criteria for a realizing group having exponent $p$, see Theorem \ref{thm-expp} and Corollary \ref{cor-expp2}. Thus combined with the above proposition, the isoclasses of groups of exponent $p$ whose and with abelian derived subgroup are in one-to-one correspondence with the equivalence classes of certain group extension data.

It is well known that the nilpotency classs of a group of exponent $p$ is no greater than $p$ in case $p=3$, while the statement is not true for general $p$, cf. Section 2.2. However, the statement still holds true for groups which can be expressed as extension of two elementary $p$-groups, cf. Proposition \ref{prop-cn-mgp}.


One of the main purposes of this paper is to introduce the notion of matrix presentations for group extension data of the form $(\Z_p^m, \Z_p^n, \rho, \varphi)$, so that we can restate the classification problem of groups of exponent $p$ as a problem in linear algebra over the field $\Z_p$. Note that one can read a presentation of the realizing group of a data from its matrix presentation, cf. Proposition \ref{prop-pres-from-mrep}. In particular, we establish a connection between isoclasses of class 2 groups of exponent $p$
and the orbits of certain Grassmannian of anti-symmetric matrices under the congruence action.

\begin{theorem}[cf. Theorem \ref{thm-data-1to1-grorbits}] Let $m, n\ge 1$ and $0\le d\le n$ be integers. Then there exist one-to-one correspondences between:
\begin{enumerate}
\item $\mathcal E_p(d;m,n)$ and $\Gr(d, \Asym_m(\Z_p))/\GL_m(\Z_p)$;
\item $\mathcal E_p(m,n)$ and $\Gr(\le n, \Asym_m(\Z_p))/\GL_m(\Z_p)$;
\item $\mathcal G_2(p;m,n)$ and $\Gr(n, \Asym_m(\Z_p))/\GL_m(\Z_p)$.
\end{enumerate}
\end{theorem}

Note that $\mathcal G_2(p;m,n)$ denotes the set of isoclasses of class 2 groups of order $p^{m+n}$ and exponent $p$ with derived subgroup isomorphic to $\Z_p^n$; $\mathcal E_p(m, n)$ denotes the equivalence classes of data of the form $\D=(\Z_p^m, \Z_p^n, 1, \varphi)$ with $G(\D)$ of exponent $p$; and $\mathcal E_p(d; m, n)$ denotes the subclasses with $(G(\D))' = \Z_p^d$ for $0\le d\le m$.

Let $M_m(\Z_p)$ be the space of square matrices of order $m$ over $\Z_p$, and $\Asym_m(\Z_p)$ be the subspace consisting of  anti-symmetric ones. Let $\Gr(d, \Asym_m(\Z_p))$ and $\Gr(\le d, \Asym_m(\Z_p))$ denote respectively the set of $d$-dimensional subspaces and the set of subspaces of dimension less than or equal to $d$. Then
the congruence action of $\GL_m(\Z_p)$ on $\Asym_m(\Z_p)$ induces an action of $\GL_m(\Z_p)$ on $\Gr(d, \Asym_m(\Z_p))$ as well as on $\Gr(\le d, \Asym_m(\Z_p))$. We use $\Gr(d, \Asym_m(\Z_p))/\GL_m(\Z_p)$ and $\Gr(\le d, \Asym_m(\Z_p))/\GL_m(\Z_p)$ to denote the orbit spaces respectively.

By the above theorem, to find class 2 groups of exponent $p$ is equivalent to the problem to find a complete set of representatives of the orbits of  $\Gr(d, \Asym_m(\Z_p))$ under the congruence action of $\GL_m(\Z_p)$. The latter may have its own interest from different perspectives.

In fact, if $d=1$, then it is equivalent to find the normal form of a nonzero $m\times m$ anti-symmetric matrix under the congruence action. The answer has been long known: any such a matrix is congruent to a unique matrix of the form
 \[\diag\left(\begin{pmatrix}
                                          0 & 1 \\
                                          -1 & 0 \\
                                        \end{pmatrix},
                                        \cdots,
                                         \begin{pmatrix}
                                           0 & 1 \\
                                           -1 & 0 \\
                                         \end{pmatrix},
                                         0, 0, \cdots, 0
\right). \]

If $d=2$, then the problem is closely related but not equivalent to the problem to find the canonical form of pencils of anti-symmetric matrices. The canonical form of pencils of anti-symmetric matrices is well-known, see for instance \cite[Chap XII]{gan}. However, in our case we need further to consider the $\GL_2(\Z_p)$ action on the set of canonical forms of pencils.

In the case $d\ge 3$, no results are known to our knowledge.

We can give a description of $\Gr(n, \Asym_m(\Z_p))/\GL_m(\Z_p)$ for $n=1, 2$, and hence obtain a classification of $\mathcal G_2(p;m,1)$ and $\mathcal G_2(p;m,2)$.

\begin{theorem}(1)[cf. Theorem \ref{thm-1dim-dergp}]
The set $\{W_{m,k}\mid 1\le k\le \frac m 2\}$ gives a complete set of representatives of $\Gr(1, \Asym_m(\Z_p))/\GL_m(\Z_p)$, and $\{G_{m,k}\mid 1\le k\le \frac m 2\}$ gives a complete set of representatives of $\mathcal G_2(p;m,1)$.

(2) [cf. Theorem \ref{thm-2dim-dergp}] There exists a bijection between $\Gr(2, \Asym_m(\k))/\GL_m(\k)$ and the set of equivalence classes of canonical forms of pencils of anti-symmetric matrices. Consequently, $\mathcal G_2(p; m, 2)$ is in one-to-one correspondence with the equivalence classes of canonical forms.
\end{theorem}

We refer to Section 5.3 and 6.1 for unexplained notions.
We mention that in the above theorem, (1) should be well known to experts; and (2) generalizes to some extend the result obtained by  Vishnevetskii in \cite{vish}, which provides a classification of groups in $\mathcal G_2(2; m, n)$ that can not be expressed as a central product of proper subgroups.

\begin{theorem}[cf. Theorem \ref{thm-3of4}, Theorem \ref{thm-4of4}, and Theorem \ref{thm-3of5}] We obtain a complete set of representatives of the orbits of $\Gr(3, \Asym_4(\Z_3))/\GL_4(\Z_3)$, $\Gr(4, \Asym_4(\Z_3))/\GL_4(\Z_3)$ and $\Gr(3, \Asym_5(\Z_3))/\GL_5(\Z_3)$.
\end{theorem}

The proof is given by traditional ``hand" calculation. Consequently, we obtain a classification of groups of exponent $3$ and of order up to $3^8$, which recovers the results obtained in \cite{wil}  and \cite{lee2} in case $p=3$, cf. Theorem \ref{thm-ord1to7} and Theorem \ref{thm-ord8}.
Note that we use different methods here. It is not known to us whether our method can be used to improve the existing $p$-groups generation algorithms.

The paper is organized as follows.

In Section 2, we recall some basics on groups and modules. In Section 3, we introduce the notion of group extension datum, which is a reformulation of extensions of group with abelian kernel. We show that under some mild assumption, two group extension data are equivalent if and only if the realizing group are isomorphic. We also discuss the derived subgroup and center of the realizing group of an abelin datum.

In Section 4, by using the cocycle formula for finite abelian groups, we show a criterion for the realizing group of an abelian datum having exponent $p$. We also show when the derived subgroup of the realizing group of an abelian extension is equal to the kernel. We give a bound to the nilpotency class of the realizing group of a $p$-elementary data.

In Section 5, we introduce matrix presentations for a $p$-elementary data, and show that when two matrix presentations give equivalent data. As a consequence, we show that class 2 groups of exponent $p$ are in one-to-one correspondence with congruence classes of subspaces of anti-symmetric matrices over $\Z_p$.

In Section 6, we give a description of the congruences classes of 2-dimensional subspaces of $\Asym_m(\Z_p)$. Consequently we can find all groups of exponent $p$ and with derived subgroup of rank 2. As an easy consequence, we give a complete set of groups of order $3^{m}$ and exponent $3$ whose derived subgroup is isomorphic to $\Z_3^2$ for $m\le 8$. Such groups are shown to have nilpotency class 2.

In Section 7 and Section 8, we give a complete set of representatives of isoclasses of groups of exponent $3$ and order $\le 3^8$.

In the last two sections we calculate the orbits of $\Gr(n, \Asym_m(\Z_3))/\GL_m(\Z_3)$ for small $m, n$. More precisely, we give representatives of all congruence classes of 3 and 4-dimensional subspaces of $\Asym_4(\Z_3)$ in Section 9, and the ones of 3-dimensional subspaces of $\Asym_5(\Z_3)$ in Section 10.

\section{Basics on groups and modules}
We recall some basic facts on finite groups and modules in this section.

\subsection{Nilpotent group and nilpotency class} Let $G$ be a group.
For elements $g, h\in G$, by the \emph{commutator} of $g$ and $h$ we mean the element $[g,h]=ghg^{-1}h^{-1}$. Clearly $[g,h]=1$ if and only if $gh=hg$. For subsets $X, Y\subset G$, we use $[X,Y]$ to denote the subgroup of $G$ generated by elements of the form $[g,h]$ with $g\in X$ and $h\in Y$. The subgroup $[G,G]$ is called the \emph{derived subgroup} (or the commutator subgroup) of $G$, and denoted by $G'$ or $G^{(1)}$. We may also define the $n$-th derived subgroup $G^{(n)}=(G^{(n-1)})'$ of $G$ inductively. The following well-known lemma explains the importance of the derived subgroup.

\begin{lemma} Let $G$ be a group. Then $G'$ is a normal subgroup; and for any normal subgroup $N\trianglelefteq G$, the quotient group $G/N$ is abelian if and only if $G'\le N$.
\end{lemma}

Recall that the \emph{lower central series} of $G$ is the descending series of subgroups
\[  G=G_1\trianglerighteq G_2\trianglerighteq \cdots G_n\trianglerighteq\cdots,
\]
where each $G_{n+1}=[G_n,G]$. By definition, $G$ is nilpotent if and only if the lower central series terminates, i.e., $G_n=1$ for some $n$. The following notion of nilpotency class measures the nilpotency of a group, see for instance Definition 3.10 in \cite{khu}.

\begin{definition}
If a group $G$ satisfies $G_{c+1}=1$, then we say that $G$ is \emph{nilpotent} of class $\le c$; the least such number $c$ is called the \emph{nilpotency class} of $G$.
\end{definition}

\begin{remark}
(1) There are other equivalent definitions for nilpotent groups by using upper cental series or any other central series.

(2) If a group $G$ has nilpotency class $\le c$, then it is sometimes called a \emph{nil-$c$ group}.
\end{remark}

We collect some well-known facts on groups of lower nilpotency class without proof.

\begin{lemma}\label{lem-facts-nilclass} Let $G$ be a finite nilpotent group. Then
\begin{enumerate}
\item[(0)]  $G$ has nilpotency class $0$ if and only if $G$ is the trivial group.
\item[(1)]  $G$ has nilpotency class $\le 1$ if and only if $G$ is an abelian group.
\item[(2)]  $G$ has nilpotency class $\le 2$ if and only if the derived subgroup $G'$ is contained in the center of $G$.
\item[(3)]  If $G$ has nilpotency class $3$, then $G'$ is commutative.
\end{enumerate}
\end{lemma}

\subsection{Exponent of a group} The \emph{exponent} of a group $G$ is defined to be the least common multiple of the orders of all elements. If there is no least common multiple, the exponent is defined to be infinity. For finite groups, the exponent is a devisor of the order of the group and hence finite. Clearly, abelian groups of prime exponent $p$ are exactly the elementary $p$-groups.

We recall the following well-known result, which can be directly deduced from Theorem 6.5 and Theorem 6.6 in \cite[III]{HU}.

\begin{lemma} Any finite group of exponent $3$ is nilpotent of nilpotency class at most $3$.
\end{lemma}

We draw the following consequence.

\begin{corollary}\label{cor-exp3-dersubgp}
Let $G$ be a group of exponent 3. Then $G'$ is an elementary abelian 3-group.
\end{corollary}

\begin{proof}
Since $G$ is a group of exponent 3, the above lemma implies that $G$ is nilpotent of class at most 3. Then $[G',G']\le G_4=1$ by \cite[III, Satz 2.8]{HU}, which means that $G'$ is an abelian group, and hence an elementary abelian 3-group for $G$ has exponent 3.
\end{proof}

\subsection{Modules of a group and $p$-subgroups of $\GL_n(\Z_p)$} A \emph{module} of a given group $G$ is by definition an abelian group $M$ together with a map $G\times M\to M$, $(g,m)\mapsto g\cdot m$, such that $1_G\cdot m=m$ and $(gh)\cdot m=g\cdot(h\cdot m)$ for any $g,h\in G$ and $m\in M$. We also say that $G$ acts linearly on the abelian group $M$. We simply write $g\cdot m$ as $gm$ when there is no risk of confusion.
Clearly a $G$-module structure on an abelian group $M$ is equivalent to a group homomorphism $\rho\colon G\to \Aut(M)$ with $\rho$ given by $\rho(g)(m) = g m$ for any $g\in G$ and $m\in M$. In case we need to specify $\rho$, we also say that $A$ is a $G$-module via the homomorphism $\rho$, or simply say that $(A,\rho)$ is a $G$-module.

We remark that any group homomorphism $\rho\colon G\to \Aut(M)$ extends to a ring homomorphism
from $\Z [G]$ to $\mathrm{End} (M)$, where $\Z [G]$ is the group ring of $G$ over $\Z$ and $\mathrm{End} (M)$ is the ring of endomorphisms from $M$ to itself. Thus a $G$-module is equivalent to a $\Z [G]$-module.

Any finite abelian group can be written as a direct sum of cyclic subgroups. For any positive integer $m$, we use $\Z_m=\Z/m\Z$ to denote the cyclic group of order $m$. In fact, $\Z_m$ inherits a ring structure from $\Z$, in which the underlying additive group is exactly a cyclic group of order $m$.

Now let $p$ be a prime number. Then $\Z_p$ is the unique prime field of characteristic $p$, and any elementary $p$-group is equivalent to a vector space over $\Z_p$. In particular, the cyclic group $\Z_p$ is viewed as a one dimensional vector space over $\Z_p$.

Let $A\cong \Z_p^n$ be an $n$-dimensional $\Z_p$-vector space. Clearly $\Aut(A)=\GL(A)$, the general linear group of $A$. Any basis of $A$ gives an isomorphism $\GL(A)\cong \GL_n(\Z_p)$, mapping a linear transformation to its matrix under this basis, where $\GL_n(\Z_p)$ denotes the general linear group of degree $n$ over $\Z_p$.

By calculating the orders, it is easy to show that $UT(n,\Z_p)$, the subgroup consisting of $n\times n$ unipotent upper triangular matrices, is a Sylow $p$-subgroup of $\GL_n(\Z_p)$. Unipotent means that all entries in the main diagonal are 1. The following useful result is easy.

\begin{lemma} Let $G$ be a $p$-group, $A\cong \Z^n_p$, and $\rho\colon G\to \Aut A$ a group homomorphism. Then there exists a basis of $A$, under which $G$ maps into $UT(n,\Z_p)$. In particular,
if $A\cong \Z_p$, then $G$ acts trivially on $A$.
\end{lemma}

For a proof we use the fact that the matrices of a linear transformation under two bases are conjugate, and the fact that any two Sylow-$p$ subgroups of a finite group are conjugate to each other. We need also the following result for later use.

\begin{proposition}\label{prop-rad0}
Let $G$ be an abelian subgroup of $\GL_n(\Z_p)$. Then the  following are equivalent:
\begin{enumerate}
\item $1+ g + \cdots + g^{p-1} =0$ for any $g\in G$;
\item $(g-1)^{p-1}=0$ for any $g\in G$;
\item $(g_1-1)\cdots (g_{p-1}-1) =0$ for any $g_i\in G$;
\item $(g_1-1)\cdots (g_{p-1}-1) =0$ for any $g_i\in X$, where $X$ is a set of generators of $G$.
\end{enumerate}
\end{proposition}

\begin{proof}
First by the binomial expansion we have
\[(1-g)^{p-1}=\sum_{i=0}^{p-1} (-1)^i \frac{(p-1)(p-2)\cdots (p-i)}{1\cdot 2\cdots i} g^i= \sum_{i=0}^{p-1} g^i,\]
where we use the fact $p-i=-i$ in $\Z_p$. Then the equivalence $(1)\Longleftrightarrow(2)$ is obvious.

$(2)\Longrightarrow(3)$ follows from Lemma \ref{lemma-strongnil} below and $(3)\Longrightarrow (2), (4)$ is obvious.

We are left to show that $(4)\Longrightarrow (3)$. Assume (4) holds, and let $g_1,\cdots, g_{p-1}$ be in $G$. Since $X$
generates $G$, each $g_i= g_{i1}g_{i2}\cdots g_{i, t_i}$ for some $g_{i1},g_{i2},\cdots ,g_{i, t_i}\in X$.
Then \[g_i-1= g_{i1}g_{i2}\cdots g_{i, t_i} -1= \sum_{k=1}^{t_i} g_{i1}g_{i2}\cdots g_{i,k-1}(g_{ik}-1),\]
and it follows from (4) that $(1-g_1)\cdots (1-g_{p-1}) =0$.
\end{proof}

\begin{remark} (1) A group satisfies the equivalent conditions in the proposition has exponent $p$, and hence
isomorphic to a finite dimensional $\Z_p$-vector space.

(2) In the proposition the assumption that $G$ is commutative cannot be dropped. For instance, for any $a, b ,c\in \Z_3$, we set
\[g_{a,b,c}=\begin{pmatrix}
             1 & a & b & c \\
               & 1 &   & -b\\
               &   & 1 & a \\
               &   &   & 1 \\
           \end{pmatrix}\in \GL_4(\Z_3).\]
Consider the subgroup $G=\{g_{a,b,c}\mid a,b,c\in \Z_3\}\le \GL_4(\Z_3)$.
Then for any $g\in G$, $(g-1)^2=0$, while $(g_{1,0,0}-1)(g_{0,1,0}-1)\ne 0$.
\end{remark}

\begin{lemma}\label{lemma-strongnil} Let $G\le  \GL_n(\Z_p)$ be an abelian subgroup, $r\le p-1$ a positive integer, and $M$ an $n\times n$ matrix
over $\Z_p$. Assume that $(g-1)^rM=0$ for all $g\in G$. Then $$(g-1)^{r-1}(h-1)M=0$$ for all $g, h\in G$. Consequently,
$(g_1-1)(g_2-1)\cdots (g_r-1)M=0$ for all $g_1, g_2\cdots, g_r\in G$.
\end{lemma}

\begin{proof} Given $g, h\in G$, we write $x=g-1$, and $y= h-1$. Then $g=1+x, h=1+y$. We set $w_k=1 + h + \cdots+ h^{k-1}$ for $1\le \k\le r$. Clearly $w_k$ is invertible, and $w_k-w_l$ is invertible for any $1\le k\neq l\le r$.

By assumption, $(g-h^k)^{r}M=g^r(1-g^{-1}h^k)^rM=0$ for any $k$. It is direct to check that $g-h^k=x - yw_k$ and hence we get
\[
\begin{pmatrix} w_1& w_1^2 & \cdots& w_1^{r-1}\\
w_2& w_2^2 & \cdots &w_2^{r-1}\\
\cdots& \cdots& \cdots& \cdots\\
w_{r-1}&w_{r-1}^2 & \cdots & w_{r-1}^{r-1}
\end{pmatrix}
\begin{pmatrix} -C_r^{r-1}x^{r-1}yM\\C_r^{r-2}x^{r-2}y^2M\\ \cdots \\(-1)^{r-1} C_r^1xy^{r-1}M
\end{pmatrix}
= 0,
\]
where $C_r^i = \frac {r!}{i!(r-i)!}$ is the binomial coefficient. Using the Van der Monde determinant, we know that each
$C_r^i x^i y ^{r-i}M=0$, hence $x^i y ^{r-i}M=0$. In particular, we have the desired $(g-1)^{r-1}(h-1) M=0$. The last conclusion follows easily by using induction on $r$.
\end{proof}

\subsection{Radical and socle} We collect some well-known facts on finite length modules, cf. \cite{ars}. Let $R$ be a finite dimensional algebra over an arbitrary field $\k$, and $M$ an $R$-module of finite length.

Recall that the \emph{socle} of $M$, denoted by $\soc(_RM)$ or simply $\soc(M)$ when there is no confusion, is by definition the sum of all simple submodules of $M$; and the radical $\rad(_RM)$ (or $\rad(M)$) is the intersection of all its maximal submodules. Then $\rad(_RR)=\rad(R_R)$ is equal to the intersection of all two-sided maximal ideal of $R$, and we denote it by $\rad(R)$. Moreover, $\rad(_RM)=\rad(R)M$ for any left $R$-module $M$ and $\rad(N_R)=N\rad(R)$ for any right $R$-module $N$.

The filtration $0= \rad^{l}(M)\subseteq \rad^{l-1}(M)\subseteq\cdots\subseteq \rad(M)\subseteq M$ is called the radical filtration of an $R$-moduled $M$, where $\rad^{i}(M)=\rad(\rad^{i-1}M)=(\rad(R))^iM$ for each $i$, and the least
$l$ with $\rad^{l}(M)=0$ is called the Loewy length of $M$. We use $\ell\ell(M)$ to denote the Loewy length of $M$.

The coradical filtration is the filtration $0=\soc^0(M)\subseteq \soc^1(M)\subseteq\cdots \subseteq \soc^l(M)= M$ of submodules such that $\soc^1(M)=\soc(M)$, and $\soc^i(M)/\soc^{i-1}(M)= \soc(M/\soc^{i-1}(M))$ for each $i$. The least $l$ with $\soc^l(M)= M$ is called the length of the coradical filtration. It is well known that the radical filtration and the coradical filtration have equal length, cf. \cite[Chapter II, Proposition 4.7]{ars}.
\begin{proposition}\label{prop-loewylength} Let $G$ be a finite abelian group, $\Z_p[G]$ be the group algebra over $\Z_p$, and $(M,\rho)$ be a $\Z_p[G]$-module. Let $1\le r\le p-1$ be a given positive integer.
\begin{enumerate}
\item[(1)] For any $g\in G$, if $(\rho(g)-1)^p= 0$, then $\rho(g)^{p}=1$;
\item[(2)] If $(\rho(g)-1)^r= 0$ for all $g\in G$, then $\ell\ell(M)\le r$.
\end{enumerate}
\end{proposition}

\begin{proof} (1) follows from the fact that the Frobenius map respects the addition, say $\rho(g)^p-1 = (\rho(g)-1)^p =0$.

(2) Let $T_g=\rho(g)-1$. Extend $\rho$ to a ring homomorphism from $\Z_p[G]$ to $\End(M)$, and set $R=\rho(\Z_p[G])$. Then $\rad(R)=\sum_{g\in G}RT_g$, and $\rad^{k}(R)=\sum_{g_1, g_2,\cdots, g_k\in G} RT_{g_1}T_{g_2}\cdots T_{g_k}$ for any $k\ge 1$. Therefore by Lemma \ref{lemma-strongnil} we have $\rad^{r}(M)=\rad^r(R)M=0$ and hence $\ell\ell(M)\le r$.
\end{proof}

\begin{remark} (1) We mention that if we drop the assumption $r\le p-1$ on $r$ in Proposition \ref{prop-loewylength} (2), then the conclusion does not hold in general. For instance, we may take $r= p$ and some $m> 1$. Consider the action of the group $\Z_p^m$ on the group ring $\Z_p[\Z_p^m]$ given by multiplication. Then $(\rho(g)-1)^p=0$ for any $g\in \Z_p$, while $\ell\ell(\Z_p[\Z_p^m])=(p-1)^m+1 > p$.

(2) Let $f\colon R\to R'$ be a homomorphism of rings, and $M$ be an $R'$-module of finite length. Then $M$ can be viewed as an $R$-module. Clearly $_RM$ is simple if and only if $_{R'}M$ is simple. It follows that $\rad(_RM)=\rad(_{R'}M)$ and $\soc(_RM)=\soc(_{R'}M)$ as subgroups, and hence $\ell\ell(_RM)=\ell\ell(_{R'}M)$.

(3) Let $M$ be a nonzero $R$-module with finite length and $X$ a semisimple module. Then $\ell\ell(M)=\ell\ell(M+X)$.
\end{remark}

\section{Group cohomology and group extension}

In this section, we recall some basics on group cohomology and group extension. One can consult Chapter 6 in \cite{wei} and Chapter 7 in \cite{rot}. We reformulate the group extension and introduce the notion of group extension datum, which plays a crucial role in this work.

\subsection{The bar resolution}
Let $G$ be a group. Consider the following bar resolution $B_*$ of the trivial $G$-module $\Z$.
\[ \cdots\to B_3\xrightarrow{d} B_2 \xrightarrow{d} B_1 \xrightarrow{d} B_0 \xrightarrow{\varepsilon} \Z.
\]
For each $n\ge1$, $B_n$ is the free $\Z [G]$-module on the set of all symbols $[g_1,\cdots ,g_n]$ with $g_i\in G$; and $B_0$ is the free $\Z [G]$-module of rank one and with basis $[\ ]$. The differentials are given by $\varepsilon([\ ])= 1$, $d([g])= g[\ ]-[\ ]$, and \[d([g_1,\cdots, g_n])= g_1[g_2,\cdots, g_n] +\sum_{1\le i\le n-1}(-1)^i[g_1,\cdots, g_ig_{i+1}, \cdots, g_n] +(-1)^n[g_1,\cdots, g_{n-1}].\]

Let $A$ be a $G$-module. An \emph{$n$-cochain} is a set map $\varphi$ from $G^n=G\times\cdots \times G$ to $A$; the set of $n$-cochains is identified with $\hom_G(B_n, A)$. A cochain $\varphi$ is normalized if $\varphi(g_1, \cdots)$ vanishes whenever some $g_i=1$. The differential $d\varphi$ of $\varphi$ is  an $n+1$-cochian given by
\[ d\varphi(g_0, \cdots, g_n)=g_0\varphi(g_1,\cdots, g_n)+ \sum(-1)^{i+1}\varphi(\cdots, g_ig_{i+1},\cdots) + (-1)^n\varphi(g_0,\cdots, g_{n-1}).
\]
By definition, \emph{$n$-cocycles} are those $n$-cochains such that $d\varphi=0$, and \emph{$n$-coboundaries} are the ones of the form $d\varphi$. We denote by $B^n(G; A)$ and $Z^n(G; A)$ the set of $n$-coboundaries and the set of $n$-cocycles of $G$ (with coefficients in $A$) respectively. Clearly, $B^n(G; A)\subseteq Z^n(G; A)$, and the quotient $H^n(G; A)=Z^n(G; A)/B^n(G; A)$ is called the $n$-th cohomology group of $G$ with coefficients in $A$. By definition $H^n(G; A)=\mathrm{Ext}_{\Z [G]}^n(\Z, A)$. We mention that any $n$-cocycle is cohomologous to a normalized one.

\begin{example} Let $A$ be a $G$-module, and $\varphi\colon G\times G\to A$ a map. Then
$\varphi$ is a 2-cocycle if and only if $g\varphi(h,k)-\varphi(gh,k)+\varphi(g,hk)-\varphi(g,h)=0$ for all $g,h,k\in A$; and $\varphi$ is a 2-coboundary if and only if there exists some $f\colon G\to A$ such that $\varphi(g,h)=g f(h)-f(gh)+ f(g)$ for all $g$, $h\in G$.
\end{example}

\subsection{Group extension datum} Let $K$ and $H$ be groups. An extension of $H$ by $K$ is a short exact sequence $1\to K\to G\to H\to 1$ of groups, and $K$ is usually called the kernel of the extension.

 The question to find all group extensions of given groups $H$ and $K$ is called the extension problem, which has been studied heavily since the late nineteenth century.
By Jordan-H\"older theorem, any finite group is obtained iteratively by group extension from simple groups. The solution to the extension problem would give us a complete classification of all finite groups. The extension problem is a very hard problem, and no general theory exists which treats all possible extensions at one time. However, group extensions with abelian kernel can be studied by using the group cohomology. We mainly deal with the case that both $K$ and $H$ are elementary $p$-groups here.

Now let $A$ be an abelian group. As usual, we use + to denote the multiplication and 0 the identity element in $A$. Let $0\to A \xrightarrow{\iota} G\xrightarrow{\pi} H\to 1$ be a group extension. We may identify $A$ with a normal subgroup of $G$. Then $G$ acts on $A$ by conjugate in $G$. To avoid confusion, we write $^g\!a$ for the conjugate $gag^{-1}$ in $G$. This induces an $H$-module structure on $A$, and the induced $H$-action  is written as $h\triangleright a$ for $h\in H$ and $a\in A$.

\begin{remark} By definition, the conjugate action is trivial if and only if $A$ is contained in the center of $G$. A group extension of this type is called a \emph{central extension}.
\end{remark}

Let $s\colon H\to G$ be a (set-theoretic) section, i.e., a set map such that $\pi\circ s= \id_H$. Then we obtain a well-defined map $[\ ]_s\colon H\times H\to A$ by setting $[h_1,h_2]_s=s(h_1)s(h_2)s(h_1h_2)^{-1}$. We call $[\ ]_s$ the \emph{factor set} determined by the group extension and $s$. It is direct to check that $[\ ]_s\in Z^2(H; A)$ is a 2-cocycle, and $[\ ]_s$ is normalized if $s(1_H)=1_G$. Moreover, $[\ ]_s-[\ ]_{s'}$ is a 2-coboundary for any other section $s'$.

We say that two extensions of $H$ by $A$ are equivalent if we have a commutative diagram
\[\begin{CD}
0    @>>>  A@>>>G@>>>H@>>>1\\
@. @V\id_AVV        @VV\sigma V @VV\id_HV @.\\
0    @>>>  A@>>>\tilde{G}@>>>H@>>>1
\end{CD}
\]
of group homomorphisms. Clearly $\sigma$ is an isomorphism in this case. Then we have the following well-known classification result.

\begin{proposition} Let $H$ be a group and $A$ an $H$-module. Then the equivalence classes of
group extensions such that the induced $H$-action on $A$ agrees with the $H$-module structure
are in 1-1 correspondence with the cohomology group $H^2(G, A)$.
\end{proposition}

We aim to apply group extensions to the classification of $p$-groups, and for this we need a generalized
version of the above proposition. We use the following notion of extension datum, which is essentially a combination of the notions of ``data'' and ``factor set'' as introduced in Page 179-180 in \cite{rot}.

\begin{definition} A \emph{group extension datum} (\emph{datum} for short) is by definition a quadruple $\mathcal D= (H, A, \rho, \varphi)$, where $H$ is a group, $A$ is an abelian group with an $H$-action given by $\rho\colon H\to \Aut(A)$, and $\varphi\colon H\times H\to A$ is a normalized 2-cocycle. $A$ is called the \emph{kernel} of the datum and $H$ is called the \emph{cokernel} of the datum.

If $H$ is an abelian group, then we call $\mathcal D$ an \emph{abelian datum}; If $H\cong \Z_p^m$ and $A\cong \Z_p^n$ are both elementary $p$-groups, then we call $\mathcal D$ a \emph{$p$-elementary datum of type $(m,n)$}.
\end{definition}

Let $0\to A\to G\to H\to 1$ be a group extension with abelian kernel. As we have shown above, any set-theoretic section $s\colon H\to G$ with $s(1_H)=1_G$ gives to a datum $(H, A, \rho, \varphi)$. We say that the group extension $0\to A\to G\to H\to 1$ is a realization of $(H, A, \rho, \varphi)$. Clearly $\rho$ is uniquely determined by the extension, while $\varphi$ depends on the choice of $s$.

\begin{lemma-definition} Given a datum $\mathcal D= (H, A, \rho, \varphi)$. We can define a product on the set $A\times H$ by $(a,h)(b,k)= (a+ h\triangleright b +\varphi(h,k), hk)$. Then
\begin{enumerate}
\item The product makes $A\times H$ a group with the identity element $(0, 1)$, and the inverse given by $(a, h)^{-1}=(-h^{-1}\triangleright(a+\varphi(h,h^{-1})), h^{-1})$. The resulting group is called the \emph{realizing group} of the datum $\D$ ( or the group realizing $\D$), and is denoted by $G({\mathcal D})$ or $A\rtimes_{\rho,\varphi} H$.
\item $A\rtimes_{\rho, 0} H$ gives the semidirect product of $A$ and $H$.
\item $(A,1)\subseteq Z(G(\mathcal D))$, the center of $G(\mathcal D)$, if and only if $H$ acts on $A$ trivially.
\end{enumerate}
\end{lemma-definition}

\begin{proof} (1) and (2) are trivial and we omit the proof here. We only prove (3). First assume that $A$ is a trivial module, say $h\tr a=a$ for any $a\in A$ and $h\in H$. Then
\[(a,1)(b,h)= (a+b+\varphi(1,h), h)=(a+b, h)=(h\tr a+b+ \varphi(h,1), h)= (b,h)(a,1)\]
holds for any $(b,h)\in G(\mathcal D)$, hence $(a,1)\in Z(G(\mathcal D))$. Note that $\varphi$ is assumed to be normalized here.

Conversely, for any $a\in A$, if $(a,1)\in Z(G(\mathcal D))$, then we have $$(a+b+\varphi(1,h), h)=(h\tr a+b+ \varphi(h,1), h)$$ for any $b\in A, h\in H$. In particular, $h\tr a = a$ for any $x\in H$, which means that the $H$-action is trivial.
\end{proof}

\begin{remark} \label{rem-data}(1)  We simply write $\rtimes_{\varphi}=\rtimes_{1,\varphi}$ and $\rtimes_{\rho}=\rtimes_{\rho,0}$ when there is no confusion. Sometimes we also use $a\rtimes h$ to denote the element $(a,h)\in A\rtimes_{\rho,\varphi} H$.

(2) It is not hard to show that $A\rtimes_{\rho,\varphi} H \cong A\rtimes_{\rho,\varphi'} H$ if $\varphi$ and $\varphi'$ are cohomologous. We freely identify $A$ as the subgroup $(A, 1)$ of $A\rtimes_{\rho,\varphi} H$
and $H$ a subset via the maps $a\mapsto (a, 1_H)$ and $h\mapsto (0,h)$ for any $a\in A$ and $h\in H$.
We mention that $(0, H)$ is not a subgroup in general, it is if and only if $\varphi=0$.

(3) The above construction works well if $\varphi$ is not normalized. In this case, $(A,1)$ will not be a subgroup. In fact, $\varphi$ is normalized if and only if $1_{A\rtimes_{\rho,\varphi} H} =(0,1_H)$.

(4) Extensions with non-abelian kernel relate also closely to certain cohomology theory, cf. \cite[Section IV.6]{br}.
\end{remark}

By definition, any solvable group is realized as an iterated extension of abelian groups, and hence of cyclic groups. This fact seems to motivate (at least partly) the early development of group cohomology theory.

\begin{proposition} Let $G$ be a finite group. Then

(1) $G$ is solvable if and only if
\[G\cong A_n\rtimes_{\rho_n,\varphi_n}(A_{n-1}\rtimes_{\rho_{n-1},\varphi_{n-1}}(\cdots (A_2\rtimes_{\rho_1,\varphi_1}(A_1\rtimes_{\rho_1,\varphi_1}A_0))\cdots))\]
for some cyclic groups $A_0, A_1,\cdots, A_n$, group actions $\rho_1, \cdots, \rho_n$ and 2-cocycles $\varphi_1, \cdots, \varphi_n$.

(2) $G$ is nilpotent if and only if all $\rho_i$'s above can be chosen trivial, say
\[G\cong A_n\rtimes_{\varphi_n}(A_{n-1}\rtimes_{\varphi_{n-1}}(\cdots (A_2\rtimes_{\varphi_1}(A_1\rtimes_{\varphi_1}A_0))\cdots))\]
for some cyclic groups $A_0, A_1,\cdots, A_n$  and 2-cocycles $\varphi_1, \cdots, \varphi_n$.
\end{proposition}

We end this subsection with the construction of product of group extension data. Let $\D_1= (H_1, A_1, \rho_1, \varphi_1)$ and $\D_2= (H_2, A_2, \rho_2, \varphi_2)$ be data. We set $H= H_1\times H_2$ and $p_1\colon H\to H_1$ and $p_2\colon H\to H_2$ to be the projection maps. Then $A_1$ and $A_2$ are $H$-modules with the obvious action given by $\rho_1\circ p_1$ and $\rho_2\circ p_2$
respectively. Set $A= A_1\oplus A_2$. It is direct to show that the map $\varphi\colon H\times H\to A, \varphi((h_1, h_2), (h_1', h_2')) = (\varphi_1(h_1, h_1'), \varphi_2(h_2, h_2'))$ gives a 2-cocycle of $H$. Then we have a datum $(H, A, \rho, \varphi)$, which is called the product datum of $\D_1$ and $\D_2$ and denoted by $\D_1\times \D_2$.

\begin{proposition}\label{prop-prod-data}  Keep the above notations. Then

(1) $G(\D_1\times \D_2)\cong G(\D_1)\times G(\D_2)$.

(2) Given a datum $\mathcal D= (H, A, \rho, \varphi)$. Let $A_1$ and $A_2$ be $H$-submodules of $A$, such that $A= A_1\oplus A_2$, $H$ acts trivially on $A_2$, and $\mathrm{Im}(\varphi)\subseteq A_1$. Then $G(\D)\cong G(\D_1)\times A_2$, where $\D_1= (H, A_1, \rho, \varphi)$.
\end{proposition}

The proof is easy. In fact, the map $(a_1, a_2)\rtimes(h_1,h_2))\mapsto (a_1\rtimes h_1, a_2\rtimes h_2)$ is a desired isomorphism in $(1)$; and (2) is the special case of $(1)$ with $H_1=H$ and $H_2=\{1\}$.

\subsection{Equivalence of data}
 Let $\sigma_A\colon A\to \tilde A$ and $\sigma_H\colon H\to \tilde H$ be isomorphisms of groups. Clearly $\sigma_A$ induces an isomorphism $\Aut(A)\to \Aut(\tilde A)$, mapping $\tau\in \Aut(A)$ to the composition map $\sigma_A\circ\tau\circ\sigma_A^{-1}$. The composition with $\rho\circ \sigma_H^{-1}$ gives a group homomorphism $\tilde \rho\colon \tilde H\to \Aut(\tilde A)$, making $\tilde A$ an $\tilde H$-module. More precisely, the $\tilde H$-action is given by $\tilde h\tr\tilde a=\sigma_A(\sigma_H^{-1}(\tilde h)\tr \sigma_A^{-1}(a))$ for any $\tilde h\in \tilde H$ and $\tilde a\in \tilde A$. Consider the map $\tilde\varphi\colon \tilde H\times \tilde H\to \tilde A$, $\tilde \varphi(\tilde h, \tilde k)= \sigma_A(\varphi(\sigma_H^{-1}(\tilde{h}), \sigma_H^{-1}(\tilde{k})))$. It is direct to check that $\tilde\varphi\in Z^2(\tilde H, \tilde A)$, and $A\rtimes_{\rho,\varphi} H\cong \tilde A\rtimes_{\tilde \rho,\tilde \varphi} \tilde H$. This suggests the following definition.

\begin{definition}\label{def-equi-data}
Two data $\mathcal{D}=(H, A, \rho, \varphi)$ and $\tilde{\mathcal{D}}=(\tilde H, \tilde A, \tilde\rho, \tilde\varphi)$ are said to be \emph{equivalent} if
there exist group isomorphisms $\sigma_H\colon H\to \tilde H$ and $\sigma_A\colon A\to \tilde A$ and a map $f\colon H\to\tilde A$ , such that
$\sigma_A(h\tr a)=\sigma_H(h)\tr \sigma_A(a)$  and \[\tilde \varphi(\sigma_H(h), \sigma_H(k)) -\sigma_A(\varphi(h, k))=\sigma_H(h)\tr f(k)-f(hk)+f(h)\] for all $a\in A$ and $h, k\in H$; if furthermore, $f$ can be chosen to be 0, then $\mathcal D$ and $\tilde {\mathcal D}$ are said to be \emph{isomorphic}.
\end{definition}

\begin{remark}\label{rem-equi-data} We denoted two equivalent data by $\D\sim\tilde D$, or $\D \stackrel{(\sigma_H, \sigma_A, f)}{\sim} \tilde D$ in case one needs to specify $\sigma_H, \sigma_A$ and $f$.  It is easy to show that $``\sim "$ is an equivalence relation, and hence provides a partition of the set of data into equivalence classes. As usual, we use $[\D]$ to denote the equivalence class containing $\D$, that is $[\D]=\{\tilde\D \mid D\sim \tilde D\}$.
\end{remark}

We need the following result. The proof is straightforward and we omit it here.

\begin{proposition}\label{prop-equidata-isogp} Let $\mathcal{D}=(H, A, \rho, \varphi)$ and $\tilde{\mathcal{D}}=(\tilde H, \tilde A, \tilde\rho, \tilde\varphi)$ be data. Then $\mathcal D$ and $\tilde {\mathcal D}$ are equivalent if and only if  there exists a group isomorphism $\sigma\colon G(\mathcal D) \to G(\tilde{\mathcal D})$ that restricts to an isomorphism $\sigma_A\colon A\to\tilde A$.
\end{proposition}

\begin{remark} We mention that not all isomorphisms $\sigma\colon G(\mathcal D)\to G(\tilde{\mathcal D})$ will restrict to an isomorphism between $A$ and $\tilde A$. It will be the case if  $A=G(\mathcal D)'$ and $\tilde A= G(\tilde{\mathcal D})'$, or if
$A= Z(G(\mathcal D))$ and $\tilde A= Z(G(\tilde{\mathcal D}))$.

An obvious counterexample is as follows. Let $A_1$ and $A_2$ be nonisomorphic abelian groups. Consider the trivial datum $\mathcal D= (A_1, A_2, \rho, \varphi)$ and $\tilde{\mathcal D}=(A_2, A_1, \tilde\rho, \tilde\varphi)$, where trivial means that the group actions and the 2-cocyles are all trivial. Then the natural isomorphism $A_2\times A_1\cong A_1\times A_2$ gives an isomorphism $G(\mathcal D)\cong G(\tilde{\mathcal D})$, which does not restrict to an isomorphism from $A_2$ to $A_1$.
\end{remark}

We may draw an easy consequence, which says that in many cases we are interested, isomorphisms of realizing groups is equivalent to equivalences of data.

\begin{corollary}\label{cor-equidata-isogp} Let $\mathcal{D}=(H, A, \rho, \varphi)$ and $\tilde{\mathcal{D}}=(\tilde H, \tilde A, \tilde\rho, \tilde\varphi)$ be data. Assume that one of the following holds:
\begin{enumerate} \item $A=G(\mathcal D)'$ and $\tilde A= G(\tilde{\mathcal D})'$;
\item $A= Z(G(\mathcal D))$ and $\tilde A= Z(G(\tilde{\mathcal D}))$.
\end{enumerate}
Then $G(\mathcal D)$ and $G(\tilde{\mathcal D})$ are isomorphic if and only if $\mathcal D$ and $\tilde {\mathcal D}$ are equivalent.
\end{corollary}

\subsection{Abelian data} We are interested in groups that are extensions of abelian groups, or equivalently, realizing groups of abelian data. For instance, by Lemma \ref{lem-facts-nilclass} and Corollary \ref{cor-exp3-dersubgp} any finite group of exponent 3 or any nilpotent group of class 3 is such a group.

\begin{lemma}\label{lem-abel-extension} Let $\D=(H, A, \rho, \varphi)$ be an abelian datum. Then
\begin{enumerate}
\item For any $(a,k), (b,h)\in G(\mathcal D)$, we have \[[(a,k),(b,h)]=((1-h)\tr a+(k-1)\tr b+\varphi(k,h)-\varphi(h,k),1);\]
\item The derived subgroup of $G(\mathcal D)$ is abelian;
\item The center of $G(\D)$ is given by \[\{(a,k)\in G(\D)\mid k\tr b= b\ \forall b\in A,\ \varphi(k,h)-\varphi(h,k)= (h-1)\tr a\ \forall h\in G\}.\]
    Moreover, there exist some $\tilde\varphi$ cohomologous to $\varphi$ and a subgroup $K\le H$, such that

    (a) $K$ acts on $A$ trivially;

    (b) $\tilde\varphi(k,h)\in \soc(A)$ and $\tilde\varphi(k, h)=\tilde\varphi(h, k)$ for any $k\in K$ and $h\in H$;

    (c) $Z(A\rtimes_{\rho,\tilde\varphi}H)=\{(a,k)\mid a\in\soc(A), k\in K\}$.
\end{enumerate}
\end{lemma}

\begin{proof}
(1) By definition $(a,k)(b,h)= [(a,k),(b,h)](b,h)(a,k)$. Write $[(a,k),(b,h)]=(c,g)$. Then we have
\begin{align*}(a+ k\tr b+\varphi(k,h), kh) = &(c,g)(b+ h\tr a+\varphi(h,k), hk)
\\ = &(c + g\tr(b+ h\tr a+\varphi(h,k)) + \varphi(g, hk), ghk).
\end{align*}
It follows that $g=1$, and $a+ k\tr b+\varphi(k,h)=c +  b+ h\tr a+\varphi(h,k)$, and we obtain the desired equality. Note that we use the assumptions $hk=kh$ and $\varphi$ is normalized here.

(2) By (1) $G(\mathcal D)'$ is a subgroup of $A$ and hence abelian.

(3) Note that $(a, k)\in Z(G)$ if and only if $[(a,k), (0, h)]=(0,1)= [(a,k), (b, 1)]$ for any $b\in A$ and $k\in G$. The first conclusion follows easily from in (1).

It is direct to show that $(a, 0)\in Z(G(\D))$ if and only if $(h-1)\tr a= 0$ for any $h\in H$, which is equivalent to $a\in\soc(A)$. For $a, a'\in A$, if $(a,k), (a',k)\in Z(G(\D))$ for some $k\in K$, then $a-a'\in \soc(A)$.

Set $K=\{k\in H\mid (a, k)\in Z(G(\D)) \text{ for some } a\in A\}$. Then $K\le H$ is a subgroup. In fact,
consider the projection from $G(\D)$ to $H$, then $K$ is the image of $Z(G(\D))$. Clearly (a) is true by the first assertion.

For any $k\in K$, we fix some $a(k)\in A$ such that $(a(k),k)\in Z(G(\D))$ and we choose $a(1_H)=0$.
Then $(a(k), k)(a(k'), k')=(a(k)+ a(k') +\varphi(k, k'), kk')\in Z(G(\D))$ and hence $a(kk')- a(k)- a(k') - \varphi(k, k')\in \soc(A)$.

Let $S\subseteq H$ be a complete set of representatives of left cosets of $K$ in $H$.
Then each $h\in H$ is uniquely written as $ks$ for some $k\in K$ and $s\in S$.
We then have a map $f\colon H\to K$ given by $f(ks)= a(k) +\varphi(k,s)$.
We claim that $\varphi+d(f)$ gives the desired $\tilde\varphi$.

 In fact, for any $k, k'\in K$ and $s\in S$,
\begin{align*}
\tilde\varphi(k, k's)= &\varphi(k,k's)+ k\tr(a(k')+ \varphi(k', s))- (a(kk')+ \varphi(kk', s))+ a(k)\\
=& \varphi(k', s)- \varphi(kk', s)+ \varphi(k,k's)-\varphi(k, k') - a(kk') + a(k)+ a(k')+\varphi(k, k')\\
=& a(k)+ a(k')+\varphi(k, k') - a(kk')\in \soc(A),
\end{align*}
and
\begin{align*}
\tilde\varphi(k's, k)= \varphi(k's, k)+ k's\tr a(k) - (a(kk')+ \varphi(kk', s))+ (a(k')+ \varphi(k', s)),
\end{align*}
therefore we have
\[\tilde\varphi(k's, k)-\tilde\varphi(k, k's)= \varphi(k's, k)- \varphi(k,k's) + (k's-1)\tr a(k)=0,\]
which proves (b).

Now (c) follows from (b) and the first assertion of (3).
\end{proof}

Let $\tilde\D=(\tilde H, \tilde A, \tilde \rho, \tilde\varphi)$ be another abelian data. Then combined with Proposition \ref{prop-prod-data}, we have the following result, generalizing Corollary \ref{cor-equidata-isogp}(1).

\begin{proposition}\label{prop-equiabeldata-isogp} Let $\D=(H, A, \rho, \varphi)$ and $\tilde\D=(\tilde H, \tilde A, \tilde \rho, \tilde\varphi)$ be abelian data with $A\cong \tilde A$ as abelian groups. Assume that $ A= G(\D)'\oplus A_1$ and $\tilde A= G(\tilde\D)'\oplus \tilde A_1$ for some trivial modules $_H\!A_1$ and $_{\tilde H}\!\tilde A_1$, and $\mathrm{Im}(\varphi)\subseteq G(\D)'$ and $\mathrm{Im}(\tilde\varphi)\subseteq G(\tilde\D)'$. Then $G(\D)\cong G(\tilde D)$ if and only if $\D\sim\tilde\D$.
\end{proposition}

\begin{proof} The sufficiency follows from Proposition \ref{prop-equidata-isogp}. We need only to prove the necessity. Assume $G(\D)\cong G(\tilde{\D})$, we will show that $\D\sim \tilde \D$.

We set $A_0=G(D)'$. Let $\rho_0\colon H\to \Aut(A_0)$ and $\varphi_0\colon H\times H\to A_0$ be the obvious maps induced by $\rho$ and $\varphi$. Then $\D_0=(H, A_0, \rho_0, \varphi_0)$ is a datum. It is easy to check that $G(\D)\cong G(\D_0)\times A_1$ and $G(\D_0)'=A_0$. Moreover, we have a decomposition $\D=\D_0\times \D_1$, where $\D_1=(\{1\}, A_1, 1, 0)$ is a trivial datum.

Similarly we have $\tilde\D_0=(\tilde H, \tilde A_0, \tilde \rho_0, \tilde\varphi_0)$, and $G(\tilde\D)\cong G(\tilde\D_0)\times \tilde A_1$ and $G(\tilde\D_0)'=\tilde A_0$. Now $G(\D)\cong G(\tilde{\D})$ implies that $A_0\cong \tilde A_0$, and hence $A_1\cong \tilde A_1$ as abelian groups. Applying Jordan-H\"{o}lder theorem, we obtain an group isomorphism $G(\D_0) \cong G(\tilde \D_0)$. By Corollary \ref{cor-equidata-isogp}, we know that $\D_0\sim\tilde \D_0$, and the conclusion $\D\sim \tilde\D$ follows easily.
\end{proof}

\subsection{Data with realizing group having exponent $p$}
Since we are interested in groups of exponent $p$, we discuss in this subsection when a realizing group has exponent $p$.
\begin{proposition}\label{prop-ext-p} Let $\mathcal D = (H, A, \rho, \varphi)$ be a datum with $A$ and $H$ nontrivial. Then $G(\mathcal D)$ has exponent $p$ if and only if the following conditions hold:
\begin{enumerate}
\item Both $H$ and $A$ have exponent $p$;
\item $1+ \rho(h) + \rho(h^2) +\cdots +\rho(h^{p-1})=0$ for all $h\in H$;
\item $\varphi(h,h)+\varphi(h^2,h)+\cdots +\varphi(h^{p-1},h) =0$ for all $h\in H$.
\end{enumerate}
\end{proposition}

\begin{proof} First we prove the necessity. We assume that $G(\mathcal D)$ has exponent $p$. Then (1) is obvious since any nontrivial subgroup and quotient group of $G(\mathcal D)$ will have exponent $p$. Now for any $(a,h)\in G(\mathcal D)$, it is direct to check that
\[(a, h)^p= (a+h\tr a + \cdots+ h^{p-1}\tr a + \varphi(h,h)+\varphi(h^2,h)+\cdots +\varphi(h^{p-1},h), h^p),\]
which equals $(0,1)$ by assumption.
By taking $a=0$ we will obtain (3),  and (2) follows by considering arbitrary $a$. The above calculation also proves the sufficiency.
\end{proof}

\section{A formula for 2-cocycles and application to abelian data}

\subsection{2-cocycles of a finite abelian group} Let $H$ be a finite abelian group. Then $H$ is a direct product of nontrivial cyclic subgroups $H=H_1\times H_2\times \cdots\times H_m$. Assume that for any $1\le r\le m$, $H_r=\langle h_r\rangle \cong \Z_{n_r}$ for some $h_r\in H_r$ and some positive integer $n_r$. Clearly $h_r$ has order $n_r$.
Given any $h\in H$, $h=h_1^{i_1}h_2^{i_2}\cdots h_m^{i_m}$ for some array $(i_1,i_2,\cdots, i_m)\in \Z^m$, which is uniquely determined if we require that $0\le i_r\le n_r-1$ for any $1\le r\le m$.
We mention that the decomposition $H=H_1\times H_2\times \cdots\times H_m$ is not unique in general.

Given a ring $R$ and an element $x\in R$, we set $(x)_n= 1 + x +\cdots +x^{n-1}\in R$ for any positive integer $n$. For consistency of notations, we set $(x)_0=0$. Then $1-x^n=(1-x)(x)_n$ for any $x\in R$ and $n\ge 0$.

Let $(A,\rho)$ be an $H$-module. Let $(\varphi_{rs})_{1\le r\le s\le m}$ be a collection of elements in $A$.
Then $(\varphi_{st})$ gives a normalized 2-cochain $\varphi\colon H\times H\to A$  by setting
\begin{align}\label{formula-cochain}
\notag &\varphi(h_1^{i_1}h_2^{i_2}\cdots h_m^{i_m}, h_1^{j_1}h_2^{j_2}\cdots h_m^{j_m})\\
=&\sum_{r=1}^m [\frac{i_r+j_r}{n_r}](h_1^{i_1+j_1}\cdots h_{r-1}^{i_{r-1}+j_{r-1}})\tr \varphi_{rr} \\
\notag
-&\sum_{1\le r<s\le m} (h_1^{i_1}\cdots h_{s-1}^{i_{s-1}}h_1^{j_1}\cdots h_{r-1}^{j_{r-1}}(h_r)_{j_r}(h_s)_{i_s})\tr\varphi_{rs}
\end{align}
for any $0\le i_r, j_r\le n_r-1$, $r=1, 2, \cdots, m$, where $[\frac{i_r+j_r}{n_r}]$ denotes the greatest integer which is less than or equal to $\frac{i_r+j_r}{n_r}$.

We write $T_{r}=h_{r}-1$ and $N_{r}=(h_r)_{n_r} =1+ h_r + h_r^2 +\cdots + h_{r}^{n_r-1}$. The following characterization is essentially given in \cite{hlyy, hwy}.

\begin{proposition}\label{prop-2cocycle} Let $H$, $A$ be as above. Then
\begin{enumerate}
\item The 2-cochain $\varphi$ is a $2$-cocycle if and only if
 \begin{enumerate}
 \item[(i)] $T_r\tr\varphi_{rr}=0$ for $1\le r\le m$;
 \item[(ii)] $N_r\tr\varphi_{rs}+T_s\tr \varphi_{rr}=0$, $T_r\tr \varphi_{ss}-N_s\tr \varphi_{rs}=0$ for $1\le r<s\le m$; and
 \item[(iii)] $T_r\tr \varphi_{st}-T_s\tr \varphi_{rt}+T_t\tr \varphi_{rs}=0$ for $1\le r<s<t\le m$.
 \end{enumerate}

\item $\varphi$ is a $2$-coboundary if and only if
there exist $a_1, a_2,\cdots, a_m\in A$, such that
  \begin{enumerate}
  \item[(i)] $\varphi_{rr}=N_r\tr a_r$ for $1\le r\le m$, and
  \item[(ii)] $\varphi_{rs}=T_r\tr a_s-T_s\tr a_r$ for $1\le r<s\le m$.
  \end{enumerate}

\item Let $\phi\colon H\times H\to A$ be a 2-cocycle. Then $\phi$ is cohomologous to the $2$-cochain corresponding to $(\phi_{rs})_{1\le r\le s\le m}$, where $\phi_{rr} = \phi(1,h_r) + \phi(h_r, h_r) + \cdots + \phi(h_r^{n_r-1}, h_r)$ for $1\le r\le m$, and $\phi_{rs}= \phi(h_r, h_s)-\phi(h_s, h_r)$ for $1\le r<s\le m$.

\item If moreover $H$ acts on $A$ trivially. Then $H^2(H; A)$ is in one-to-one correspondence with the set of collections $(\varphi_{rs})_{1\le r\le t\le m}$.
\end{enumerate}
\end{proposition}

Combined with Lemma \ref{lem-abel-extension}, we obtain a characterization of the derived subgroup of an extension of abelian groups.

\begin{proposition}\label{prop-der-abext} Let $H$, $A$,  $(\varphi_{rs})_{1\le r\le s\le m}$ be as above. Assume that the corresponding 2-cochain $\varphi$ is a 2-cocycle. Then the derived subgroup $(A\rtimes_{\rho,\varphi} H)'$ is exactly the $H$-submodule of $A$ generated by $\{\varphi_{rs}\mid 1\le r< s\le m\}$ and $\{T_r\tr a\mid 1\le r\le m, a\in A\}$.
\end{proposition}

\begin{proof} Set $G=A\rtimes_{\rho,\varphi} H$, and let $M$ be the $H$-submodule of $A$ which is generated by $\{\varphi_{rs}\mid 1\le r< s\le m\}$ and $\{T_r\tr a\mid 1\le r\le m, a\in A\}$.

By Lemma \ref{lem-abel-extension}, for any $a, b\in A$ and $h, k\in H$, we have \[[(a,h),(b,k)]=((1-k)\tr a+(h-1)\tr b+\varphi(h,k)-\varphi(k,h),1).\]
Then $(T_r\tr a,1)= [(0,h_r),(a,1)]\in G'$ and $(\varphi_{rs},1)=[(0,h_r),(0, h_s)]\in G'$ for $1\le r<s\le m$, which implies $M\subseteq G'$.

Conversely, for any $h=h_1^{i_1}\cdots h_m^{i_m}$, $k=h_1^{j_1}\cdots h_m^{j_m}$ and any $a, b\in A$,
\[(h-1)\tr b= \sum_{u=1}^m h_1^{i_1}\cdots h_{u-1}^{i_{u-1}}(h_u^{i_{u}}-1)\tr b= \sum_{u=1}^m h_1^{i_1}\cdots h_{u-1}^{i_{u-1}}(h_u)_{i_u}(h_u-1)\tr b\in M,\]
and similarly $(1-k)\tr a\in M$. Moreover, we have
\begin{align*}\varphi(h,k)&=\varphi(h_1^{i_1}h_2^{i_2}\cdots h_m^{i_m}, h_1^{j_1}h_2^{j_2}\cdots h_m^{j_m})\\
&= \sum_{r=1}^m [\frac{i_r+j_r}{n_r}](h_1^{i_1+j_1}\cdots g_{r-1}^{i_{r-1}+j_{r-1}})\tr \varphi_{rr} \\
&
-\sum_{1\le r<s\le m} (h_1^{i_1}\cdots h_{s-1}^{i_{s-1}}h_1^{j_1}\cdots h_{r-1}^{j_{r-1}}(h_r)_{j_r}(h_s)_{i_s})\tr\varphi_{rs},
\end{align*}
and
\begin{align*}\varphi(k,h)
&= \sum_{r=1}^m [\frac{j_r+i_r}{n_r}](h_1^{j_1+i_1}\cdots g_{r-1}^{j_{r-1}+i_{r-1}})\tr \varphi_{rr} \\
&
-\sum_{1\le r<s\le m} (h_1^{j_1}\cdots h_{s-1}^{j_{s-1}}h_1^{i_1}\cdots h_{r-1}^{i_{r-1}}(h_r)_{i_r}(h_s)_{j_s})\tr\varphi_{rs},
\end{align*}
and therefore $\varphi(h,k)-\varphi(k,h)\in M$. It follows that
$G'\subseteq M$ and hence $G'=M$.
\end{proof}

\begin{remark} We mention that in the above proposition, the $H$-submodule of $A$ generated by $\{T_r\tr a|1\le r\le m, a\in A\}$ is exactly the radical $\rad(A)$ of $A$, i.e., the intersection of all maximal submodule of $A$.
\end{remark}

\subsection{A technical lemma}
The following observation is handful. It is equivalent to certain combinatorial identities in characteristic $p$, which are probably well-known to experts. For the convenience of the readers we also include a proof here.

\begin{lemma}\label{keylem} Let $R$ be a commutative ring over $\Z_p$, and $x, y\in R$ be two nonzero elements such that $x^iy^{p-1-i}=0$ for any $0\le i\le p-1$. Then the following statements are equivalent:

(1) $\sum_{r=1}^{p-1}(1+x)^{ri}(1+y)_{rj}=0$ for any $1\le i, j\le p-1$;

(2) $\sum_{r=1}^{p-1}(1+x)^{ri}(1+y)_{r}=0$ for any $1\le i\le p-1$;

(3) $\sum_{r=1}^{p-1}(1+x)^{r}(1+y)_{rj}=0$ for any $1\le j\le p-1$;

(4) $x^iy^{p-2-i}=0$ for any $0\le i\le p-2$.
\end{lemma}

\begin{proof} Let $z= f(x,y)$ be in $R$ such that $f$ is a polynomial in $x$ and $y$ with zero constant term. It follows from the assumption on $x$ and $y$ that $z^{p-1}=0$, hence $(1+z)$ is invertible. If $z\ne 0$, then $1+z$ has order $p$, and $(1+z)^i=(1+z)^{i'}$ if and only if $\bar i=\bar {i'}$ in $\Z_p$.

There is no difficulty to show that $(1+z)_i$ is invertible for any $1\le i\le p-1$. In fact, $(1+z)_i = i + \tilde f(x,y)$ for some polynomial $\tilde f$ in $x$ and $y$ with zero constant term, hence $(1+z)_i = i(1 +i^{-1}\tilde f(x,y))$ is invertible. Moreover,
\[(1+z)_p = 1 + (1+z) + \cdots +(1+z)^{p-1}= ((1+z)-1)^{p-1} = f(x,y)^{p-1}= 0.\]
In particular, $(1+y)_{i+p}- (1+y)_{i} = (1+y)^i(1+y)_p =0$, and therefore $(1+y)_i= (1+y)_{i'}$ whenever $\bar i=\bar {i'}$ in $\Z_p$. Thus $(1)\Longleftrightarrow (2)\Longleftrightarrow(3)$ follows.

Next we show the equivalence $(2)\Longleftrightarrow (4)$. For each $1\le i\le p-1$, we have
\[\sum_{r=1}^{p-1}(1+x)^{ri}(1+y)_{r} =\sum_{u,v\ge 0} \lambda^i_{u,v} x^uy^v,
\]
where $\lambda^i_{u,v}$'s are integer coefficients which are independent on $x, y$ and the ring $R$.
This inspires us to consider the ring $S=\Z_p[X,Y]/\langle X^{p-1}, X^{p-2}Y, \cdots, Y^{p-1} \rangle$. It is easy to show that $S$ has a $\Z_p$-basis $X^iY^j, 0\le i, j\le p-2, 0\le i+j\le p-2$. Clearly, the subring $\Z_p[x,y]$ of $R$ can be viewed as a quotient ring of $S$ under the map $X\mapsto x$ and $Y\mapsto y$.

By direct calculation, we have
\begin{align*}
 &((1+X)^i-1)\sum_{r=1}^{p-1}(1+X)^{ri}(1+Y)_{r}\\
=& \sum_{r=2}^{p}(1+X)^{ri}(1+Y)_{r-1} - \sum_{r=1}^{p-1}(1+X)^{ri}(1+Y)_{r}\\
=& (1+X)^{pi}(1+Y)_{p-1} - \sum_{r=2}^{p-1}(1+X)^{ri}((1+Y)_r- (1+Y)_{r-1}) -(1+X)^i(1+Y)_1\\
=& (1+Y)_{p-1} -   \sum_{r=2}^{p-1}(1+X)^{ri}(1+Y)^{r-1} - (1+X)^i \\
=& -(1+Y)^{p-1}-  \sum_{r=2}^{p-1}(1+X)^{ri}(1+Y)^{r-1} - (1+X)^i \\
=& -(1+Y)^{-1}(1+ \sum_{r=2}^{p-1}(1+X)^{ri}(1+Y)^{r} + (1+X)^i(1+Y))\\
=& -(1+Y)((1+X)^{i}(1+Y))_{p}\\
=& 0
\end{align*}
 for each $1\le i\le p-1$. Using the facts $(1+X)^i-1 = X(1+X)_i$ and $(1+X)_i$ is invertible, we know that \[X\sum_{r=1}^{p-1}(1+X)^{ri}(1+Y)_{r}=0.\]
Note that the converse is also true, say $X\sum_{r=1}^{p-1}(1+X)^{ri}(1+Y)_{r}=0$ if and only if $((1+X)^i-1)\sum_{r=1}^{p-1}(1+X)^{ri}(1+Y)_{r}$. Similarly,
\begin{align*}
 Y\sum_{r=1}^{p-1}(1+X)^{ri}(1+Y)_{r}
=& ((1+Y) -1)\sum_{r=1}^{p-1}(1+X)^{ri}(1+Y)_{r}\\
=& \sum_{r=1}^{p-1}(1+X)^{ri}((1+Y)^{r}-1)\\
=& \sum_{r=1}^{p-1}(1+X)^{ri}(1+Y)^{r}- \sum_{r=1}^{p-1}(1+X)^{ri}\\
=& ((1+X)^i(1+Y))_p - ((1+X)^i)_p\\
=& 0.
\end{align*}
It forces that $\bar \lambda^i_{u, v} = 0$ in $\Z_p$ for any $u+v< p-2$.

We are left to determine $\lambda^i_{u, v}$'s with $u+v=p-2$. Note that the higher terms vanish automatically in $S$ as well in the subring $\Z_p[x,y]$.
For simplicity we set $g= 1+X$ and $h=1+Y$. We write the elements $\sum_{r=1}^{p-1}g^{ri}(h)_r$, $i=1, \cdots, p-1$ into a column vector $(F_1,F_2,\cdots, F_{p-1})^T$ with entries in $S$, say
\[
\begin{pmatrix}
   F_1 \\
   F_2\\
    \cdots\\
   F_{p-1}\\
 \end{pmatrix}
=\begin{pmatrix}
  g   & g^2 & \cdots & g^{p-1} \\
  g^2 & g^4 & \cdots & g^{2(p-1)} \\
  \cdots & \cdots & \cdots & \cdots \\
  g^{p-1} & g^{(p-1)2} & \cdots & g^{(p-1)(p-1)} \\
\end{pmatrix}
\begin{pmatrix}
   (h)_1 \\
   (h)_2\\
    \cdots\\
   (h)_{p-1}\\
 \end{pmatrix}.
\]
Consider the following column vector obtained by elementary row operations
\begin{align*}
\begin{pmatrix}
E_1\\E_2\\ \cdots\\ E_{p-1}
\end{pmatrix}
=&\begin{pmatrix}
  1 &   &   &   &   &   \\
  1 & -1 &  &  &  &  \\
  \vdots & \vdots & \ddots &  &   &   \\
  1 & -C_i^1 & \cdots & (-1)^iC_i^i &   &   \\
  \vdots & \vdots & \vdots & \vdots & \ddots &  \\
  1 & -C_{p-2}^1 & \cdots & (-1)^iC_{p-2}^i & \cdots & -C_{p-2}^{p-2} \\
\end{pmatrix}
\begin{pmatrix}
   F_1 \\
   F_2\\
    \cdots\\
   F_{p-1}\\
 \end{pmatrix} \\
 =& \begin{pmatrix}
     g       &  g^2          & \cdots & g^{p-1} \\
     g(1-g)  &  g^2(1-g^2)    & \cdots & g^{p-1}(1-g^{p-1}) \\
     g(1-g)^2&  g^2(1-g^2)^2  & \cdots & g^{p-1}(1-g^{p-1})^2 \\
     \cdots & \cdots & \cdots  & \cdots \\
     g(1-g)^{p-2}&  g^2(1-g^2)^{p-2}  & \cdots & g^{p-1}(1-g^{p-1})^{p-2} \\
   \end{pmatrix}
\begin{pmatrix}
   (h)_1 \\
   (h)_2\\
    \cdots\\
   (h)_{p-1}\\
 \end{pmatrix}.
\end{align*}

Again we may expand each $E_i$ into a polynomial in variables $X, Y$ with integer coefficients, which is independent on the ring $S$. Assume that
\[E_i= \sum_{u\ge i-1} \mu^i_{u,v} X^uY^v= X^{i-1}f_i(Y)+ X^{i}f'_i(X,Y),\] where $f_i(Y)$  and $f'_i(X,Y)$ are polynomials.
Clearly $XE_i=YE_i=0$ for each $i$, and hence $\mu^i_{u,v}=0$ in $\Z_p$ for any $u+v< p-2$.
It is direct to check that
\[\begin{pmatrix}
   (h)_1 \\
   (h)_2\\
   \cdots\\
   (h)_{p-1}\\
 \end{pmatrix}
 =\begin{pmatrix}
    C_1^1 &     &  &  \\
    C_2^1  & C_2^2  &  &  \\
    \vdots  &\vdots   & \ddots &  \\
    C_{p-1}^1  &C_{p-1}^2    &\cdots   & C_{p-1}^{p-1} \\
  \end{pmatrix}
  \begin{pmatrix}
    1 \\
    Y \\
    \cdots \\
    Y^{p-2} \\
  \end{pmatrix},
\]
then we have
\begin{align*}\begin{pmatrix}
f_1(Y)\\
f_2(Y)\\
\cdots\\
f_{p-1}(Y)
\end{pmatrix}
=&\begin{pmatrix}
    1  &  1  & \cdots & 1 \\
    1  &  2  & \cdots & p-1 \\
     \cdots    &\cdots   & \cdots & \cdots \\
    1  &  2^{p-2}  & \cdots & (p-1)^{p-2} \\
  \end{pmatrix}
\begin{pmatrix}
    C_1^1 &     &  &  \\
    C_2^1  & C_2^2  &  &  \\
    \vdots  &\vdots   & \ddots &  \\
    C_{p-1}^1  &C_{p-1}^2    &\cdots   & C_{p-1}^{p-1} \\
  \end{pmatrix}
  \begin{pmatrix}
    1 \\
    Y \\
    \cdots \\
    Y^{p-2} \\
  \end{pmatrix} \\
=& \begin{pmatrix}
     &     &  & \mu^1_{0,p-2} \\
      &   & \mu^2_{1,p-3} & *   \\
     & \cdots  &\cdots &\cdots \\
    \mu^{p-1}_{p-2,0} & \cdots  & * & *  \\
  \end{pmatrix}
   \begin{pmatrix}
    1 \\
    Y \\
    \cdots \\
    Y^{p-2} \\
  \end{pmatrix}.
\end{align*}
It follows that
\[\begin{pmatrix}
E_1\\E_2\\ \cdots\\ E_{p-1}
\end{pmatrix}
=\begin{pmatrix}
    * & \cdots    & * & \mu^1_{0,p-2} \\
    *  & \cdots  & \mu^2_{1,p-3} &   \\
    \cdots &\cdots   &   & \\
    \mu^{p-1}_{p-2,0} &   &  &  \\
  \end{pmatrix}
  \begin{pmatrix}
    X^{p-2} \\
    X^{p-3}Y \\
    \cdots \\
    Y^{p-2} \\
  \end{pmatrix}.
\]

Applying the Van der Monde determinant, we know that $\mu^1_{0,p-2} \mu^2_{1,p-3}\cdots \mu^{p-1}_{p-2,0}\ne 0$ in $\Z_p$, and hence the following matrix
\[\Lambda =
\begin{pmatrix}
    \lambda^1_{p-2,0} & \lambda^2_{p-3,1} & \cdots & \lambda^1_{0,p-2} \\
    \lambda^2_{p-2,0} & \lambda^2_{p-3,1} & \cdots & \lambda^2_{0,p-2} \\
    \cdots & \cdots & \cdots & \cdots \\
    \lambda^{p-1}_{p-2,0} & \lambda^{p-1}_{p-3,1} & \cdots & \lambda^{p-1}_{0,p-2}\\
  \end{pmatrix}
\]
 is invertible in $\Z_p$. Now the equivalence $(2)\Longleftrightarrow (4)$ follows from the equality
\[
\begin{pmatrix}
\sum_{r=1}^{p-1}(1+x)^{r}(1+y)_{r}\\
\sum_{r=1}^{p-1}(1+x)^{2r}(1+y)_{r}\\
\cdots\\
\sum_{r=1}^{p-1}(1+x)^{(p-1)r}(1+y)_{r}
\end{pmatrix}
=\Lambda
\begin{pmatrix}
x^{p-2}\\
x^{p-3}y\\
\cdots\\
y^{p-2}.
\end{pmatrix}.
\]
\end{proof}

\subsection{A criterion for a realizing group having exponent $p$} Now we consider the special case that $H$ and $A$ are both elementary $p$-groups.

Let $H=\langle h_1\rangle\times \langle h_2\rangle\times\cdots \times\langle h_m\rangle $ be an elementary $p$-group with a basis $h_1, h_2, \cdots, h_m$. Then it is well known that there exists a ring isomomorphism \[\Z_p[H]\xrightarrow{\cong} \Z_p[x_1, \cdots, x_m]/\langle x_1^{p},\cdots, x_m^p\rangle,\ \  h_i\mapsto x_i+1, 1\le r\le m, \]
where $\Z_p[H]$ is the group ring of $H$ over $\Z_p$. Clearly, the Jacobson radical $\mathrm{rad}(\Z_p[H])$ of $\Z_p[H]$ is the ideal generated by $h_1-1, \cdots, h_m-1$. Recall that the Jacobson radical of a ring is by definition the intersection of all its maximal ideals.
 Combing Proposition \ref{prop-ext-p} and Proposition \ref{prop-2cocycle}, we obtain the following criterion to determine when a realizing group has exponent $p$.

\begin{theorem} \label{thm-expp} Let $H=\langle h_1\rangle\times \langle h_2\rangle\times\cdots \times\langle h_m\rangle $ be an elementary $p$-group with a basis $h_1, h_2,\cdots, h_m$. Set $T_r=h_r-1$ for $r=1, 2, \cdots, m$. Let $(A,\rho)$ be an $H$-module, and let $(\varphi_{rs})_{1\le r\le s\le m}$ be a family of elements in $A$ and $\varphi$ the corresponding 2-cochain.
Then $\varphi$ is a 2-cocyle such that $A\rtimes_{\rho, \varphi} H$ has exponent $p$ if and only if
\begin{enumerate}
\item  $A$ is an elementary $p$-group;
\item  $T_{r_1}T_{r_2}\cdots T_{r_{p-1}} \tr a=0$ for all $1\le r_1\le r_2\le \cdots\le r_{p-1}\le m$ and $a\in A$;
\item  $\varphi_{rs}$'s satisfy the following conditions:
    \begin{enumerate}
    \item  $\varphi_{rr}=0$ for any $1\le s\le m$;
    \item  For any $1\le l\le p-1$, $1\le r_1<r_2<\cdots<r_{l+1}\le m$, and $i_1,\cdots, i_{l+1}\ge0$ with $i_1+i_2+\cdots+ i_{l+1}= p-l-1$,
        \[T_{r_1}^{i_1}T_{r_2}^{i_2}\cdots T_{r_{l+1}}^{i_{l+1}}\tr (\sum_{u=1}^l(i_u+1)T_{r_1}\cdots T_{r_{u-1}}T_{r_{u+1}}\cdots T_{r_l}\tr \varphi_{r_u,r_{l+1}})=0;\]
    \item $T_r\tr\varphi_{st} - T_s\tr\varphi_{rt}+T_t\tr\varphi_{rs}=0$ for any $1\le r<s<t\le m$.
    \end{enumerate}
\end{enumerate}
\end{theorem}

\begin{proof} By Proposition \ref{prop-2cocycle}, $\varphi$ is a 2-cocyle if and only if
\begin{enumerate}
 \item[(I.1)] $T_r\tr\varphi_{rr}=0$ for $1\le r\le m$;
 \item[(I.2)] $N_r\tr\varphi_{rs}+T_s\tr \varphi_{rr}=0$, $T_r\tr \varphi_{ss}-N_s\tr \varphi_{rs}=0$ for $1\le r<s\le m$; and
 \item[(I.3)] $T_r\tr \varphi_{st}-T_s\tr \varphi_{rt}+T_t\tr \varphi_{rs}=0$ for $1\le r<s<t\le m$.
 \end{enumerate}
By Proposition \ref{prop-ext-p}, the realizing group $A\rtimes_{\rho, \varphi} H$ has exponent $p$ if and only if
\begin{enumerate}
\item[(II.1)] $A$ has exponent $p$;
\item[(II.2)] $1+ \rho(h) + \rho(h^2) +\cdots +\rho(h^{p-1})=0$ for all $h\in H$;
\item[(II.3)] $\varphi(h,h)+\varphi(h^2,h)+\cdots +\varphi(h^{p-1},h) =0$ for all $h\in H$.
\end{enumerate}

First we claim that assuming (II.1) and (II.2), then $\varphi$ satisfies (II.3) if and only if
$\varphi_{rr} = 0$ for any $1\le r\le m$, and
\begin{equation}\label{eqM1}(\sum_{k=1}^{p-1}
(h_{1}^{\lambda_1}h_{2}^{\lambda_2}\cdots h_{r}^{\lambda_{r}})^k (h_{r+1})_{k\lambda_{r+1}})\tr (\sum_{u=1}^{r}
h_{1}^{\lambda_1}h_{2}^{\lambda_2}\cdots h_{u-1}^{\lambda_{u-1}} (h_{u})_{\lambda_u}\tr \varphi_{u,r+1})=0
\end{equation}
for any $1\le r\le m-1$, $0\le \lambda_1, \cdots, \lambda_{r+1}\le p-1$.

We mention that $(\ref{eqM1})$ holds true automatically if $\lambda_{r+1}=0$, or if $\lambda_u=0$ for all $1\le u\le r$. The proof of the claim is straightforward. In fact, by applying (II.3) to each $h_r$ we have
\[0 = \sum_{k=1}^{p-1}\varphi(h_r^k, h_r) =  \sum_{k=1}^{p-1}[\frac{k+1}{p}]\varphi_{rr}=\varphi_{rr}.
\]
 For any $1\le r\le m-1$, and any $0\le \lambda_1, \cdots, \lambda_{r+1}\le p-1$, by applying (II.3) to $h_{1}^{\lambda_1}h_{2}^{\lambda_2}\cdots h_{r}^{\lambda_r}$ and $h_{1}^{\lambda_1}h_{2}^{\lambda_2}\cdots h_{r}^{\lambda_r}h_{r+1}^{\lambda_{r+1}}$, we have
\[\sum_{k=1}^{p-1} \sum_{1\le u<v\le r}
(h_{1}^{k\lambda_1}h_{2}^{k\lambda_2}\cdots h_{v-1}^{k\lambda_{v-1}} (h_{v})_{k\lambda_v}
h_{1}^{\lambda_1}h_{2}^{\lambda_2}\cdots h_{u-1}^{\lambda_{u-1}} (h_{u})_{\lambda_u})\tr \varphi_{u,v}=0,
\]
and
\[\sum_{k=1}^{p-1} \sum_{1\le u<v\le r+1}
(h_{1}^{k\lambda_1}h_{2}^{k\lambda_2}\cdots h_{v-1}^{k\lambda_{v-1}} (h_{v})_{k\lambda_v}
h_{1}^{\lambda_1}h_{2}^{\lambda_2}\cdots h_{u-1}^{\lambda_{u-1}} (h_{u})_{\lambda_u})\tr \varphi_{u,v}=0,
\]
and $(\ref{eqM1})$ follows by taking the difference of the above equalities. Thus we have proved the necessity, and the argument works well for the sufficiency.

Now we can prove the theorem. Clearly (II.1) is the same as the condition (1) in the theorem, and by Proposition \ref{prop-rad0}, (II.2) is equivalent to (2). We will prove that under the assumptions (1) and (2), the condition (3) is equivalent to (I.1)-(I.3) and (II.3).

Assume (I.1)-(I.3) and (II.1)-(II.3). We need to prove (3). Clearly (3)(a) follows from the above statement and (3)(c) is the same as (I.3), and we are left to prove (3)(b).

By the above claim, $(\ref{eqM1})$ holds for any $1\le r\le m-1$, $0\le \lambda_1, \cdots, \lambda_{r+1}\le p-1$. By Lemma \ref{keylem}, this is equivalent to
\begin{equation}\label{eqM2}(h_{1}^{\lambda_1}h_{2}^{\lambda_2}\cdots h_{r}^{\lambda_{r}}-1)^i(h_{r+1}-1)^j\tr
(\sum_{u=1}^{l}
h_{1}^{\lambda_1}h_{2}^{\lambda_2}\cdots h_{u-1}^{\lambda_{u-1}} (h_{u})_{\lambda_u}\tr \varphi_{u,r+1})=0
\end{equation}
holds for any $0\le \lambda_1, \cdots, \lambda_r\le p-1$ and for any $i+j=p-2$.
Under the assumption $(2)$, the above equality is equivalent to
\begin{equation}\label{eqM3}(\lambda_1T_{1} + \cdots +\lambda_rT_{r})^iT_{r+1}^j\tr
(\sum_{u=1}^{l}
\lambda_u\varphi_{u,r+1})=0,
\end{equation}
and by expanding the left hand side, it turns into
\begin{equation}\label{eqM4}\sum_{j_1,j_2,\cdots, j_r\ge 0\atop j_1+ j_2 +\cdots +j_r=i+1}
\lambda_1^{j_1}\cdots \lambda_r^{j_r}
T_{r+1}^j \sum_{1\le u\le r\atop j_u>0}  C_{i; j_1,\cdots, j_u-1,\cdots, j_r} T_{r_1}^{j_1}\cdots T_{u}^{j_u-1}\cdots T_{r}^{j_r}\tr\varphi_{u, r+1}=0,
\end{equation}
where $C_{i; u_1, u_2, \cdots, u_r} =\frac{i!}{u_1! u_2!\cdots u_r!}$ is the multinomial coefficient for any $u_1+\cdots+u_l=i$.

Since the above equality holds for arbitrary $\lambda_1, \cdots, \lambda_r$, by using the Van der Monde determinant, $(\ref{eqM4})$ is equivalent to
\begin{equation}\label{eqM5}T_{r+1}^j \sum_{1\le u\le r\atop j_u>0}  C_{i; j_1,j_2,\cdots, j_u-1,\cdots, j_r} T_{1}^{j_1}T_{2}^{j_2}\cdots T_{u}^{j_u-1}\cdots T_{r}^{j_r}\tr\varphi_{u, r+1}=0
\end{equation}
for any $j_1, j_2, \cdots, j_r, j\ge 0$ such that $j_1+\cdots + j_r + j = p-2$.

Given  $j_1, j_2, \cdots, j_l, j\ge 0$ such that $j_1+\cdots + j_l + j = p-2$. Consider the subset $X=\{1\le u\le r\mid j_u>0\}= \{r_1, r_2,\cdots, r_l\}$ with $r_1<r_2<\cdots <r_l$. We set $r_{l+1}=r+1$, $i_u = j_{r_u}-1$ for $1\le u\le l$ and $i_{l+1} = j$, then $i_1+i_2 +\cdots + i_{l+1} = p-l-1$, and the above equality reads as
\begin{equation}\label{eqM6}T_{r_1}^{i_1}\cdots T_{r_{l+1}}^{i_{l+1}}\tr (\sum_{u=1}^l\frac{(i_1+\cdots+i_l+l-1)(i_u+1)}{(i_1+1)!(i_2+1)!\cdots(i_l+1)!}
        T_{r_1}\cdots T_{r_{u-1}}T_{r_{u+1}}\cdots T_{r_l}\tr \varphi_{r_u,r_{l+1}})=0,
\end{equation}
which is obviously equivalent to $(3)(b)$.

Conversely, assume that the assumptions (1)-(3) hold, we will prove (I.1)-(I.3) and (II.1)-(II.3). As we have shown above, it suffices to prove (I.1)-(I.3) and (II.3).

Obviously (I.1) holds for each $\varphi_{rr}=0$ by the assumption (3)(a); (I.2) follows from the assumption (2) and (3)(a); and (I.3) is the same as (3)(c). We are left to prove (II.3).

By the claim in the begin of the proof, it suffices to show that (\ref{eqM1}) holds for any $1\le r\le m-1$, $0\le \lambda_1, \cdots, \lambda_{r+1}\le p-1$, which has been shown to be equivalent to (3)(b). This completes the proof.
\end{proof}

Then combined with Proposition \ref{prop-der-abext}, we have a description of the derived subgroup of the resulting group of an abelian datum.

\begin{corollary}\label{cor-der-abext} Let $(H, A, \rho, \varphi)$ be an abelian datum such that $G=A\rtimes_{\rho,\varphi} H$ has exponent $p$.
 Then $G'$ is exactly the $H$-submodule of $A$ generated by $\mathrm{Im}(\varphi)$ and $\rad(A)$. In particular, $A=G'$ if and only if $A$ is generated by $\mathrm{Im}(\varphi)$ as an $H$-module.
If moreover, $\rho$ is a trivial action, then $A=G'$ if and only if $A=\mathrm{Im}(\varphi)$.
\end{corollary}

If the group action $\rho$ is trivial, then $\rho(T_r)=0$ for all $r$, and hence all conditions concerning $T_r$'s in the theorem automatically hold.

\begin{corollary}\label{cor-expp2} Let $H$ and $A$ be elementary $p$-groups such that $H$ acts on $A$ trivially. Let $(\varphi_{rs})_{1\le r\le s\le m}$ be a collection of elements in $A$, and $\varphi$ the
corresponding 2-cochain defined by (\ref{formula-cochain}). Then
\begin{enumerate}
\item $\varphi$ is a 2-cocycle; and $\varphi$ is a 2-coboundary if and only if $\varphi=0$, if and only if $\varphi_{rs}=0$ for all $r, s$;
\item $A\rtimes_{\varphi} H$ has exponent $p$ if and only if $\varphi_{rr}=0$ for $r=1, 2, \cdots, m$.
\end{enumerate}
\end{corollary}

In case $p=3$, the conditions in Theorem \ref{thm-expp} are much more simplified.

\begin{corollary}\label{cor-expp1} Let $H=\langle h_1\rangle\times \langle h_2\rangle\times\cdots \times\langle h_m\rangle $ be an elementary $3$-group, where $h_1, h_2,\cdots, h_m$ is a basis. Let $(A,\rho)$ be an $H$-module, and let $(\varphi_{rs})_{1\le r\le s\le m}$ be a family of elements in $A$ and $\varphi$ the corresponding 2-cochain.
Then $\varphi$ is a 2-cocyle such that $A\rtimes_{\rho, \varphi} H$ has exponent $3$ if and only if
\begin{enumerate}
\item  $A$ is an elementary $3$-group;
\item  $(h_r-1)(h_s-1) \tr a=0$ for any $1\le r\le s\le m$ and $a\in A$;
\item  $\varphi_{st}$'s satisfy the following conditions:
    \begin{enumerate}
    \item  $\varphi_{rr}=0$ for any $1\le r\le m$;
    \item  $h_r\tr \varphi_{rs} = h_s\tr \varphi_{rs}= \varphi_{rs}$ for any $1\le r< s\le m$;
    \item  $(h_r-1)\tr \varphi_{st}=-(h_s-1)\varphi_{rt}=(h_t-1)\varphi_{rs}$ for any $1\le r<s<t\le m$.
    \end{enumerate}
\end{enumerate}
\end{corollary}

\begin{proof} Consider the condition $(3)(b)$ in Theorem \ref{thm-expp}. The case $l=1$ reads as
$T_r\tr\varphi_{rs} = 0$ and $T_s\tr\varphi_{rs}=0$ for any $1\le r< s\le m$, which is equivalent to $h_r\tr \varphi_{rs} = h_s\tr \varphi_{rs}= \varphi_{rs}$. The case $l=2$ reads as
$T_r\tr\varphi_{st} + T_s\tr\varphi_{rt}=0$ for all $1\le r<s<t\le m$, which gives $T_r\tr\varphi_{st}=- T_s\tr\varphi_{rt}$. Now $T_r\tr\varphi_{st} - T_s\tr\varphi_{rt}+T_t\tr\varphi_{rs}=0$ is equivalent to $T_t\tr\varphi_{rs}=2T_s\tr\varphi_{rt}=-T_s\tr\varphi_{rt}$.
\end{proof}

We also draw the following consequence, which  says that for a group of exponent 3 and order $\le 3^6$, the derived subgroup will be contained in its center.

\begin{corollary}\label{cor-expp3}
Let $\D= (H, A,\rho, \varphi)$ be a datum such that $G=H\rtimes_{\varphi} A$ has exponent $p$ and $G'=A$. Assume $A=\mathbb{Z}_3^n$ and $H=\mathbb{Z}_3^m$. Then $\rho$ is trivial unless  $m\ge3$ and $n\ge4$.
\end{corollary}
\begin{proof} Keep the above notations $h_1, \cdots, h_m$ and $a_1, \cdots, a_n$, and $\varphi_{rs}$, $1<r<s\le m$. Set $T_r= h_r-1$. Then $H$ acts trivially on $A$ if and only if $T_r\tr a_u=0$ for all $1\le r\le m$ and $1\le u\le n$.

Assume that $\rho$ is nontrivial. We will show that $m\ge3$ and $n\ge4$.

First we prove $m\ge 3$. By Proposition \ref{prop-der-abext},  $G(\D)'$ is the submodule generated by $T_r\tr a_u$'s and $\varphi_{rs}$'s. If $m=1$, then $G(\D)'=T_1\tr A=\mathrm{rad}(A)$, which never equals $A$ by Nakayama Lemma. If $m=2$, then $T_1(\varphi_{12})= T_2\tr\varphi_{12}=0$, which means that $\varphi_{12}\in \mathrm{soc}(A)$, the socle of the $H$-module $A$, say the sum of simple submodules of $A$. Thus $G(\D)'\subseteq \mathrm{soc}(A) + \mathrm{rad}(A)$, which equals $A$ only if $A$ is a simi-simple $H$-module, or equivalently, $H$ acts on $A$ trivially. Thus we must have $m\ge 3$.

Now we assume $m\ge 3$. We claim that there exist some $1\le r< s< t\le m$ such that $\varphi_{rs}, \varphi_{st}, \varphi_{st}$ are linearly independent in the quotient space $A/\mathrm{soc}(A)$. Otherwise, for any $1\le r< s <t \le m$, we have $ \varphi_{rs} \in \Z_p \varphi_{rt} + \Z_p\varphi_{st} +\mathrm{soc}(A)$. Recall that $a\in \mathrm{soc}(A)$ if and only if $T_r\tr a=0$ for all $1\le r\le m$. Then $T_t\tr \varphi_{rs} = 0$, for $T_t\tr \varphi_{rt}= T_t\tr\varphi_{st} = 0$ and $T_t\tr(\mathrm{soc}(A)) = 0$. It follows that $T_s\tr \varphi_{rt} = T_r\tr\varphi_{st}=0$, and hence $\varphi_{rs}\in \mathrm{soc}(A)$. Therefore $A = G(\D)'\subseteq \mathrm{soc}(A)$, which forces that $A= \mathrm{soc}(A)$ which happens only when $A$ is a semi-simple module.

Let $1\le r< s< t\le m$ be such that $\varphi_{rs}, \varphi_{st}, \varphi_{st}$ are linearly independent in the quotient space $A/\mathrm{soc}(A)$. Using the fact that the socle of a finite dimensional module is always nonzero, we know that $n= \dim_{\Z_p} A\ge 4$.
\end{proof}

\begin{remark} We mention that the above corollary does not hold true in case $p>3$. For instance, let $G$ be the group
\[\langle a, b, c, d\mid a^p=b^p=c^p=d^p=1, [a,b]=c, [a,c]=d, [a,d]=[b,c]=[b,d]=[c,d]=1\rangle.
\]
It is easy to show that $G'=\langle c,d\rangle$, and $Z(G)=\langle d \rangle$. Now $|G'|=p^2$, and the $G$-action on $G'$ is nontrivial.
\end{remark}

\subsection{Nilpotency class of an extension of elementary $p$-groups} In this subsection, we will show that the nilpotency class of the realizing group $A\rtimes_{\rho, \varphi}H$ relates closely to the Loewy length of $A$ as an $H$-module.

\begin{proposition}\label{prop-cn-mgp} Let $(A, H, \rho, \varphi)$ be an abelian datum with $A=\Z_p^m$ for some $m>0$. Let $c$ be the nilpotency class number of $A\rtimes_{\rho,\varphi}H$, and $l=\ell\ell(_HA)$ the Loewy length. Then
\begin{enumerate}
\item $l\le c\le l+1$, and $c=l$ if and only if there exists some $\psi$ cohomologous to $\varphi$ such that $\psi(h_1,h_2)\in \soc^l(A)$ and $\psi(h_1, h_2)- \psi(h_2, h_1)\in \soc^{l-1}(A)$ for any $h_1, h_2\in H$;
\item $c\le p$ if in addition $A\rtimes_{\rho,\varphi}H$ has exponent $p$.
\end{enumerate}
\end{proposition}

\begin{proof} (1) We set $G= A\rtimes_{\rho,\varphi}H$. Then $l=\ell\ell(_HA)=\ell\ell(_GA)$. By definition $\soc^{l-1}(A)\ne \soc^l(A)=A$. We use induction on $l$.

If $l=1$, then $A$ is semisimple and $H$ acts trivially on $A$. Then $A\in Z(G)$ and $H\cong G/A$ is ableian, hence $c=1$ or $2$. Clearly $c=1$ if and only if $G$ is commutative, if and only if $\varphi(h_1,h_2)=\varphi(h_2, h_1)$ for all $h_1, h_2\in H$.

Now we assume (1) holds for data with $\ell\ell(A)\le l-1$. Suppose $\ell\ell(A)=l\ge 2$. By Lemma \ref{lem-abel-extension},
there exist some $\tilde\varphi$ cohomolgous to $\varphi$ and a subgroup $K\le H$, such that
\begin{enumerate}
   \item[(a)] $K$ acts on $A$ trivially;
   \item[(b)] $\tilde\varphi(k,h)\in \soc(A)$ and $\tilde\varphi(k, h)=\tilde\varphi(h, k)$ for any $k\in K$ and $h\in H$;
   \item[(c)] $Z(A\rtimes_{\rho,\tilde\varphi}H)=\{(a,k)\mid a\in\soc(A), k\in K\}$.
\end{enumerate}
Note that $A\rtimes_{\rho,\tilde\varphi}H\cong A\rtimes_{\rho,\varphi}H$. Then
$$G/Z(G)\cong \frac{A\rtimes_{\rho,\tilde\varphi}H}{Z(A\rtimes_{\rho,\tilde\varphi}H)}
\cong\frac {A}{\soc(A)}\rtimes_{\rho', \varphi'}\frac H K,$$
where $\rho'$ is the induced action of $H/K$ on $A/\soc(A)$, and $\varphi'\colon H/K \times H/K\to A/\soc(A)$ the 2-cocycle induced from $\tilde\varphi$.
By definition $\ell\ell(A/\soc(A))=l-1$. By the induction hypothesis, the class number of the quotient group $G/Z(G)$ is equal to $l-1$ or $l$, and hence the class number of $G$ is $l$ or $l+1$.

Now assume $\varphi$ is cohomologous to some $\psi$ with $\psi(h_1,h_2)\in \soc^l(A)$ and $\psi(h_1, h_2)- \psi(h_2, h_1)\in \soc^{l-1}(A)$ for any $h_1, h_2\in H$. Then by easy calculation, $\soc^{l-1}(A)\vartriangleleft A\rtimes_{\rho, \psi} H$, and $\tilde H=A\rtimes_{\rho, \psi} H/\soc^{l-1}(A)$ is abelian, and hence $G\cong \soc^{l-1}(A)\rtimes_{\tilde \rho, \tilde\psi} \tilde H$ for some induced action $\tilde\rho$ and 2-cocycle $\tilde\psi$. Since $\ell\ell(\soc^{l-1}(A))=l-1$, by the induction hypothesis, we have $l-1\le c\le l$ and it forces that $c=l$.

Conversely, we assume $c=l$. Then by definition $G/Z(G)$ has nilpotency class number $c-1=l-1$ and $\ell\ell(A/\soc(A))=l-1$. Now by the induction hypotheses, there exists some $\varphi_1\colon H/K\times H/K\to A/\soc(A)$, which is cohomologous to $\varphi'$ and
\[\varphi_1(\bar h_1, \bar h_2)\in \soc^{l-1}(A/\soc(A))=\soc^l(A)/\soc(A),\]
and \[ \varphi_1(\bar h_1, \bar h_2)-\varphi_1(\bar h_2, \bar h_1)\in \soc^{l-2}(A/\soc(A))=\soc^{l-1}(A)/\soc(A)\]
for any $h_1, h_2\in H$, where $\bar h_1, \bar h_2$ are the image of $h_1$ and $h_2$ under the canonical projection map $\pi_H\colon H\to H/K$.

Since $\varphi'$ is cohomolgous to $\varphi_1$, we have $\varphi_1-\varphi' = d(\bar f)$ for some $\bar f\colon H/K\to A/\soc(A)$. Let $\pi_A\colon A\to A/\soc(A)$ be the canonical projection map. Take a set theoretic section map $s\colon A/\soc(A)\to A$, say $\pi_A\circ s= \id_{A/\soc(A)}$.
Clearly $f=s\circ \pi_A\circ \pi_H\colon H\to A$ lifts the map $\bar f$. Then $\psi= \tilde\varphi+ d(f)$ is the desired 2-cocycle, that is, $\psi(h_1,h_2)\in \soc^l(A)$ and $\psi(h_1, h_2)- \psi(h_2, h_1)\in \soc^{l-1}(A)$ for any $h_1, h_2\in H$.

(2) If in addition $G$ has exponent $p$, then $H\cong \Z_p^n$ for some $n$, and applying Theorem \ref{thm-expp} (2) and Proposition \ref{prop-loewylength} we have $\ell\ell(A)\le p-1$, and hence $c\le p$.
\end{proof}

\begin{remark} If we consider an arbitrary abelian datum, then the inequality $c\le \ell\ell(A)+1$ in the above proposition still holds, while the one $c\ge \ell\ell(A)$ may not hold true in general. For instance, we take $A=\Z/p^2\Z$ and $H$ to be the trivial group. Then $\ell\ell(A)=2$ while the realizing group is abelian and hence has nilpotency class 1.
\end{remark}

\section{Matrix presentation and extensions of elementary $p$-groups}

In this section, we discuss the classification problem of extensions of elementary $p$-groups, say groups of the form $\mathbb{Z}_p^n\rtimes_{\rho,\varphi}\mathbb{Z}_p^m$. From now on, $H$ and $A$ are both assumed to be elementary $p$-groups without otherwise stated, where $p\ge 3$ is a prime number.
\subsection{Matrix presentation of a $p$-elementary datum}
Let $\D=(H, A, \rho, \varphi)$ be a $p$-elementary datum of type $(m,n)$ for some positive integer $m$ and $n$, say $H\cong\Z_p^m$ and $A\cong\Z_p^n$. Let $\B_H=\{h_1, h_2,\cdots, h_m\}$ be a basis of $H$ and $\B_A=\{a_1, a_2,\cdots, a_n\}$ be a basis of $A$. We mention that all bases discussed here will be ordered.

Then under the basis $\B_A$, any linear operator $\rho(h)\colon A\to A$, $h\in H$, is represented by some $n\times n$ matrix $\Gamma_{\B_A}(h)$. Set $\Gamma^{(r)}=\Gamma_{\B_A}(h_r)$. Thus for fixed $\B_H$ and $\B_A$, $\rho$ is uniquely determined by an (ordered) collection of matrices $\Gamma^{(1)},\Gamma^{(2)}, \cdots, \Gamma^{(m)}$.

Set $\varphi_{rr} = \varphi(1,h_r) + \varphi(h_r, h_r) + \cdots + \varphi(h_r^{p-1}, h_r)$ for $1\le r\le m$, and $\varphi_{rs}= \varphi(h_r, h_s)-\varphi(h_s, h_r)$ for $1\le r\neq s\le m$. Clearly $\varphi_{rs}=-\varphi_{sr}$ for $r\ne s$. By Proposition \ref{prop-2cocycle}, $\varphi$ is cohomologous to the cocycle corresponding to $(\varphi_{rs})_{1\le r< s\le m}$. Since $\B_A$ is a basis of $A$, each $\varphi_{rs}=\varphi^1_{rs}a_1 + \cdots +\varphi^n_{rs}a_n$ for some uniquely determined $\varphi^i_{rs}\in \Z_p$, $i=1,2,\cdots, n$. Set
$\Phi^{(i)}=(\varphi^i_{rs})_{m\times m}\in M_m(\Z_p)$ to be the $m\times m$ matrix with the $(r,s)$-entry given by $\varphi^i_{rs}$.

\begin{definition} The collection of matrices $(\Gamma^{(1)}, \cdots, \Gamma^{(m)}; \Phi^{{(1)}}, \cdots, \Phi^{(n)})$, denoted by $\M(\D;\B_H,\B_A)$, is called the \emph{matrix presentation} of the datum $\D=(H, A, \rho, \varphi)$ with respect to the bases $\B_H$ and $\B_A$, where $\Gamma^{(r)}$'s and $\Phi^{(i)}$'s are defined as above.
\end{definition}

Conversely, let $\M=(\Gamma^{(1)}, \cdots, \Gamma^{(m)}; \Phi^{{(1)}}, \cdots, \Phi^{(n)})$ be a collection with $\Gamma^{(1)}, \cdots, \Gamma^{(m)}\in \GL_n(\Z_p)$, and $\Phi^{{(1)}}, \cdots, \Phi^{(n)}\in M_m(\Z_p)$ such that $\Phi^{(i)}_{rs}=-\Phi^{(i)}_{sr}$ for any $i$ and $r\ne s$. Then $\M$ is a matrix presentation of some datum  with respect to suitable bases.

One can easily read a presentation of the realizing group of $\D$ from its matrix presentation. Note that we may identify $A$ and $H$ with subsets of $G(\D)$, cf. Remark \ref{rem-data}.

\begin{proposition}\label{prop-pres-from-mrep} Keep the above notations. Then the group $G(\mathcal D)$ is generated by elements $a_1, a_2, \cdots, a_n$ and $h_1, h_2, \cdots, h_m$, subject to the relations:
\[a_i^p=h_r^p= [a_i, a_j]=1, [h_r, h_s] = \prod_{i=1}^na_i^{\varphi^i_{rs}}, [h_r,a_i]= a_i^{-1}\prod_{j=1}^{n} a_j^{(\Gamma^{(r)})_{ij}},
\]
$1\le i,j\le n, 1\le r,s\le m$, where $(\Gamma^{(r)})_{ij}$ is the $(i, j)$-entry of $\Gamma^{(r)}$.
\end{proposition}

\begin{proof}
Let $G$ be the group generated by  $a_1, a_2, \cdots, a_n$ and $h_1, h_2, \cdots, h_m$ and subject to the relations given in the proposition. Clearly, by Lemma \ref{lem-abel-extension} (2) there exists a natural epimorphism from $G$ to $G(\mathcal D)$. Moreover,
each element in $g$ can be written in the form $a_1^{\alpha_1}\cdots a_n^{\alpha_n} h_1^{\beta_1}\cdots h_m^{\beta_m}$, it follows that $|G|\le p^{m+n} = |G(\mathcal D)|$, and hence $G\cong G(\mathcal D)$.
\end{proof}

Let $\tilde{\mathcal D}=(\tilde H, \tilde A, \tilde\rho, \tilde\varphi)$ be another datum which is equivalent to $\D$ via isomorphisms $\sigma_H\colon H\to \tilde H$ and $\sigma_A\colon A\to \tilde A$. Let $\B_{\tilde H}=\sigma_H(\B_H)$ and $\B_{\tilde A}= \sigma_A(\B_A)$ be the corresponding bases in $H$ and $A$. Let $\M(\tilde\D;\B_{\tilde H}, \B_{\tilde A})=(\tilde\Gamma^{(1)},\cdots, \tilde\Gamma^{(m)}; \tilde\Phi^{(1)}, \cdots, \tilde\Phi^{(n)})$ be the matrix presentation of $\tilde\D$ with respect to the bases $\B_{\tilde H}$ and $\B_{\tilde A}$. We have the following observation.

\begin{proposition}\label{prop-mr-eq}  Keep the above notations. Then
\begin{enumerate}
\item $\tilde \Gamma^{(r)}=\Gamma^{(r)}$ for any $1\le r\le m$;
\item If $(1-\rho(h_r))^{p-1}=0$ for any $1\le r\le m$, then there exists some $n\times m$ matrix $X$ over $\Z_p$, such that $\Phi^{(i)}-\tilde\Phi^{(i)}= W_iX-(W_iX)^T-X+X^T$ for each $i$, where $X^T$ is the transpose of $X$ and $W_i$ is an $m\times n$ matrix with $(W_i)_{rj}=(\Gamma^{(r)})_{ij}$.
\item If $H$ acts trivially on $A$, then $\Gamma^{(r)}= \tilde\Gamma^{(r)}=I_n$ for any $1\le r\le m$, and $\Phi^{(i)}=\tilde\Phi^{(i)}$ for any $1\le i\le n$.
\end{enumerate}
\end{proposition}

The proposition is an easy consequence of Proposition \ref{prop-2cocycle} and Definition \ref{def-equi-data}. The proof is given by routine check and we omit it here.

We can rewrite the above equalities by using the notion of \emph{block transpose} of a matrix. Let $M$ be a $m\times n$ matrix. Assume $m=rm'$ and $n=sn'$ for some positive integers $r, s, m', n'$. We may view $M$ as a $r\times s$ block matrix with all blocks of size $m'\times n'$, say
\[M=\begin{pmatrix}
                             M_{11} & M_{12} & \cdots & M_{1s} \\
                              M_{21} & M_{22} & \cdots & M_{2s} \\
                             \cdots & \cdots & \cdots & \cdots \\
                              M_{r1} & M_{r2} & \cdots & M_{rs} \\
                           \end{pmatrix},
\]
where each $M_{ij}$ is a $m'\times n'$ matrix. We set
\[ T_{m',n'}(M)= \begin{pmatrix}
                             M_{11} & M_{21} & \cdots & M_{r1} \\
                              M_{12} & M_{22} & \cdots & M_{r2} \\
                             \cdots & \cdots & \cdots & \cdots \\
                              M_{1s} & M_{2s} & \cdots & M_{rs} \\
                           \end{pmatrix},
\]
the $sm'\times rn'$ matrix obtained by transposing the blocks of $M$, called the block transpose of $M$ of size $m'\times n'$. In particular,  $T_{1,1}(M) = M^T$, the transpose of $M$, and $T_{m,n}(M) = M$.

\begin{remark}\label{rem-mr-eq} We may put $\Phi^{(i)}$'s into an $m\times m$ block matrix $\Phi$ with the $(r,s)$-entry given by the $n$-dimensional column vector $(\varphi_{rs}^1, \varphi_{rs}^2, \cdots, \varphi_{rs}^n)^T$, or equivalently, an $mn\times m$ matrix with the $((r-1)m+i, s)$-entry $\varphi_{rs}^i$. Similarly, we put $\tilde\Phi^{(i)}$'s into one matrix $\tilde\Phi$.

Conversely, given an $mn\times n$ matrix $\Phi$, we may obtain a sequence of $m\times m$ matrix $\Phi^{(1)}, \cdots, \Phi^{(n)}$ such that $(\Phi^{(i)})_{rs} = (\Phi)_{(r-1)n +i, s}$ for any $1\le i\le n$ and $1\le r, s\le m$, where $(\Phi^{(i)})_{rs}$ is the $(r,s)$-entry of $\Phi^{(i)}$ and $(\Phi)_{(r-1)n +i, s}$ is the $((r-1)n +i, s)$-entry of $\Phi$.

Then the equalities in Proposition \ref{prop-mr-eq} (2) can be rewritten as $$\Phi-\tilde\Phi= \Gamma X- T_{n,1}(\Gamma X),$$ where $\Gamma$ is the block matrix   $\begin{pmatrix}
                                \Gamma^{(1)}-I\\
                                \Gamma^{(2)}-I\\
                                \cdots\\
                                \Gamma^{(m)} -I\\
                                \end{pmatrix}.$
\end{remark}

\subsection{Changing of basis} We will use the matrix presentations to study group extensions. Then a naive question arises: given two matrix presentations, how to detect whether they come from equivalent data. Due to Proposition \ref{prop-mr-eq}, we need only to consider the matrix presentations of a given $p$-elementary datum $\D=(H, A, \rho, \varphi)$ under different choices of bases of $A$ and $H$.

Let $\B_H'=(h_1', \cdots, h_m')$ and $\B_A'=(a_1', \cdots, a_n')$ be another bases of $H$ and $A$ respectively. Let $C_H$ and $C_A$ be the transition matrices, say $\B_H'=\B_HC_H$ and $\B_A'=\B_AC_A$. Assume $C_H=(\beta_{uv})_{m\times m}$ and $C_A=(\alpha_{ij})_{n\times n}$. Then $h_r'=\prod_{u=1}^mh_u^{\beta_{ur}}$ for any $1\le r\le m$, and $a_i'=\sum_{l=1}^n\alpha_{li}a_l$. Notice that we write $H$ as a multiplicative group and $A$ an additive group.

Assume $\M(\D;\B_H', \B_A')=(\Gamma'^{(1)}, \cdots, \Gamma'^{(m)}; \Phi'^{{(1)}}, \cdots, \Phi'^{(n)})$, and let $\Phi$ and $\Phi'$ be the $mn\times m$ matrices obtained from $(\Phi^{(1)}, \cdots, \Phi^{(n)})$ and $(\Phi'^{(1)}, \cdots, \Phi'^{(n)})$ as in Remark \ref{rem-mr-eq}.
For any $1\le r, u\le m$, we set $h_{r;u}=h_1^{\beta_{1r}}h_2^{\beta_{2r}}\cdots h_{u-1}^{\beta_{u-1, r}} (h_u)_{\beta_{ur}}$, and use $\Sigma_{ru}$ to denote the matrix of $\rho(h_{r;u})$ under the basis $\B_A$. Set $\Sigma= (\Sigma_{ru})_{m\times m}$ to be the $m\times m$ blocks matrix.

\begin{proposition}\label{prop-mr-bc} Keep the above notations. Assume that $1+\rho(h)+\cdots + \rho(h)^{p-1}=0$ and $\varphi(1,h) + \varphi(h, h) + \cdots + \varphi(h^{p-1}, h)=0$ for all $h\in H$, and $\varphi$ is given by (\ref{formula-cochain}). Then
\begin{enumerate}\item $\Gamma'^{(r)}= C_A^{-1}(\Gamma^{(1)})^{\beta_{1r}}(\Gamma^{(2)})^{\beta_{2r}}\cdots (\Gamma^{(m)})^{\beta_{mr}}C_A$ for any $1\le r\le m$;
\item $ \Phi' = YT_{m,1}(\Sigma T_{m, 1}(\Sigma \Phi))$, where $Y= \diag(C_A^{-1}, C_A^{-1}, \cdots, C_A^{-1})$ is the $m\times m$ block diagonal matrix with equal main diagonal block $C_A^{-1}$.
\end{enumerate}
\end{proposition}

\begin{proof} The proof of (1) follows from the fact that the matrices of a linear transformation under two bases are conjugate to each other. We sketch the proof of (2). By assumption,
 $$\varphi_{rs} = \varphi(h_r, h_s) -\varphi(h_s, h_r)= -\varphi_{sr}$$ for any $1\le r\ne s\le m$, and $\varphi_{rr}=0$ for any $r$. Then for $r\ne s$ we have
\begin{align*}\varphi'_{rs} &= \varphi(h_r', h_s') -\varphi(h_s', h_r')\\
&=- \sum_{1\le u<v\le m}(h_{r;v}h_{s;u}\tr \varphi_{uv})+ \sum_{1\le u<v\le m}(h_{r;u}h_{s;v}\tr \varphi_{uv})\\
&=\sum_{u=1}^{m}\sum_{v=1}^{m}h_{r;u}h_{s;v}\tr \varphi_{uv},\end{align*}
and
\[\varphi'_{rr} = 0 = \sum_{u=1}^{m}\sum_{v=1}^{m}h_{r;u}h_{r;v}\tr \varphi_{uv}\]
for any $r$, where we use the assumptions $\varphi_{vu}=-\varphi_{uv}$ and $\varphi_{rr}=0$. Now we write the above equalities into the matrix form to obtain $(2)$.
\end{proof}

\begin{remark} \label{rem-mr-bc} By Proposition \ref{prop-2cocycle}(3), $\varphi$ is cohomologous to the cocycle obtained from $(\varphi_{rs})_{1\le r\le s\le m}$, thus without loss of generality we may assume that $\varphi$ is given by (\ref{formula-cochain}). Moreover, if the realizing group $G(\D)$ has exponent $p$, then $1+\rho(h)+\cdots + \rho(h)^{p-1}=0$ and $\varphi(h,h)+\varphi(h^2,h)+\cdots +\varphi(h^{p-1},h) =0$ for all $h\in H$, see Proposition \ref{prop-ext-p}. In this case, all $\Phi^{(i)}$'s are $m\times m$ anti-symmetric matrices.
\end{remark}

The formula in the proposition is a bit complicated, while if the group action $\rho$ is trivial, then the it becomes much easier to understand.

\begin{corollary}\label{cor-mr-ta-bc} Keep assumptions as in Proposition \ref{prop-mr-bc}. Assume that $H$ acts trivially on $A$. Then
\begin{enumerate}\item  $\Gamma^{(r)}=\Gamma'^{(r)}= I_n$ for any $1\le r\le m$;
\item $\Phi'^{(i)}=\sum_{l=1}^n \hat \alpha_{i,l}C_H\Phi^{(l)}C_H^T$ for any $1\le i\le n$, where the matrix $(\hat \alpha_{ij})_{n\times n}$ is the inverse of $C_A$.
\end{enumerate}
\end{corollary}

In fact, since $H$ acts trivially on $A$, then $\Sigma_{ru}=\beta_{ur}I_n$ for any $1\le r, u\le m$, and the conclusion follows.

\subsection{The classification theorem: trivial action case} In this subsection, we focus on the case such that $H$ acts trivially on $A$. In this situation, the realizing group has nilpotency class 2, and any class 2 groups of exponent $p$ is obtained from such a datum.

Let $\D=(H, A, 1, \varphi)$ be a $p$-elementary datum of type $(m,n)$ with realizing group having exponent $p$. Consider its matrix presentation. Obviously $\Gamma^{(r)}=I_n$ for $r=1, 2,\cdots, m$ for any choice of bases $\B_H$ and $\B_A$. Thus $\M(\D; \B_H, \B_A)= (I_n, \cdots, I_n; \Phi^{(1)}, \Phi^{(2)}, \cdots, \Phi^{(n)})$. As explained in Remark \ref{rem-mr-eq}, we may assume $\varphi$ is given by (\ref{formula-cochain}), and every $\Phi^{(i)}$ is anti-symmetric.

As it was shown in Proposition \ref{prop-der-abext}, $(G(\D))'= \langle \varphi_{rs}\mid 1\le r,s\le m \rangle$, the linear span of $\varphi_{rs}$'s, which is equal to the $H$-submodule generated by $\{\varphi_{rs}\mid 1\le r, s\le m\}$. Note that the $H$-action on $A$ is assumed to be trivial. We have an easy observation.

\begin{lemma} Keep the above notations. Then $\dim((G(\D))')=\rank(\Phi^{(1)}, \Phi^{(2)}, \cdots, \Phi^{(n)})$,
the dimension of the subspace of $\Asym_m$ spanned by $\Phi^{(1)}, \Phi^{(2)}, \cdots, \Phi^{(n)}$. In particular,
$G(\D)'=A$ if and only if $\Phi^{(1)}, \Phi^{(2)}, \cdots, \Phi^{(n)}$ are linearly independent.
\end{lemma}

\begin{proof} The proof is given by elementary linear algebra. By applying $T_{1,m}$, we turn $\Phi^{(i)}$ into a row vector, say
$$R_i= T_{1,m}(\Phi^{(i)}) = (\varphi_{11}^{i},\cdots, \varphi_{1m}^{i},\varphi_{21}^{i}, \cdots, \varphi_{2m}^{i},\cdots, \varphi_{m1}^{i},\cdots, \varphi_{mm}^{i}).$$
Then $\rank(\Phi^{(1)}, \Phi^{(2)}, \cdots, \Phi^{(n)})=\rank(R_1, \cdots, R_n)$. We consider the block matrix
$$X=\begin{pmatrix}
  R_1 \\
  R_2 \\
  \cdots \\
  R_m \\
\end{pmatrix}.$$
Clearly $\dim(\langle \varphi_{rs}\mid 1\le r,s\le m \rangle)$ is equal to the column rank of the matrix $X$,  which is equal to the row rank of $X$, and the conclusion follows.
\end{proof}

Let $A_1= \langle \varphi_{rs}\mid 1\le r,s\le m \rangle$ and $A_2$ be a complement of $A_1$ in $A$, say $A= A_1\oplus A_2$. Let $\varphi_1\colon H\times H\to A_1$ be the 2-cocycle whose composite with the embedding $A_1\subseteq A$ gives rise to $\varphi$. Then $G(\D)= (A_1, H) \times (A_2, 1)\cong G(\tilde \D) \times A_2$, where $\tilde\D=(H, A_1, 1, \varphi_1)$. Thus we have proved the following result.

\begin{lemma} Let $\D=(H, A, 1, \varphi)$ be a $p$-elementary data of type $(m, n)$. Then $G(\D)'\cong \Z_p^d$ for some $d\le n$, and $G(\D)\cong G(\tilde \D)\times \Z_p^{n-d}$ for some $p$-elementary datum $\tilde\D$ of type $(m,d)$.
\end{lemma}

We denote by $\mathcal E_p(m, n)$ the equivalence classes of those $p$-elementary data $\D=(H, A, 1, \phi)$ of type $(m, n)$ such that  $G(\D)$ has exponent $p$, or equivalently, $\varphi(h,h)+\varphi(h^2,h)+\cdots +\varphi(h^{p-1},h) =0$ for all $h\in H$; and denote by $\mathcal E_p(d; m, n)$ the subclasses such that $(G(\D))' = \Z_p^d$ for any $0\le d\le m$. By definition, data in  $\mathcal E_p(m, n)$
exactly give all central extensions of $\Z_p^m$ by $\Z_p^n$.

\begin{lemma}\label{lem-class2todata} (1) For any $[\D]\in \mathcal E_p(m, n)$, the realizing group $G(\D)$ has nilpotency class 2 and exponent $p$.

(2) Assume $[\D_1], [\D_2]\in \mathcal E_p(m, n)$. Then  $G(\D_1)\cong G(\D_2)$ if and only if $[\D_1]=[\D_2]$.

(3) Let $G$ be a group with exponent $p$ and nilpotencey class 2  such that $|G|=p^{m+n}$ and $|G'|= p^n$.  Then $G\cong G(\D)$ for some $[\D]\in \mathcal E_p(n; m, n)$.
\end{lemma}

\begin{remark} We use $\mathcal G_2(p;m,n)$ to denote the isoclasses of class 2 groups which have order $p^{m+n}$ and exponent $p$ and whose derived subgroup is isomorphic to $\Z_p^n$. Clearly the disjoint union $\bigsqcup_{0\le d< n} \mathcal G_2(p;m-d,d)$ gives isocalsses of all class 2 groups of order $p^n$ and exponent $p$.
\end{remark}

Recall the notion of Grassmannian. Let $V$ a vector space over a field $\mathbbm k$. We use $\Gr(d, V)$ to denote the set of all $d$-dimensional subspaces of $V$ for any integer $d\ge 0$, and $\Gr(\le d, V)$ the set of all subspaces with dimension no greater than $d$.

We use $\Asym_m(\Z_p)$, or simply $\Asym_m$ when no confusion arises, to denote the subspace of $M_m(\Z_p)$ consisting of $m\times m$ anti-symmetric matrices. Clearly $\GL_m(\Z_p)$ acts on $\Asym_m(\Z_p)$ congruently, that is, $P\cdot M = PMP^T$ for any $P\in \GL_m(\Z_p)$ and any $M\in \Asym_m(\Z_p)$.
The congruence action induces an action of $\GL_m(\Z_p)$ on $\Gr(d, \Asym_m(\Z_p))$ as well as on $\Gr(\le d, \Asym_m(\Z_p))$. For any subspace $V\subset\Asym_m(\Z_p)$, we denote by $[V]$ the orbit of $V$ under the congruence action, i.e., $[V]=\{PVP^T\mid P\in \GL_m(\Z_p)\}$. We use $\Gr(d, \Asym_m(\Z_p))/\GL_m(\Z_p)$ and $\Gr(\le d, \Asym_m(\Z_p))/\GL_m(\Z_p)$ to denote the sets of orbits respectively.

Combing Proposition \ref{prop-mr-eq} and Corollary \ref{cor-mr-ta-bc}, we have the following observation, which plays a crucial role in our classification on groups of exponent $p$.

\begin{theorem}\label{thm-data-1to1-grorbits} Let $m, n\ge 1$ and $0\le d\le n$ be integers. Then there exist one-to-one correspondences between:
\begin{enumerate}
\item $\mathcal E_p(d;m,n)$ and $\Gr(d, \Asym_m(\Z_p))/\GL_m(\Z_p)$;
\item $\mathcal E_p(m,n)$ and $\Gr(\le n, \Asym_m(\Z_p))/\GL_m(\Z_p)$;
\item $\Gr(n, \Asym_m(\Z_p))/\GL_m(\Z_p)$ and $\mathcal G_2(p;m,n)$.
\end{enumerate}
\end{theorem}

\begin{proof} We need only to prove (1), and the rest follows from (1) and Lemma \ref{lem-class2todata} (3).

 Given $[\D]\in \mathcal E_p(d;m,n)$, where $\D=(H, A, 1, \varphi)$ is a $p$-elementary datum of type $(m,n)$ with $H$ acts on $A$ trivially. By assumption, $H\cong \Z_p^m$, $A\cong \Z_p^n$, and $\varphi$ is defined by (\ref{formula-cochain}) for some collection $(\varphi_{rs})_{1\le r\le s\le m}$ of elements in $A$ with $\varphi_{rr}=0$ for $1\le r\le m$.

Choose a basis $\B_H$ of $H$ and a basis $\B_A$ of $A$. Then the matrix presentation of $\D$ is $$\M(\D; \B_H, \B_A) = (I_n,\cdots, I_n; \Phi^{(1)}, \cdots, \Phi^{(n)}),$$ where $\Phi^{(i)}$, $i=1, 2, \cdots, n,$ are anti-symmetric $m\times m$ matrices. Set $V$ to be the linear subspace of $\Asym_m(\Z_p)$ spanned by $\Phi^{(1)}, \cdots, \Phi^{(n)}$. Then we can define a map
\[\Theta\colon \mathcal E_p(d;m,n)\to \Gr(d, \Asym_m(\Z_p))/\GL_m(\Z_p)\]
by sending $[\D]$ to $[V]$.

First we claim that $\Theta$ is well-defined, that is, $[V]$ is independent on the choice of $\B_H$ and $\B_H$. In fact, by Corollary \ref{cor-mr-ta-bc}(2), $V$ is independent on the choice of $B_A$; and if we take another basis $\B'_H$ of $H$ and let $V'$ be the corresponding subspace of $\Asym_m(\Z_p)$, again by Corollary \ref{cor-mr-ta-bc}, $V'$ is congruent to $V$, i.e., $[V]=[V']$.

Next we show that $\Theta$ is surjective. Let $V$ be a subspace of $\Asym_m(\Z_p)$ of dimension $d$. Take $n$ matrices $\Phi^{(1)}, \cdots, \Phi^{(n)}\in V$ which span $V$. Let $H=\Z_p^m$ and $A=\Z_p^n$ and let $\B_H$ and $\B_A$ be the standard bases. Consider the trivial action of $H$ on $A$. Clearly $\Phi^{(1)}, \cdots, \Phi^{(n)}\in V$ will define a 2-cocycle $\varphi\colon H\times H\to A$. Thus we obtain a datum $\D$, which is a preimage of $[V]$.

We are left to show that $\Theta$ is injective.

For this it suffices to show that if $\M(\D; \B_H, \B_A) = \M(\tilde \D; \B_{\tilde H}, \B_{\tilde A)}$, then $\D\sim \tilde \D$. Assume $\B_H=(h_1, \cdots, h_r)$, $\B_A=(a_1,\cdots, a_r)$, $\B_{\tilde H}=(\tilde h_1, \cdots, \tilde h_r)$ and $\B_{\tilde A} =(\tilde a_1,\cdots, \tilde a_r)$. Then we may define $\sigma_H\colon H\to \tilde H$ by setting $\sigma_H(h_r)= \tilde h_r$ for $1\le r\le m$, and define $\sigma_A\colon A\to \tilde A$ by setting $\sigma_A(a_i) =\tilde a_i$ for any $1\le i\le n$. Now it is direct to show that $(\sigma_H, \sigma_A, 0)$ defines an equivalence from $\D$ to $\tilde\D$, which completes the proof.
\end{proof}

For a $p$-elementary data of type $(m, 1)$, the group action is automatically trivial, and hence the above theorem applies.

\begin{theorem}\label{thm-1dim-dergp}(1) Set $W_{m,k}$ to be the one dimensional subspace of $\Asym_m(\Z_p)$ consisting of anti-symmetric block diagonal matrices of the form \[\diag\left(\underbrace{\begin{pmatrix}
                                          0 & a \\
                                          -a & 0 \\
                                        \end{pmatrix},
                                        \cdots,
                                         \begin{pmatrix}
                                           0 & a \\
                                           -a & 0 \\
                                         \end{pmatrix}}_k,
                                         \underbrace{0, 0, \cdots, 0}_{m-2k}
\right). \]
Then $\{W_{m,k}\mid 0\le k\le \dfrac m 2\}$ is a complete set of representatives of $\Gr(\le 1, \Asym_m(\Z_p))/\GL_m(\Z_p)$.

(2) Let $G$ be a group of order $p^{m+1}$ and exponent $p$. Assume that $G'\cong Z_p$. Then there exists a unique $1\le k\le \dfrac m 2$, such that $G$ is isomorphic to the group
\[G_{m,k}=\left\langle {a, b_1, \cdots, b_{2k},\atop c_1, \cdots, c_{m-2k}}\left| {{a^p=b_i^p=c_j^p=[a, b_i]=[a, c_j]=[b_i, c_j]=[c_i,c_j]=1, \forall i, j }\atop {[b_{2i-1}, b_{2i}]= a\ \forall i, [b_{2i-1}, b_j]=[b_{2i}, b_j]=1, \forall j > 2i}}\right. \right\rangle.
\]
\end{theorem}

\begin{proof} (1) Let $V$ be a one dimensional subspace of $\Asym_m(\Z_p)$, and $X\in V$ a nonzero anti-symmetric matrix of rank $2k$. It is easy to show that $X$ is congruent to the matrix \[\diag\left(\underbrace{\begin{pmatrix}
                                          0 & 1 \\
                                          -1 & 0 \\
                                        \end{pmatrix},
                                        \cdots,
                                         \begin{pmatrix}
                                           0 & 1 \\
                                           -1 & 0 \\
                                         \end{pmatrix}}_k,
                                         \underbrace{0, 0, \cdots, 0}_{m-2k}
\right), \] and hence $V$ is equivalent to $W_{m,k}$. On the other hand side, congruent matrices have equal rank, hence $W_{m,k}$'s are not equivalent to each other and the assertion follows.

(2) is an easy consequence of Proposition \ref{prop-pres-from-mrep} and Theorem \ref{thm-data-1to1-grorbits}.
\end{proof}

\begin{remark}\label{rmk-1dim-dergp} $G_{m,k}\cong \mathbb{Z}_3^{m-2k}\times G_{2k, k}$ and $G_{2k, k}\cong\underbrace{M_p(1,1,1)*\cdots*M_p(1,1,1)}_k$, where $M_p(1,1,1)$ is the unique
non-abelian group of order $p^3$ and exponent $p$, and $*$ denotes the (iterated) central product.
\end{remark}

%

\section{Two dimensional subspaces of anti-symmetric matrices}

In this section, $\k$ is a fixed base field. We will give a description of congruence classes of 2 dimensional subspaces of anti-symmetric matrices of order $m$ over $\k$. Consequently, we obtain a description of $\mathcal G_2(p;m,2)$, the isoclasses of class 2 groups of exponent $p$ and order $p^{m+2}$ and with derived subgroup isomorphic to $\Z_p^{2}$. We also calculate detailed examples in the case $m=3$ for arbitrary  base field $\k$, and in the case $m=4,5,6$ for $\k=\Z_3$.

Recall that in \cite{vish}, Vishnevetskii gave a classification of ``indecomposable" class 2 groups of prime exponent with derived subgroup (commutator subgroup) of rank 2, here a group is indecomposable means that it is not a central product of proper subgroups. Moreover, he showed that any class 2 group of prime exponent with derived subgroup of rank two is a central product of ``indecomposable" ones in a unique way. However, it was not shown there when two central products are isomorphic. Our description on $\mathcal G_2(p;m,2)$ recovers and strengthens to some extend the results by Vishnevetskii. We stress that the notion ``indecomposable group" used by Vishnevetskii is different from the one we used in the present paper, which refers to those groups that are not direct product of proper subgroups, cf. Section 7.

We heavily depends on the theory of so-called ``pencil", which gives a standard form for a pair of anti-symmetric matrices under the congruence. For more details we refer to \cite{gan}, \cite{schar} and \cite{vish}.

\subsection{Canonical form of a pair of anti-symmetric matrices.}\

 We first introduce some notations. For any $k\le 1$, let $L_k = (I_k, 0)^T, R_k=(0, I_k)^T$ be $(k+1)\times k$ matrices over $\k$, $N_k= \begin{pmatrix}
                                                0 &  &    &   \\
                                                1  & 0 &  &   \\
                                                  & \ddots  & \ddots & \\
                                                  &   &   1& 0 \\
                                              \end{pmatrix}
$ be the Jordan block of size $k$, where $I_k$ is the identity matrix of order $k$. For any monic polynomial $f= x^k+ a_{k-1} x^{k-1} + \cdots  + a_0\in \k[x]$, the companion matrix of characteristic polynomial $f(x)$ is defined to be \[C(f) = \begin{pmatrix}
                                                                                                   0 & 1 & 0 &\cdots & 0\\
                                                                                                   0 & 0 & 1 &\cdots & 0 \\
                                                                                                   0 & 0 & 0 &\ddots & 0 \\
                                                                                                   \vdots &\vdots  &\vdots &\ddots  & 1 \\
                                                                                                   -a_0 & -a_1 & -a_2 &\cdots  & -a_{k-1} \\
                                                                                                 \end{pmatrix}.
\]

Let $f(x,y)= a_kx^k+ a_{k-1}x^{k-1}y+\cdots + a_0 y^{k}\in \k[x,y]$ be a homogeneous polynomial of degree $k\ge 1$ which is a power of an irreducible one. Then $f=a_0y^k, a_0\ne 0$ or $a_k\ne 0$, and in the latter case $\frac 1 {a_k} f(x,-1)\in \k[x]$ is a monic polynomial of degree $k$. We set $L(f)= N_k$ and $R(f)= I_k$ if $f= a_0 y^k$; and $L(f)= I_k$ and $R(f)= C(\frac{1} {a_k} f(x, -1))$ else. For two polynomials $f_1, f_2\in \k[x,y]$, we say $f_1\sim f_2$ if $f_1=\lambda f_2$ for some $0\neq\lambda\in\k$.
Clearly, $L(f_1) = L(f_1)$ and $R(f_1)= R(f_2)$ if and only if $f_1\sim f_2$.

We use $\Irr(\k[x,y])$ to denote the set of irreducible homogenous polynomials in two variables $x, y$, and $\Irr(\k[x,y]_k)$ the subset of the ones of degree $k$ for each $k\ge1$.

\begin{lemma} The set of irreducible homogenous polynomials in $\Z_3$ of degree $\le3$ are listed as follows:

\noindent(1) $\Irr(\Z_3[x,y]_1) = \{\pm x, \pm y\}$;

\noindent(2) $\Irr(\Z_3[x,y]_2) = \{\pm(x^2+y^2), \pm(x^2+xy-y^2), \pm(x^2-xy-y^2)\}$;

\noindent(3) $\Irr(\Z_3[x,y]_3) = \left\{\pm(x^3+x^2y+xy^2-y^3),\ \pm(x^3+x^2y-xy^2+y^3),\ \pm(x^3+x^2y-y^3),\ \pm(x^3-x^2y+y^3),\atop \pm(x^3-x^2y-xy^2-y^3),\ \pm(x^3-x^2y+xy^2+y^3),\ \pm(x^3-xy^2+y^3),\ \pm(x^3-xy^2-y^3)\right\}$
\end{lemma}

Let $A, B\in \Asym_m(\k)$ be a pair of anti-symmetric matrices. We say two pairs $(A, B)$ and $(\tilde A,\tilde B)$ are congruently equivalent, or congruent for short, if
there exists an invertible matrix $P\in M_m(\k)$ such that $\tilde A = PAP^T$ and $\tilde B = PBP^T$.

\begin{definition}  Let $k_1\ge \cdots\ge k_r\ge 1, d_1, \cdots, d_s \ge 1$ be integers, and let $f_1, f_2, \cdots, f_s\in\Irr(\k[x,y])$ be irreducible homogeneous polynomials. The pair of $m\times m$ anti-symmetric matrices
$$\left(\begin{pmatrix}0 &C_1 &0\\ -C_1^T &0 &0 \\ 0&0&0\end{pmatrix}, \begin{pmatrix}0 &C_2 &0\\ -C_2^T &0 &0 \\ 0&0&0\end{pmatrix}\right)$$
is called a \emph{canonical form} of type  $(m; k_1,\cdots, k_r; f_1^{d_1}, \cdots, f_s^{d_s})$, where $$C_1= \diag(L_{k_1}, \cdots, L_{k_r}, L(f_1^{d_1}), \cdots, L(f_s^{d_s})),$$
$$C_2= \diag(R_{k_1}, \cdots, R_{k_r}, R(f_1^{d_1}), \cdots, R(f_s^{d_s}))$$ are block diagonal matrices.
\end{definition}

Clearly, $m\ge 2\sum_{i=1}^r k_i + r + 2\sum_{j=1}^s \deg(f_j)d_j$.
The following result says that any pair of anti-symmetric matrices is congruent to a unique canonical form.

\begin{theorem} \cite[9.1]{schar} Let $(A, B)$ be a pair of $m\times m$ anti-symmetric matrices. Then there exist integers $k_1\ge \cdots\ge k_r\ge 1, d_1, \cdots, d_s \ge 1$, and $f_1, f_2, \cdots, f_s\in\Irr(\k[x,y])$, such that $(A,B)$ is congruent to the canonical form  of type $(m; k_1,\cdots, k_r; f_1^{d_1}, \cdots, f_s^{d_s})$.

If $(A, B) $ is congruent to another canonical form of type $(m; k'_1, \cdots k'_{r'};  {f'_1}^{d'_1}, \cdots, {f'_{s'}}^{d'_{s'}})$, then $r=r'$, $s= s'$, and $k'_i=k_i$ for $1\le i\le r$, and there exists some permutation $\sigma\in S_{\{1, 2,\cdots, s\}}$, such that $d_j= d'_{\sigma(j)}$ and $f_j\sim f'_{\sigma(j)}$ for $1\le j\le s$.
\end{theorem}

Consider the natural action of $\GL_2(\k)$ on $\k[x, y]$. Precisely, for any $P=\begin{pmatrix}
                                                                      a & b \\
                                                                      c & d \\
                                                                    \end{pmatrix}
$ and any $f(x,y)\in \k[x,y]$, $(P\cdot f)(x, y) = f(ax+by, cx +dy)$. Clearly $f(x,y)$ is irreducible if and only if $P\cdot f$ is. Thus $\GL_2(k)$ acts on $\Irr(\k[x,y])$ and on each $\Irr(\k[x,y]_k)$ as well.

\begin{remark} One can check that $\GL_2(\Z_3)$ acts transitively on the set $\Irr(\Z_3[x,y]_k)/\{\pm  1\}$ for $k=1, 2, 3$.
\end{remark}

\begin{definition}Two types $(m; k_1,\cdots, k_r; f_1^{d_1}, \cdots, f_s^{d_s})$ and $ (m; k'_1, \cdots k'_{r'};  {f'_1}^{d'_1}, \cdots, {f'_{s'}}^{d'_{s'}})$ are said to be \emph{equivalent} if $r=r'$, $s= s'$, $k'_i=k_i$ for $1\le i\le r$, and there exists some permutation $\sigma\in S_{\{1, 2,\cdots, s\}}$ and some $P\in \GL_2(k)$, such that $d_j= d'_{\sigma(j)}$ and $f_j\sim P\cdot f'_{\sigma(j)}$ for $1\le j\le s$.
\end{definition}

Applying the above theorem, we obtain a description of congruence classes of two dimensional subspaces of anti-symmetric matrices.

\begin{theorem}\label{thm-2dim-dergp} There exists a bijection between $\Gr(2, \Asym_m(\k))/\GL_m(\k)$ and the set of equivalence classes of canonical forms. Consequently, $\mathcal G_2(p; m, 2)$ is in one-to-one correspondence with the equivalence classes of canonical forms over $\Z_p$.
\end{theorem}

\begin{proof} Let $V\in \Gr(2, \Asym_m(\k))$, and $A, B\in V$ be a basis of $V$. Then $(A,B)$ is congruent to some canonical form $(m; k_1,\cdots, k_r; f_1^{d_1}, \cdots, f_s^{d_s})$, and $V$ is congruent to some space spanned by the canonical forms determined by $(A, B)$. Let $(\tilde A, \tilde B)$ be another basis of $V$, then there exists some  $P=\begin{pmatrix}
                                                                      a & b \\
                                                                      c & d \\
                                                                    \end{pmatrix}
\in \GL_2(\k)$, such that $\tilde A= aA+ bB$ and $\tilde B= cA +dB$.
We need to determine the canonical form of $(\tilde A, \tilde B)$, and it suffices to consider the case
\[A(k) = \begin{pmatrix}
    0 & L_{k} \\
    -L_{k}^T & 0 \\
  \end{pmatrix},
 B(k) = \begin{pmatrix}
    0 & R_{k} \\
    -R_{k}^T & 0 \\
  \end{pmatrix}\]
for some positive integer $k\ge 0$, and the case
\[A(f^d)= \begin{pmatrix}
    0 & L(f^{d}) \\
    -L(f^{d})^T & 0 \\
  \end{pmatrix},
 B(f^d)= \begin{pmatrix}
    0 & R(f^{d}) \\
    -R(f^{d})^T & 0 \\
  \end{pmatrix}\]
for some irreducible homogeneous polynomial $f\in\k[x,y]$ and some positive integer $d$.

It is direct to check that the canonical form of $(aA(k)+bB(k), cA(k)+dB(k))$ is $(A(k), B(k))$,
and the one of $(aA(f^d) + bB(f^d), cA(f^d) + dB(f^d))$ is $(A((P\cdot f)^d), B((P\cdot f)^d))$.
The conclusion follows.
\end{proof}

\subsection{m=3.} \label{sec-2of3}\ In this case,  $\k$ can be an arbitrary field.

Consider the type $(3; k_1,\cdots, k_r; f_1^{d_1}, \cdots, f_s^{d_s})$. Since $2\sum_{i=1}^r k_i + r + 2\sum_{j=1}^s \deg(f_j)d_j\le3$,
we have $r=1, k_1=1$ and $s=0$. Thus there exists only one congruence class of subspaces of $\Asym_3(\k)$, say
\[\begin{pmatrix}&&a\\&&b\\-a&-b&\end{pmatrix}=\left\{\left.a\begin{pmatrix}
       & &1   \\
     &   0&   \\
      -1&  & \\
   \end{pmatrix}
+ b\begin{pmatrix}
      0 &  & \\
        &   &1   \\
&  -1&  \\
   \end{pmatrix} \right| a, b\in \k\right\}.
\]

In the rest of this section, we assume $\k=\Z_p$. We will give a full list of representatives for
congruence classes of 2 dimensional subspaces of $\Asym_m(\Z_3)$ for $m=4, 5$, and $6$.

\subsection{m=4.}\label{sec-2of4}
Consider the type $(4; k_1,\cdots, k_r; f_1^{d_1}, \cdots, f_s^{d_s})$.

Since $2\sum_{i=1}^r k_i + r + 2\sum_{j=1}^s \deg(f_j)d_j\le4$,
there are 4 possible cases:
\begin{enumerate}
\item[(1)] $r=1$, $k_1=1$, $s=0$;
\item[(2)] $r=0$, $s=2$, $d_1= d_2 =1$, and $\deg(f_1)=\deg(f_2)=1$;
\item[(3)] $r=0$, $s=1$, $d_1=2$, and $\deg(f_1)=1$;
\item[(4)] $r=0$, $s=1$, $d_1=1$, and $\deg(f_1)=2$.
\end{enumerate} The corresponding 2-dimensional subspace of  $\Asym_4(\Z_3)$ are listed as follows.
$$\left(\begin{array}{cccc}&&a&\\&&b&\\-a&-b&&\\&&&\end{array}\right),
\left(\begin{array}{cccc}&&a&\\&&&b\\-a&&&\\&-b&&\end{array}\right),
\left(\begin{array}{cccc}&&a&b\\&&&a\\-a&&&\\-b&-a&&\end{array}\right),
\left(\begin{array}{cccc}&&a&b\\&&-b&a\\-a&b&&\\-b&-a&&\end{array}\right).$$

We mention that in Case (2), $f_1\ne \pm f_2$, otherwise it gives a one dimensional subspace.  In Case (3), we need only to take $f_1(x, y)= x$, since $\GL_2(\Z_3)$ acts transitively on the set $\Irr(\Z_3[x,y]_1)$. In Case (4), we need only to take $f(x, y)= x^2+y^2$, since $\GL_2(\Z_3)$ acts transitively on the set $\Irr(\Z_3[x,y]_2)/\{\pm1\}$.

\subsection{m=5.}\label{sec-2of5}

We claim that there are six congruence classes of 2-dimensional subspaces of $\Asym_5(\Z_3)$ which are listed as follows.
$$\begin{pmatrix}&&a&&\\&&b&&\\-a&-b&&&\\&&&0&\\&&&&0\end{pmatrix},
 \begin{pmatrix}&&a&&\\&&&b&\\-a&&&&\\&-b&&&\\&&&&0\end{pmatrix},
 \begin{pmatrix}&&a&b&\\&&&a&\\-a&&&&\\-b&-a&&&\\&&&&0\end{pmatrix},$$
$$\begin{pmatrix}&&a&b&\\&&-b&a&\\-a&b&&&\\-b&-a&&&\\&&&&0\end{pmatrix},
 \begin{pmatrix}&&&a&\\&&&b&\\&&&&a\\-a&-b&&&\\&&-a&&\end{pmatrix},
 \begin{pmatrix}&&&a&\\&&&b&a\\&&&& b\\-a&-b&&&\\&-a&-b&&\end{pmatrix}.$$
In fact, for any type $(5; k_1,\cdots, k_r; f_1^{d_1}, \cdots, f_s^{d_s})$, $2\sum_{i=1}^r k_i + r + 2\sum_{j=1}^s \deg(f_j)d_j\le5$. The cases that $2\sum_{i=1}^r k_i + r + 2\sum_{j=1}^s \deg(f_j)d_j \le 4$ exactly correspond to the first four subspaces in the above list.

Now we are left to consider the cases that the $2\sum_{i=1}^r k_i + r + 2\sum_{j=1}^s \deg(f_j)d_j = 5$. In this situation we have two subcases:
\begin{enumerate}
\item[(1)] $r=s=1$, $k_1=1$, $d_1=1$, $\deg(f_1)=1$, and

\item[(2)] $r=1, k_1=2, s=0$,
\end{enumerate}
which correspond to the last two subspaces in the above list respectively. Note that in Case (1), we may take $f_1= x$, for $\GL_2(\Z_3)$ acts transitively on $\Irr(\Z_3[x,y]_1)$.

\subsection{m=6.}\label{sec-2of6} \

Let $(6; k_1,\cdots, k_r; f_1^{d_1}, \cdots, f_s^{d_s})$ be a type. Then $2\sum_{i=1}^r k_i + r + 2\sum_{j=1}^s \deg(f_j)d_j\le 6$. As above, the case $2\sum_{i=1}^r k_i + r + 2\sum_{j=1}^s \deg(f_j)d_j\le 5$ are essentially obtained in last section, and the corresponding subspaces are listed as follows.
$$\begin{pmatrix}&&&a&&\\&&&b&&\\&&&&&\\-a&-b&&&&\\&&&&0&\\&&&&&0\end{pmatrix},
 \begin{pmatrix} &&a&&&\\&&&b&&\\-a&&&&&\\&-b&&&&\\&&&&0&\\&&&&&0\end{pmatrix},
 \begin{pmatrix} &&a&b&&\\&&&a&&\\-a&&&&&\\-b&-a&&&&\\&&&&0&\\&&&&&0\end{pmatrix},$$
$$\begin{pmatrix} &&a&b&&\\&&-b&a&&\\-a&b&&&&\\-b&-a&&&&\\&&&&0&\\&&&&&0\end{pmatrix},
  \begin{pmatrix} &&&a&&\\&&&b&&\\&&&&a&\\-a&-b&&&&\\&&-a&&&\\&&&&&0\end{pmatrix},
  \begin{pmatrix} &&&a&&\\&&&b&a&\\&&&& b&\\-a&-b&&&&\\&-a&-b&&&\\&&&&&0\end{pmatrix}.
$$

We are left to consider the case $2\sum_{i=1}^r k_i + r + 2\sum_{j=1}^s \deg(f_j)d_j= 6$.
Easy calculation shows that there are the following cases:

(a) $r=0$, $s=3$, $\deg(f_1)=\deg(f_2)=\deg(f_3)=d_1=d_2=d_3=1$. There are two subcases: (a.1) $ f_1\neq\pm f_2\neq\pm f_3$, in this case, we may choose $f_1= x$, $f_2=y$, and $f_3=x+y$, and all possible such choices are equivalent; and (a.2) $f_1 = \pm f_2\neq\pm f_3$, in this case, we may choose $f_1=f_2= x$, and $f_3=y$.

(b) $r=0$, $s=2$, $\deg(f_1)=1, d_1=2, \deg(f_2)= d_2=1$. There are two subcases: (b.1) $f_1=\pm f_2$, in this case we may choose $f_1=f_2=x$; and (b.2) $f_1\neq \pm f_2$, in this case we may choose $f_1=x$, $f_2=y$.

(c) $r=0$, $s=2$, $\deg(f_1)=2, d_1=1, \deg(f_2)= d_2=1$. In this case, we may choose $f_1=x^2+y^2$ and $f_2= x$.

(d) $r=0$, $s=1$, $\deg(f_1)=3, d_1=1$. We may choose $f_1= x^3-xy^2 +y^3$ in this case.

(e) $r=0$, $s=1$, $\deg(f_1)=1, d_1=3$. We may choose $f_1= x$ in this case.

(f) $r=2$, $s=0$, $k_1= k_2 =1$. This case is uniquely determined.

The corresponding subspaces are listed as follows.
\[(a.1): \begin{pmatrix}
           &  &    & a &   &   \\
           &  &    &   & b &   \\
           &  &    &   &   & a+b \\
         -a&  &    &   &   &   \\
           &-b&    &   &   &   \\
           &  &-a-b&   &   &   \\
       \end{pmatrix},\quad
  (a.2): \begin{pmatrix}
           &  &   & a &   &   \\
           &  &   &   & a &   \\
           &  &   &   &   & b \\
         -a&  &   &   &   &   \\
           &-a&   &   &   &   \\
           &  & -b&   &   &   \\
       \end{pmatrix},
\]
\[(b.1): \begin{pmatrix}
           &  &    & a & b &   \\
           &  &    &   & a &   \\
           &  &    &   &   & a \\
         -a&  &    &   &   &   \\
         -b&-a&    &   &   &   \\
           &  & -a &   &   &   \\
       \end{pmatrix},\quad
  (b.2): \begin{pmatrix}
           &  &   & a & b  &   \\
           &  &   &   & a &   \\
           &  &   &   &   & b \\
         -a&  &   &   &   &   \\
         -b&-a&   &   &   &   \\
           &  & -b&   &   &   \\
       \end{pmatrix}
\]
\[(c): \begin{pmatrix}
           &  &    & a & b &   \\
           &  &    &-b & a &   \\
           &  &    &   &   & a \\
         -a& b&    &   &   &   \\
         -b&-a&    &   &   &   \\
           &  & -a &   &   &   \\
       \end{pmatrix},\quad
  (d): \begin{pmatrix}
           &  &   & a & b &   \\
           &  &   &   & a & b  \\
           &  &   & b & b & a \\
         -a&  & -b&   &   &   \\
         -b&-a& -b&   &   &   \\
           &-b& -a&   &   &   \\
       \end{pmatrix}
\]
\[(e): \begin{pmatrix}
           &  &    & a & b &   \\
           &  &    &   & a & b \\
           &  &    &   &   & a \\
         -a&  &    &   &   &   \\
         -b&-a&    &   &   &   \\
           &-b& -a &   &   &   \\
       \end{pmatrix},\quad
  (f): \begin{pmatrix}
           &  &   &   & a &   \\
           &  &   &   & b &   \\
           &  &   &   &   & a \\
           &  &   &   &   & b \\
         -a&-b&   &   &   &   \\
           &  & -a& -b&   &   \\
       \end{pmatrix}
\]

In summary, $\Asym_6(\Z_3)$ has 14 congruence classes of two dimensional subspaces which are listed as above.

\section{Groups of order $3^3$, $3^4$, $3^5$, $3^6$, $3^7$ and exponent 3}

In this section we will give a full list of groups of exponent $3$ and order $\le 3^7$. From now on, $p=3$, and all groups considered are of exponent 3.

Let $G$ be a group of exponent 3. Set $A=G'$ and $H=G/G'$.
Assume $A=\Z_3^{n}$ and $H=\Z_3^m$. Then $G=\mathbb{Z}_3^{n}\rtimes_{\rho,f}\mathbb{Z}_3^m$, and $|G|=3^{m+n}$. We will deal with the cases $m+n\le 7$ in this section. By Corollary \ref{cor-expp3}, in case $m+n\le 6$, it suffices to consider the trivial action, where Theorem \ref{thm-data-1to1-grorbits} applies.
Since $\dim\Asym_m(\Z_3)=\frac{m(m-1)}{2}$, we only need to consider the case $n\le \frac{m(m-1)}{2}$.

We say that a group $G$ is \emph{indecomposable} if $G$ is not a product of proper subgroups, otherwise we say that $G$ is \emph{decomposable}. Clearly, for a nilpotent group, if $G$ is decomposable, then so is its center $Z(G)$.

\begin{example} \label{eg-dec-1dim-dergp}
It is easy to show that the group $G_{m,k}$ defined in Theorem \ref{thm-1dim-dergp} is indecomposable if and only if $m=2k$. In fact, by Remark \ref{rmk-1dim-dergp}, we need only to show that $G_{2k,k}$ is indecomposable, which is true for $Z(G_{2k,k})\cong \Z_p$.
\end{example}

Clearly to classify all groups of exponent $3$, it suffices to
classify the indecomposable ones. In fact, by Krull-Schmidt Theorem, any finite group is a unique product of indecomposable ones. We will make a full list of indecomposable groups of exponent $3$.

\subsection{$\mathbf{|G|=3}$, or $\mathbf{3^2}$} \ The following result is obvious.
\begin{proposition}  $\Z_3$ is the unique indecomposable group of order 3 up to isomorphism, and there exist no indecomposable groups of order $3^2$ and exponent 3.
\end{proposition}

For consistency of notations, we set $I_1= \Z_3$.

\subsection{$\mathbf{|G|=3^3}$}\quad In this case $n=0$, or $1$.

If $n=0$, then $G\cong \Z_3^3$ and hence is decomposable.

If $n=1$, then $m=2$ and
$\Asym_2(\Z_3)$ has a unique subspace of dimension $1$. We denote by $I_3$ the group given by the only element in $\mathcal E_3(1;2,1)$. It is well known that $I_3$ is the unique nonabelian  group of order 27 and exponent 3. By Proposition \ref{prop-pres-from-mrep}, \[I_3 = \langle a, x, y\mid a^3=x^3=y^3=[x,a]=[y,a]=1, [x,y]=a\rangle.\] Thus we have shown the following result.

\begin{proposition} $I_3$ is the unique indecomposable group of order $3^3$ and exponent $3$ up to isomorphism.
\end{proposition}

\subsection{$\mathbf{|G|=3^4}$} \quad  In this case $n=0$ or $1$, for $n\le \frac{m(m-1)}{2}$.

If $n=0$, then $G$ is abelian and hence decomposable. If $n=1$, then by Theorem \ref{thm-1dim-dergp}, we know that $G$ is isomorphic to $G_{3,1}$ which is decomposable by Example \ref{eg-dec-1dim-dergp}.

\begin{proposition} There exist no indecomposable groups of exponent 3 and order $3^4$.
\end{proposition}

\subsection{$\mathbf{|G|=3^5}$}\quad In this case, $n=0, 1$ or $2$.

If $n=0$, then $G\cong \Z_3^5$ and hence is decomposable.

If $n=1$, then $G$ is isomorphic to one of the groups $G_{4,1}$ or $G_{4,2}$ in Theorem \ref{thm-1dim-dergp} (2). By Example \ref{eg-dec-1dim-dergp}, only
$G_{4,2}$ is indecomposable. We denote $G_{4,2}$ by $I_{5.1}$.

Now assume $n=2$. Then there is only one congruence class of 2 dimensional subspaces of $\Asym_3(\Z_3)$, say
$\left(\begin{array}{ccc}&a&b\\-a&&\\-b&&\end{array}\right)$. The corresponding group, denoted by $I_{5.2}$, is
\[I_{5.2} = \langle a_1,a_2, h_1,h_2,h_3\mid a_i^3, h_j^3, [a_i, h_j], [a_1,a_2], [h_2, h_3], [h_1,h_2]= a_1, [h_1, h_3]=a_2
\rangle.\]

\begin{proposition} $I_{5.1}$ and $I_{5.2}$ give a complete set of isoclasses of irreducible groups of exponent 3 and order $3^5$.
\end{proposition}

\subsection{$\mathbf{|G|=3^6}$}\quad In this case, $n=0, 1, 2, 3$.

If $n=0$, then $G=\mathbb{Z}_3^6$ and hence is decomposable.

If $n=1$, then $m=5$ and $G = G_{m,1}$ or $G_{m,2}$, and $G$ is decomposable in either case.

If $n=2$, then $m=4$. By Section \ref{sec-2of4} there are 4 congruence classes of 2 dimensional subspaces of the space of $\Asym_4(\Z_3)$, say
$$\left(\begin{array}{cccc}&&a&\\&&b&\\-a&-b&&\\&&&0\end{array}\right),
\left(\begin{array}{cccc}&&a&\\&&&b\\-a&&&\\&-b&&\end{array}\right),
\left(\begin{array}{cccc}&&a&b\\&&&a\\-a&&&\\-b&-a&&\end{array}\right),
\left(\begin{array}{cccc}&&a&b\\&&-b&a\\-a&b&&\\-b&-a&&\end{array}\right).$$
Let $G_1, G_2, G_3, G_4$ be the corresponding groups. It is easy to show that $G_1\cong I_{5.2}\times \Z_3$ and $G_2\cong I_3\times I_3$, and $G_3$ and $G_4$ are both indecomposable. We set
\[I_{6.1}=G_3=\left\langle a_1, a_2, h_1, h_2, h_3, h_4\left| {a_i^3, h_j^3, [a_i, h_j], [a_i,a_j], [h_1, h_2], [h_3, h_4], [h_2, h_3], \atop [h_1, h_3]=[h_2, h_4] = a_1, [h_1, h_4]=a_2}\right.
\right\rangle,\]
\[I_{6.2}= G_4\left\langle a_1, a_2, h_1, h_2, h_3, h_4\left| {a_i^3, h_j^3, [a_i, h_j], [a_i,a_j], [h_1, h_2], [h_3, h_4], \atop [h_1, h_3]=[h_2, h_4] = a_1, [h_1, h_4]=[h_3, h_2]=a_2}\right.
\right\rangle.\]

Now we assume $n=3$, then $m=3$. There is only one congruence class of 3-dimensional subspaces of $\Asym_3(\Z_3)$, say
$\left(\begin{array}{ccc}&a&b\\-a&&c\\-b&-c&\end{array}\right)$. Denote  the corresponding group by $I_{6.3}$. By Proposition \ref{prop-pres-from-mrep},
\[I_{6.3}=\left\langle\left.{ a_1, a_2, a_3,\atop h_1,h_2,h_3}\right | {a_i^3, h_j^3, [a_i, h_j], [a_i,a_j], [h_1,h_2]= a_1,\atop [h_1, h_3]=a_2, [h_2, h_3]=a_3}
\right\rangle.\]

\begin{proposition} $I_{6.1}$, $I_{6.2}$ and $I_{6.3}$ give a complete set of isoclasses of irreducible groups of exponent 3 and order $3^6$.
\end{proposition}

\subsection{$\mathbf{|G|=3^7}$}{\ } Then $n =0, 1, 2, 3, 4$.

In fact, if $n\ge 5$, then $m\le 2$, and
by Corollary \ref{cor-expp3}, $H$ acts on $A$ trivially, this case will not
happen since $\dim(\Asym_m(\Z_p))< n$.

If $n=0$, then $G=\mathbb{Z}_3^7$, and hence is decomposable.

If $n=1$, then $m=6$, and $G= G_{6,1}, G_{6,2}$ or $G_{6,3}$ as in Theorem \ref{thm-1dim-dergp} (2). By Example \ref{eg-dec-1dim-dergp},
only $G_{6,3}$ is indecomposable, and we denote it by $I_{7.1}$.

If $n=2$, then $m=5$. By Section \ref{sec-2of5}, there are 6 congruence classes of 2 dimensional subspaces of the space $\Asym_5(\Z_3)$:
$$\begin{pmatrix}&&a&&\\&&b&&\\-a&-b&&&\\&&&0&\\&&&&0\end{pmatrix},\quad
 \begin{pmatrix}&&a&&\\&&&b&\\-a&&&&\\&-b&&&\\&&&&0\end{pmatrix},\quad \begin{pmatrix}&&a&b&\\&&&a&\\-a&&&&\\-b&-a&&&\\&&&&0\end{pmatrix},$$
$$\begin{pmatrix}&&a&b&\\&&-b&a&\\-a&b&&&\\-b&-a&&&\\&&&&0\end{pmatrix},
 \begin{pmatrix}&&&a&\\&&&b&\\&&&&a\\-a&-b&&&\\&&-a&&\end{pmatrix},
 \begin{pmatrix}&&&a&\\&&&b&a\\&&&& b\\-a&-b&&&\\&-a&-b&&\end{pmatrix}.$$
Let $G_i, 1\le i\le 6$ be the corresponding groups. Then one checks that $G_1\cong I_{5.2}\times \Z_3^2$, $G_2\cong I_3\times I_3\times Z_3$, $G_3\cong I_{6.1}\times \Z_3$, $G_4\cong I_{6.2}\times \Z_3$, and $G_5$ and $G_6$ are indecomposable. Denote $G_5$ and $G_6$ by $I_{7.2}$ and $I_{7.3}$ respectively. Then
\[I_{7.2}=\left\langle {a_1, a_2, h_1, h_2,\atop h_3, h_4, h_5}\left| {a_i^3, h_j^3, [a_i, h_j], [a_i,a_j], [h_1, h_2], [h_1, h_5], [h_2, h_3], [h_2, h_4], [h_3, h_4],\atop [h_3, h_5], [h_4, h_5], [h_1, h_3]=[h_2, h_5] = a_1, [h_1, h_4]=a_2}\right.
\right\rangle,\]
\[I_{7.3}=\left\langle {a_1, a_2, h_1, h_2,\atop h_3, h_4, h_5}\left| {a_i^3, h_j^3, [a_i, h_j], [a_i,a_j], [h_1, h_2], [h_1, h_5], [h_2, h_3], [h_3, h_4], [h_3, h_5],\atop [h_4, h_5], [h_1, h_3]=[h_2, h_4] = a_1,[h_1, h_4]=[h_2, h_5]=a_2}\right.
\right\rangle.\]

If $n=3$, then $m=4$. By Theorem \ref{thm-3of4}, there are 6 congruence classes of 3-dimensional subspaces of the space of $\Asym_4(\Z_3)$:
$$\begin{pmatrix}&c&a&\\-c&&b&\\-a&-b&&\\&&&\end{pmatrix},
\begin{pmatrix}&&a&\\&&b&\\-a&-b&&c\\&&-c&\end{pmatrix},
\begin{pmatrix}&c&a&\\-c&&b&\\-a&-b&&c\\&&-c&\end{pmatrix},$$
$$\begin{pmatrix}&&a&c\\&&b&\\-a&-b&&\\-c&&&\end{pmatrix},
\begin{pmatrix}&c&a&\\-c&&&b\\-a&&&c\\&-b&-c&\end{pmatrix},
\begin{pmatrix}&c&a&b\\-c&&&a\\-a&&&-c\\-b&-a&c&\end{pmatrix}.$$
The first one corresponds to the group $I_{6.3}\times \Z_3$, and the others are indecomposable and correspond to the following groups respectively:
\[I_{7.4}=\left\langle {a_1, a_2, a_3,\atop h_1, h_2, h_3, h_4}\left| {a_i^3, h_j^3, [a_i, h_j], [a_i,a_j], [h_1, h_2], [h_1, h_4], [h_2, h_4], \atop [h_1, h_3]=a_1, [h_2, h_3] = a_2, [h_3, h_4]=a_3}\right.
\right\rangle,\]

\[I_{7.5}=\left\langle {a_1, a_2, a_3,\atop h_1, h_2, h_3, h_4}\left| {a_i^3, h_j^3, [a_i, h_j], [a_i,a_j],  [h_1, h_4], [h_2, h_4], \atop [h_1, h_2]= [h_3, h_4]=a_1, [h_1, h_3]=a_2, [h_2, h_3] = a_3}\right.
\right\rangle,\]

\[I_{7.6}=\left\langle {a_1, a_2, a_3,\atop h_1, h_2, h_3, h_4}\left| {a_i^3, h_j^3, [a_i, h_j], [a_i,a_j], [h_1, h_2], [h_2, h_4], [h_3, h_4], \atop [h_1, h_3]=a_1, [h_1, h_4] = a_3, [h_2, h_3]=a_2}\right.
\right\rangle,\]

\[I_{7.7}=\left\langle {a_1, a_2, a_3,\atop h_1, h_2, h_3, h_4}\left| {a_i^3, h_j^3, [a_i, h_j], [a_i,a_j], [h_1, h_4], [h_2, h_3], \atop [h_1, h_2]=[h_3, h_4]=a_1, [h_1, h_3] = a_2, [h_2, h_4]=a_3}\right.
\right\rangle,\]

\[I_{7.8}=\left\langle {a_1, a_2, a_3,\atop h_1, h_2, h_3, h_4}\left| {a_i^3, h_j^3, [a_i, h_j], [a_i,a_j], [h_2, h_3], [h_1, h_2]=[h_4, h_3]=a_1\atop, [h_1, h_3] = [h_2, h_4]= a_2, [h_1, h_4]=a_3}\right.
\right\rangle.\]

Now assume $n=4$ and $m=3$.

Then the action $\rho$ of $H$ on $A$ must be nontrivial. Otherwise, $A$ can not be the derived subgroup since $\dim(\Asym_3(\Z_3)) = 3 < 4$. Let $h_1$, $h_2$ and $h_3$ be a basis of $H$. Then by Corollary \ref{cor-expp3}, $\varphi_{12}$, $\varphi_{13}$ and $\varphi_{23}$ are linearly independent over $A/\mathrm{soc}(A)$. It forces  $\dim(\mathrm{soc}(A))=1$.
We set $a_1=\varphi_{23}$, $a_2 = -\varphi_{13}$, $a_3 = \varphi_{12}$,
 and
$$a_4= h_1\tr\varphi_{23} -\varphi_{23} =\varphi_{13} - h_2\tr\varphi_{13}= h_3\tr\varphi_{12} -\varphi_{12}.$$ Now under the basis $a_1, a_2, a_3, a_4$, the action $\rho$ of $H$ on $A$ is given by
$$\rho(h_1)=\left(\begin{array}{cccc}1&&&\\&1&&\\&&1&\\1&&&1\end{array}\right),\
\rho(h_2)=\left(\begin{array}{cccc}1&&&\\&1&&\\&&1&\\&1&&1\end{array}\right),\
\rho(h_3)=\left(\begin{array}{cccc}1&&&\\&1&&\\&&1&\\&&1&1\end{array}\right).$$
The corresponding group, denoted by $I_{7.9}$, is indecomposable and has presentation
\[I_{7.9}=\left\langle {a_1, a_2, a_3, a_4\atop h_1, h_2, h_3}\left| {a_i^3, h_j^3, [h_i, a_j] \ (i\ne j) , [h_i, a_i]=a_4\ (i=1,2,3),\atop [h_1, h_2]=a_3, [h_1, h_3]=-a_2, [h_2,h_3]=a_1}\right.
\right\rangle.\]

Thus we have shown the following propostion.

\begin{proposition} $I_{7.1}, I_{7.2}, \cdots, I_{7.9}$ give a complete set of isoclasses of irreducible groups of exponent 3 and order $3^7$.
\end{proposition}

\subsection{ } We summarize the classification results in this section as follows.

\begin{theorem}\label{thm-ord1to7} \cite{wil} The isoclasses of groups of order $3^{k}$($k\le 7$) and of exponent 3 are listed as follows,  where $n$ is the dimension of the derived subgroup.
\begin{center}
{
\begin{tabular}{|c|c|c|c|}
\hline
 $k$ & $n$& Decomposable & Indecomposable \\ \hline
 $1$   & $0$ & & $I_1=\Z_3$\\ \hline
 $2$   & $0$ &$I_1^2$ &\\ \hline
 $3$   & $0$ &$I_1^3$ &\\ \hline
 $3$   & $1$ & & $I_3$\\ \hline
 $4$   & $0$ &$I_1^4$ & \\ \hline
     & 1 &$I_1\times I_3$ &  \\ \hline
 $5$   & $0$ &$I_1^5$ & \\ \hline
     & $1$ &$I_1^2\times I_3$ & $I_{5.1}$\\ \hline
     & $2$ & & $I_{5.2}$\\ \hline
 $6$   & $0$ &$I_1^6$ & \\ \hline
     & $1$ &$I_1^3\times I_3, I_1\times I_{5.1}$ & \\ \hline
     & $2$ &$I_1\times I_{5.2}, I_3\times I_3$ & $I_{6.1}, I_{6.2}$\\ \hline
     & $3$ & & $I_{6.3}$\\ \hline
 $7$   & $0$ &$I_1^7$ & \\ \hline
     & $1$ &$I_1^4\times I_3, I_1^2\times I_{5.1}$ & $I_{7.1}$\\ \hline
     & $2$ &$I_1^2\times I_{5.2}, I_1\times I_3\times I_3, I_1\times I_{6.1}, I_1\times I_{6.2}$ & $I_{7.2}, I_{7.3}$ \\ \hline
     & $3$ &$I_1\times I_{6.3}$ & $I_{7.4}, I_{7.5}, I_{7.6}, I_{7.7}, I_{7.8}$\\ \hline
    & $4$ & & $I_{7.9}$\\ \hline
   \end{tabular}

   \footnotesize Table 1. Isoclasses of groups of order $3^k$ ($k\le 7$) and exponent 3}
   \end{center}
\end{theorem}

\section{Groups of order $3^8$ and exponent 3}
In this section, we aim to classify groups of order $3^8$ and exponent 3.

Let $G$ be such a group, and $A=G'$, $H=G/G'$.
Assume $A=\mathbb{Z}_3^n$, $H=\mathbb{Z}_3^m$, and $G=\mathbb{Z}_3^n\rtimes_{\rho,f}\mathbb{Z}_3^m$.
Then $n\le 5$, otherwise it is easy to show that $A$ can not be the derived subgroup by Corollary \ref{cor-expp3} and by comparing the dimension.

We work on $n$ case by case. Note that if $n\le 3$, then $\rho$ is trivial, and in this case, the classification
of isoclasses of $G$ is equivalent to the classification of certain congruence classes of subspaces.

\subsection{n=0, or 1}\

This case follows from Theorem \ref{thm-1dim-dergp} easily.

\begin{lemma} Keep the above notations.
\begin{enumerate}
\item[(1)] If $n=0$, then $G = \mathbb{Z}_3^8$ and is decomposable.
\item[(2)] If $n=1$, then $G = G_{7, 1}, G_{7, 2}$, or $G_{7, 3}$, and $G$ is decomposable.
\end{enumerate}
\end{lemma}

\subsection{n=2}\
Then $m=6$, and $G=\mathbb{Z}_3^2\rtimes_{\rho,f}\mathbb{Z}_3^6$.

By Section \ref{sec-2of6}, there are 14 congruence classes of 2 dimensional subspaces of $\Asym_6(\Z_3)$:
\[\begin{pmatrix}&&&a&&\\&&&b&&\\&&&&&\\-a&-b&&&&\\&&&&0&\\&&&&&0\end{pmatrix},
 \begin{pmatrix} &&a&&&\\&&&b&&\\-a&&&&&\\&-b&&&&\\&&&&0&\\&&&&&0\end{pmatrix},
 \begin{pmatrix} &&a&b&&\\&&&a&&\\-a&&&&&\\-b&-a&&&&\\&&&&0&\\&&&&&0\end{pmatrix},
\]
\[
\begin{pmatrix} &&a&b&&\\&&-b&a&&\\-a&b&&&&\\-b&-a&&&&\\&&&&0&\\&&&&&0\end{pmatrix},
  \begin{pmatrix} &&&a&&\\&&&b&&\\&&&&a&\\-a&-b&&&&\\&&-a&&&\\&&&&&0\end{pmatrix},
  \begin{pmatrix} &&&a&&\\&&&b&a&\\&&&& b&\\-a&-b&&&&\\&-a&-b&&&\\&&&&&0\end{pmatrix},
\]
\[
\begin{pmatrix}
           &  &   & a &   &   \\
           &  &   &   & a &   \\
           &  &   &   &   & b \\
         -a&  &   &   &   &   \\
           &-a&   &   &   &   \\
           &  & -b&   &   &   \\
       \end{pmatrix},\quad
       \begin{pmatrix}
           &  &    & a &   &   \\
           &  &    &   & b &   \\
           &  &    &   &   & a+b \\
         -a&  &    &   &   &   \\
           &-b&    &   &   &   \\
           &  &-a-b&   &   &   \\
       \end{pmatrix},
\]
\[
\begin{pmatrix}
           &  &    & a & b &   \\
           &  &    &   & a &   \\
           &  &    &   &   & a \\
         -a&  &    &   &   &   \\
         -b&-a&    &   &   &   \\
           &  & -a &   &   &   \\
       \end{pmatrix},\quad
   \begin{pmatrix}
           &  &   & a & b  &   \\
           &  &   &   & a &   \\
           &  &   &   &   & b \\
         -a&  &   &   &   &   \\
         -b&-a&   &   &   &   \\
           &  & -b&   &   &   \\
       \end{pmatrix},\quad
     \begin{pmatrix}
           &  &    & a & b &   \\
           &  &    &-b & a &   \\
           &  &    &   &   & a \\
         -a& b&    &   &   &   \\
         -b&-a&    &   &   &   \\
           &  & -a &   &   &   \\
       \end{pmatrix},
\]
\[
\begin{pmatrix}
           &  &   & a & b &   \\
           &  &   &   & a & b  \\
           &  &   & b & b & a \\
         -a&  & -b&   &   &   \\
         -b&-a& -b&   &   &   \\
           &-b& -a&   &   &   \\
       \end{pmatrix},\quad
    \begin{pmatrix}
           &  &    & a & b &   \\
           &  &    &   & a & b \\
           &  &    &   &   & a \\
         -a&  &    &   &   &   \\
         -b&-a&    &   &   &   \\
           &-b& -a &   &   &   \\
       \end{pmatrix},\quad
   \begin{pmatrix}
           &  &   &   & a &   \\
           &  &   &   & b &   \\
           &  &   &   &   & a \\
           &  &   &   &   & b \\
         -a&-b&   &   &   &   \\
           &  & -a& -b&   &   \\
       \end{pmatrix}.
\]

Let $G_1, G_2, \cdots, G_{14}$ be the corresponding groups respectively. Then
$G_1= I_{5.2}\times \Z_3^3$, $G_2 = I_3\times I_3\times \Z_3^2$, $G_3= I_{6.1}\times \Z_3^2$, $G_4=I_{6.2}\times \Z_3^2$,
$G_5= I_{7.2}\times \Z_3$, $G_6= I_{7.3}\times \Z_3$ , and $G_7 = I_3\times I_{5.1}$ are decomposable, and the others are all indecomposable. We denote $G_8,\cdots, G_{14}$ by  $I_{8.1},\cdots, I_{8.7}$ respectively. The presentations of these groups can be easily read by Proposition \ref{prop-pres-from-mrep} and we left it to the readers.
\subsection{n=3}
In this case, $m=5$.

By Theorem \ref{thm-3of5}, there are 22 congruence classes of 3-dimensional subspaces of the space $\Asym_5(\Z_3)$, say
\[
\begin{pmatrix}&d&&a&\\-d&&d&b&a\\&-d&&&b\\-a&-b&&&\\&-a&-b&&\end{pmatrix},
\begin{pmatrix}&d&&a&\\-d&&d&b&a\\&-d&&&b\\-a&-b&&&d\\&-a&-b&-d&\end{pmatrix},
\begin{pmatrix}&d&&a&\\-d&&&b&a\\&&&&b\\-a&-b&&&\\&-a&-b&&\end{pmatrix},
\]
\[
\begin{pmatrix}&d&&a&\\-d&&&b&a\\&&&&b\\-a&-b&&&d\\&-a&-b&-d&\end{pmatrix},
\begin{pmatrix}&d&&a&\\-d&&&b&a\\&&&&b\\-a&-b&&&-d\\&-a&-b&d&\end{pmatrix},
\begin{pmatrix}&&d&a&\\&&&b&a\\-d&&&&b\\-a&-b&&&\\&-a&-b&&\end{pmatrix},
\]
\[
\begin{pmatrix}&&&a&\\&&&b&a\\&&&&b\\-a&-b&&&d\\&-a&-b&-d&\end{pmatrix},
\begin{pmatrix}&&&a&d\\&&&b&a+d\\&&&d&b\\-a&-b&-d&&\\-d&-a-d&-b&&\end{pmatrix},
\begin{pmatrix}&&&a&\\&&&b&d\\&&&&b\\-a&-b&&&\\&-d&-b&&\end{pmatrix},
\]
\[
\begin{pmatrix}&&&a&\\&&&b&a\\&&&d&b\\-a&-b&-d&&\\&-a&-b&&\end{pmatrix},
\begin{pmatrix}&&&a&\\&&&b&a+d\\&&&-d&b\\-a&-b&d&&\\&-a-d&-b&&\end{pmatrix},
\begin{pmatrix}&&&a&\\&&&b&a+d\\&&&d&b\\-a&-b&-d&&\\&-a-d&-b&&\end{pmatrix},
\]
\[
\begin{pmatrix} &d&&a&\\-d&&&b&\\&&&&a\\-a&-b&&&\\&&-a&&\end{pmatrix},
\begin{pmatrix} &&&a&\\&&d&b&\\&-d&&&a\\-a&-b&&&\\&&-a&&\end{pmatrix},
\]
\[
\begin{pmatrix} &&&a&\\&&&b&\\&&&&a\\-a&-b&&&d\\&&-a&-d&\end{pmatrix},
\begin{pmatrix} &&&a&\\&&&b&\\&&&&d\\-a&-b&&&\\&&-d&&\end{pmatrix},
\]
\[
\begin{pmatrix}&c&a&&\\-c&&b&&\\-a&-b&&&\\&&&0&\\&&&&0\end{pmatrix},
\begin{pmatrix}&&a&&\\&&b&&\\-a&-b&&c&\\&&-c&&\\&&&&0\end{pmatrix},
\begin{pmatrix}&c&a&&\\-c&&b&&\\-a&-b&&c&\\&&-c&&\\&&&&0\end{pmatrix},
\]
\[
\begin{pmatrix}&&a&c&\\&&b&&\\-a&-b&&&\\-c&&&&\\&&&&0\end{pmatrix},
\begin{pmatrix}&c&a&&\\-c&&&b&\\-a&&&c&\\&-b&-c&&\\&&&&0\end{pmatrix},
\begin{pmatrix}&c&a&b&\\-c&&&a&\\-a&&&-c&\\-b&-a&c&&\\&&&&0\end{pmatrix}.
\]

Let $G_1, G_2, \cdots, G_{22}$ be the corresponding groups. Then $G_{16}= I_3\times I_{5.2}$, $G_{17}= I_{6.3}\times \Z_3^2$, $G_{18}= I_{7.4}\times \Z_3$, $G_{19}= I_{7.5}\times \Z_3$, $G_{20}= I_{7.6}\times \Z_3$, $G_{21}= I_{7.7}\times \Z_3$, and $G_{22}= I_{7.8}\times \Z_3$ are decomposable, and the others are indecomposable. We denote the groups $G_1,G_2,\cdots,G_{15}$ by $I_{8.8}, \cdots, I_{8.22}$ respectively. Moreover, the presentations of these groups can be read easily from the matrix presentations by Proposition \ref{prop-pres-from-mrep}.

\subsection{n=4}
Then $m=4$.

In this case, there are two cases, say the action $\rho$ of $H$ on $A$ is trivial, or nontrivial.

If $\rho$ is trivial, then the isoclasses of such groups are in one-to-one correspondence with the following four congruence class of 4-dimensional subspaces of $\Asym_4(\Z_3)$ as listed in Theorem \ref{thm-4of4}:
\[\begin{pmatrix}&&a&c\\&&b&d\\-a&-b&&\\-c&-d&&\end{pmatrix},
\begin{pmatrix}&d&a&c\\-d&&b&\\-a&-b&&\\-c&&&\end{pmatrix},
\begin{pmatrix}&d&a&c\\-d&&b&\\-a&-b&&d\\-c&&-d&\end{pmatrix},
\begin{pmatrix}&c&a&d\\-c&&d&b\\-a&-d&&c\\-d&-b&-c&\end{pmatrix}.
\]
The groups are all indecomposable, and denoted by $I_{8.23}, I_{8.24}, I_{8.25}$ and $I_{8.26}$ respectively.

Now assume $\rho$ is nontrivial.

 Let $h_1$, $h_2$ and $h_3$ and $h_4$ be a basis of $H$. Then by Corollary \ref{cor-expp3}, there exist some $r\le s\le t$ such that $\varphi_{rs}$, $\varphi_{rt}$ and $\varphi_{st}$ are linearly independent over $A/\mathrm{soc}(A)$. It follows that  $\dim(\mathrm{soc}(A))=1$. Without loss of generality, we assume $(r,s,t)= (1,2,3)$.
Let $a_1=\varphi_{23}$, $a_2 = -\varphi_{13}$, $a_3 = \varphi_{12}$,
 and $a_4= h_1\tr\varphi_{23} -\varphi_{23} =\varphi_{13} - h_2\tr\varphi_{13}= h_3\tr\varphi_{12} -\varphi_{12}$. Under the basis $a_1, a_2, a_3, a_4$, the action $\rho$ of $H$ on $A$ is given by
$$\rho(h_1)=\begin{pmatrix}1&&&\\&1&&\\&&1&\\1&&&1\end{pmatrix},\
\rho(h_2)=\begin{pmatrix}1&&&\\&1&&\\&&1&\\&1&&1\end{pmatrix},\
\rho(h_3)=\begin{pmatrix}1&&&\\&1&&\\&&1&\\&&1&1\end{pmatrix}.$$

We may assume $\rho(h_4)= \id_A$ after changing basis. In fact,
$\rho(h_4)$ commutes with $\rho(h_1)$, $\rho(h_2)$ and $\rho(h_3)$ for $H$ is abelian. Then $\rho(h_4) =\left(\begin{array}{cccc}\lambda &&&\\\lambda_1&\lambda&&\\ \lambda_2&&\lambda&\\ \lambda_3&&&\lambda\end{array}\right)$ for some $\lambda\ne 0$ and $\lambda_1, \lambda_2$ and $\lambda_3$. Moreover, $h_4^3=1$ implies $\lambda=1$. Clearly $\tilde h_4 = h_4h_1^{-\lambda_1}h_2^{-\lambda_2}h_3^{-\lambda_3}$ acts trivially on $A$, and $h_1, h_2, h_3, \tilde h_4$ also form a basis of $H$.

By Corollary \ref{cor-expp1}, we have $(h_1-1)\tr \varphi_{14} = (h_4-1)\tr \varphi_{14} =0$ and $$(h_2-1)\tr \varphi_{14}= -(h_4-1)\tr \varphi_{12}=0, $$ Similarly we have $(h_3-1)\tr \varphi_{14}=0$.
This means that $\varphi_{14}\in \mathrm{soc}(A)$ and hence $\varphi_{14}= \mu_1 a_4$, $\varphi_{24}= \mu_2 a_4$ and $\varphi_{34}=\mu_3 a_4$ for some
$\mu_1, \mu_2, \mu_3\in \Z_3$.

Set $\tilde \varphi_{ij} = \varphi_{ij}$ for $1\le i, j\le 3$, and $\tilde \varphi_{ij} = 0$ for $i=4$ or $j=4$. We claim that $\varphi$ and $\tilde\varphi$ differ by a cocycle.
In fact, by the formula \ref{formula-cochain}, we have
\[(\varphi-\tilde\varphi)(h_1^{i_1}h_2^{i_2}h_3^{i_3} h_4^{i_4}, h_1^{j_1}h_2^{j_2}h_3^{j_3} h_4^{j_4}) = -i_4(j_1 \mu_1 + j_2 \mu_2 + j_3 \mu_3) \mu_3 a_4.
\]
Let $f\colon H\to A$ be given by $f(h_1^{i_1} h_2^{i_2} h_3^{i_3} h_4^{i_4}) = i_4 (\mu_1a_1+ \mu_2a_2+\mu_3a_3)
+ i_4(\mu_1i_1 + \mu_2i_2 + \mu_3i_3)a_4$.
It is direct to show that
\begin{align*}
&d(f)(h_1^{i_1}h_2^{i_2}h_3^{i_3} h_4^{i_4}, h_1^{j_1}h_2^{j_2}h_3^{j_3} h_4^{j_4})\\
=& (h_1^{i_1}h_2^{i_2}h_3^{i_3} h_4^{i_4})\tr f(h_1^{j_1}h_2^{j_2}h_3^{j_3} h_4^{j_4}) - f(h_1^{i_1+j_1}h_2^{i_2+j_2}h_3^{i_3+j_3} h_4^{i_4+j_4})
+ f(h_1^{i_1} h_2^{i_2} h_3^{i_3} h_4^{i_4})\\
=& (h_1^{i_1}h_2^{i_2}h_3^{i_3} h_4^{i_4})\tr [j_4(\mu_1a_1+ \mu_2a_2+\mu_3a_3)
+ j_4(\mu_1j_1 + \mu_2j_2 + \mu_3j_3)a_4] \\
& - [(i_4+j_4)(\mu_1a_1+ \mu_2a_2+\mu_3a_3) \\
& + (i_4+j_4)(\mu_1(i_1+j_1) + \mu_2(i_2+j_2) + \mu_3(i_3+j_3))a_4] \\
&+ [i_4(\mu_1a_1+ \mu_2a_2+\mu_3a_3)
+ i_4(\mu_1i_1 + \mu_2i_2 + \mu_3i_3)a_4]\\
=& -i_4(j_1 \mu_1 + j_2 \mu_2 + j_3 \mu_3) \mu_3 a_4,
\end{align*}
and the claim follows.
Thus $G \cong A\rtimes_{\rho, \tilde\varphi} H \cong I_{7.9}\times \Z_3$, and hence $G$ is decomposable.

\subsection{n=5}

In this case, $m=3$ and the action $\rho$ must be nontrivial.

Let $h_1$, $h_2$ and $h_3$ be a basis of $H$. Then by Corollary \ref{cor-expp3}, $\varphi_{12}$, $\varphi_{13}$ and $\varphi_{23}$ are linearly independent over $A/\mathrm{soc}(A)$, and $h_1\tr\varphi_{23} -\varphi_{23} =\varphi_{13} - h_2\tr\varphi_{13}= h_3\tr\varphi_{12} -\varphi_{12}\ne 0$.
Let $a_1=\varphi_{23}$, $a_2 = -\varphi_{13}$, $a_3 = \varphi_{12}$,
 and $a_4= h_1\tr\varphi_{23} -\varphi_{23}$. Let $\tilde A$ denote the subspace spanned by $a_1, a_2, a_3$ and $a_4$. We pick some $a_5\in A$ such that $a_1, a_2, a_3, a_4, a_5$ form a basis of $A$. Then $(h_1-1)\tr a_5 = \sum_{i=1}^5 \lambda_i a_i$. Clearly we have $\lambda_5=0$ for $0= (h_1-1)^3\tr a_5 = \lambda_5^3 a_5$, i.e., $(h_1-1)\tr a_5 \in \tilde A$. Similary $(h_2-1)\tr a_5, (h_3-1)\tr a_5\in \tilde A$. Then by Proposition \ref{prop-der-abext}, the derived subgroup of $G'= \tilde A$, which leads to a contradiction. Thus we have proved the following result.

\begin{proposition} There exists no group of exponent $3$ and order $3^8$ whose derived subgroup has order $3^5$.
\end{proposition}

\subsection{} In summary, we have the following classification.

\begin{theorem}\label{thm-ord8} The isoclasses of groups of exponent 3 and of order $3^8$ are listed as follows, where $n$ is the dimension of the derived subgroup.
\begin{center}
{
\begin{tabular}{|c|c|c|}
\hline
 $n$ & Decomposable & Indecomposable \\ \hline
   $0$   & $I_1^8$ & \\ \hline
   $1$   & $I_1\times I_{7.1}, I_1^3\times I_{5.1}, I_1^5\times I_3$ &  \\ \hline
   $2$ & $I_1^3\times I_{5.2}, I_1^2\times I_3^2, I_1^2 \times I_{6.1}, I_1^2 \times I_{6.2}, I_1 \times I_{7.2}, I_1 \times I_{7.3}, I_3\times I_{5.1}$ & $I_{8.1}, \cdots, I_{8.7}$ \\ \hline
   $3$ & $I_1^2 \times I_{6.3}, I_1\times I_{7.4}, I_1\times I_{7.5}, I_1\times I_{7.6}, I_1\times I_{7.7}, I_1\times I_{7.8}, I_3\times I_{5.2}$ & $I_{8.8}, \cdots, I_{8.22}$\\ \hline
   $4$ &$I_1\times I_{7.9}$ & $I_{8.23},\cdots, I_{8.26}$\\ \hline
   \end{tabular}
   \footnotesize Table 2. Isoclasses of groups of order $3^8$ and exponent 3}
   \end{center}
\end{theorem}

\section{Congruence classes of 3 and 4-dimensional subspaces of $\Asym_4(\Z_3)$}

We will give a complete set of representatives of congruence classes of $3$ and $4$ dimensional subspaces of $\Asym_4(\Z_3))$ in this section. We need some basic notions.

\subsection{Pfaffian of anti-symmetric matrices of order 4}  In this subsection and the next, $\k$ can be an arbitrary base field.

Let $X= (x_{ij})\in \Asym_4(\k)$ be an anti-symmetric matrix over $\k$. Recall that the \emph{Pfaffian} of $X$ is defined to be $\mathrm{pf}(X) = x_{12}x_{34}- x_{13}x_{24}+ x_{14}x_{23}$, and $\det(X)=\mathrm{pf}(X)^2$. We remark that the notion of Pfaffian is defined for anti-symmetric matrices of arbitrary even order.

\begin{lemma} Let $V\subseteq \Asym_4(\k)$ be an $r$-dimensional subspace with a basis $X_1, \cdots, X_r$. Then $\mathrm{pf}(x_1X_1+ \cdots+ x_rX_r)$ defines a quadratic form in variables $x_1, \cdots, x_r$.

 Moreover, let $\tilde V$ be a subspace of $\Asym_4(\k)$ congruent to $V$, and let $\tilde X_1, \cdots, \tilde X_r$ be a basis of $\tilde V$. Then the quadratic forms $\mathrm{pf}(x_1\tilde X_1+ \cdots+ x_r\tilde X_r)$ and $\mathrm{pf}(x_1X_1+ \cdots+ x_rX_r)$ are equivalent up to a scalar.
\end{lemma}

We use $E_{ij}\in M_m(\k)$ to denote the matrix with $(i,j)$-entry being 1 and other entries being 0. Then the $m\times m$ identity  matrix $I_m= E_{11} + E_{22} + E_{33} + \cdots +E_{mm}$. We set
\begin{enumerate}
\item $P_{ij}= I_m -E_{ii}- E_{jj}+E_{ij}+E_{ji}$,
\item $D_i(\lambda)= I_m + (\lambda-1) E_{ii}$ for $0\ne\lambda\in\k$,
\item $T_{ij}(\lambda) = I_m + \lambda E_{ij}$ for $\lambda\in \k$.
\end{enumerate}
Matrices of the above form are known as elementary matrices. Any congruence transformation is a composite of series congruence transformations by elementary matrices.

It is direct to show that $\mathrm{pf}(P_{ij}XP_{ij}^T) = -\mathrm{pf}(X)$, $\mathrm{pf}(D_{i}(\lambda)XD_i{\lambda}^T) = \lambda\mathrm{pf}(X)$ and $\mathrm{pf}(T_{ij}(\lambda)XT_{ij}{\lambda}^T) =  \mathrm{pf}(X)$ for any $X\in \Asym_4(\k)$ and $\lambda\in\k$. It follows that $\mathrm{pf}(C^TXC)=\det(C)\mathrm{pf}(X)$ for any $C\in \GL_4(\k)$.

\subsection{Orthogonal complement and radical.} $\k$ is an arbitrary field in this subsection.

Let $m$ be a positive integer, and $X\in \Asym_m(\k)$ be an anti-symmetric matrix. For any $u\in \k^m$, the subspace $u^{\perp_X}=\{v\in \k^n\mid uXv^T=0\}$ is called the \emph{orthogonal complement} of $u$ with respect to $X$, or $X$-complement of $u$ for short. Let $\mathbf{U}\subseteq \k^m$ and $V\subseteq \Asym_m(\k)$ be subspaces. Then the $V$-complement of $u$, the $X$-complement and the $V$-complement of $\mathbf U$ are defined by
\[u^{\perp_V}=\bigcap_{X\in V}u^{\perp_X}, \quad
{\mathbf U}^{\perp_X}= \bigcap_{u\in \mathbf U}u^{\perp_X}, \quad
{\mathbf U}^{\perp_V}= \bigcap_{u\in \mathbf U, X\in V}u^{\perp_X}.\]

\begin{definition}
Let $\k$ be an arbitrary field. Let $X\in \Asym_m(\k)$ be an anti-symmetric matrix, where $m$ is a positive integer.
The subspace \[\mathrm{rad}(X) = \{v\in \k^m\mid vXu^T=0, \forall u\in \k^m\} \]
is called the \emph{radical} of $X$. Similarly, for a subspace $V\subseteq \Asym_m(\k)$, the space $\mathrm{rad}(V)=\bigcap\limits_{X\in V}\mathrm{rad}(X)$
is called the radical of $V$.
\end{definition}

We mention that the above definitions coincide with the ones of quadratic forms.

\begin{remark}\label{rem-complement-inv}

(1) For any $X\in \Asym_m(\k)$, $P\in \GL_m(\k)$ and $u\in \k^{m}$, $(uP)^{\perp_X} = (u^{\perp_{PXP^T}})P$.

(2) By definition, $\mathrm{rad}(X)= (\k^m)^{\perp_X}$, and by (1) we have $\mathrm{rad}(PXP^T)P = \mathrm{rad}(X)$. Therefore the radical of an anti-symmetric matrix or a subspace of $\Asym_m(\k)$ is invariant under congruence transformation.

(3) If $m$ is odd, then $\mathrm{rad}(X)\ne 0$ for any $X\in \Asym_m(\k)$ for any anti-symmetric matrix of odd degree is always degenerate.

\end{remark}

We have the following obvious observation.

\begin{lemma}\label{lem-dimincreasing} Let $V\subseteq \Asym_m(\k)$ be an $r$-dimensional subspace. Then $$V^\natural= \{\diag(X,0)\mid X\in \Asym_m(\k)\}$$ is an $r$-dimensional subspace of $\Asym_{m+1}(\k)$, and $V^\natural_1$ is congruent to $V^\natural_2$ if and only if $V_1$ is congruent to $V_2$. Moreover, let $W\subseteq\Asym_{m+1}(\k)$ be an $r$ dimensional subspace with $\mathrm{rad}(W)\ne 0$, then $W= V^\natural$ for some $V\subseteq\Asym_m(\k)$.
\end{lemma}

We return to the classification of $n$-dimensional subspaces of $\Asym_4(\Z_3)$ for $n\le 4$.

\subsection{n=3}\label{sec-3of4}
By Section \ref{sec-2of4}, there are 4 congruence classes of 2-dimensional subspaces of $\Asym_4(\Z_3)$, say
$$
\begin{pmatrix}&&a&\\&&b&\\-a&-b&&\\&&&0\end{pmatrix},
\begin{pmatrix}&&a&\\&&&b\\-a&&&\\&-b&&\end{pmatrix},
\begin{pmatrix}&&a&b\\&&&a\\-a&&&\\-b&-a&&\end{pmatrix},
\begin{pmatrix}&&a&b\\&&-b&a\\-a&b&&\\-b&-a&&\end{pmatrix}.
$$

Now we consider the congruence classes of 3-dimensional subspaces of $\Asym_4(\Z_3)$. The key point is that any such a 3-dimensional subspace is obtained from one of the above 2-dimensional subspaces by adding a one dimensional subspace.

Let $V\subseteq \Asym_4(\Z_3)$ be a 3-dimensional subspace. Without loss of generality, we may assume that $V$ contains a 2-dimensional subspace $W$ as listed above. Thus we have the following four cases.

\textbf{Case 1.}\quad Assume $W = \left(\begin{array}{cccc}&&a&\\&&b&\\-a&-b&&\\&&&0\end{array}\right)$.
Then there exists $(\alpha,\beta,\gamma,\delta)\ne (0,0,0,0)$, such that $V= W + \Z_3\begin{pmatrix}&\alpha && \beta\\-\alpha&&&\gamma\\&&&\delta\\-\beta&-\gamma&-\delta&\end{pmatrix}$, and we denote $V$ of this form by $\langle\alpha,\beta,\gamma,\delta\rangle$.

It is easy to show that $\langle-\alpha,-\beta,-\gamma,-\delta\rangle=\langle\alpha,\beta,\gamma,\delta\rangle\overset{D_4(-1)}\sim\langle\alpha,-\beta,-\gamma,-\delta\rangle$, where the notation $V\overset{P}\sim W$ means that $W=PVP^T$.

If $\beta\ne0$, then
$$\left(\begin{array}{cccc}&c\alpha&a&c\beta\\-c\alpha&&b&c\gamma\\-a&-b&&c\delta\\-c\beta&-c\gamma&-c\delta&\end{array}\right)
\overset{T_{21}(-\beta\gamma)}\sim\left(\begin{array}{cccc}&c\alpha&a&c\beta\\-c\alpha&&b-a\beta\gamma&\\-a&-b+a\beta\gamma&&c\delta\\-c\beta&&-c\delta&\end{array}\right)$$
$$\overset{T_{24}(-\beta\alpha)}\sim\left(\begin{array}{cccc}&&a&c\beta\\&&b-a\beta\gamma+c\alpha\beta\delta&\\-a&-b+a\beta\gamma-c\alpha\beta\delta&&\\-c\beta&&&\end{array}\right)$$
Hence $\langle\alpha,\beta,\gamma,\delta\rangle\sim\langle0,\beta,0,0\rangle=\langle0,1,0,0\rangle$.

Similarly, if $\gamma\ne0$, then $\langle\alpha,\beta,\gamma,\delta\rangle\sim\langle0,0,\gamma,0\rangle=\langle0,0,1,0\rangle$.

If $\beta=\gamma=0$, then clearly $\langle\alpha,\beta,\gamma,\delta\rangle\sim \langle1,0,0,0\rangle$, $\langle0,0,0,1\rangle$ or $\langle1,0,0,1\rangle$.

Combined with the fact $\langle0,1,0,0\rangle\overset{P_{12}}\sim\langle0,0,1,0\rangle$, we know that $\langle\alpha,\beta,\gamma,\delta\rangle$ is congruent to one of $\langle1,0,0,0\rangle$, $\langle0,1,0,0\rangle$, $\langle0,0,0,1\rangle$ and $\langle1,0,0,1\rangle$. We claim that these four subspaces are not congruent to each other.

The subspaces $\langle1,0,0,0\rangle$, $\langle0,1,0,0\rangle$, $\langle0,0,0,1\rangle$ and $\langle1,0,0,1\rangle$ are given by
\[\begin{pmatrix}
  & c& a& \\
 c&  & b& \\
-a&-b&  & \\
  &  &  &0\\
\end{pmatrix},
\begin{pmatrix}
  &  & a&c \\
  &  & b& \\
-a&-b&  & \\
-c&  &  & \\
\end{pmatrix},
\begin{pmatrix}
  &  & a& \\
  &  & b& \\
-a&-b&  &c\\
  &  &-c& \\
\end{pmatrix},
\begin{pmatrix}
  & c& a& \\
 c&  & b& \\
-a&-b&  &c\\
  &  &-c& \\
\end{pmatrix},
\]
and the Pfaffians are $0$, $bc$, $0$ and $c^2$ respectively. By comparing the Pfaffians, to prove the claim it suffices to show that $\langle1,0,0,0\rangle\not\sim\langle0,0,0,1\rangle$. Easy calculation shows that $$\mathrm{rad}(\langle1,0,0,0\rangle)=(0,0,0,\Z_3):=\{(0,0,0,a)\mid a\in\Z_3\},$$
while $\mathrm{rad}(\langle0,0,0,1\rangle)=0$, and it follows that $\langle1,0,0,0\rangle\not\sim\langle0,0,0,1\rangle$.

\textbf{Case 2.}\quad
Assume $W=\left(\begin{array}{cccc}&&a&\\&&&b\\-a&&&\\&-b&&\end{array}\right)$.
Then $V= W+ \Z_3\begin{pmatrix}&\alpha&&\beta\\-\alpha&&\gamma&\\&-\gamma&&\delta\\-\beta&&-\delta&\end{pmatrix}$, and we denote this space by $(\alpha,\beta,\gamma,\delta)$.

First we have $(\alpha,\beta,\gamma,\delta)=(-\alpha,-\beta,-\gamma,-\delta)$, and moreover, \[(\alpha,\beta,\gamma,\delta)\overset{P_{24}}\sim(\beta,\alpha,-\delta,-\gamma)
\overset{P_{13}}\sim (\delta,-\gamma,-\beta,\alpha)\overset{P_{24}}\sim
(-\gamma,\delta,-\alpha,\beta)
,\]
\[(-\alpha,-\beta,\gamma,\delta)\overset {D_1(-1)}\sim(\alpha,\beta,\gamma,\delta)\overset {D_4(-1)}\sim(\alpha,-\beta,\gamma,-\delta).\]
Thus we may assume that $\alpha\ne0$. Without loss of generality, we can take $\alpha=1$. Now we have $(1,\beta,\gamma,\delta)\overset{T_{42}(-\beta)}\sim
(1, 0, \gamma, \delta+ \beta\gamma)\overset{T_{31}(-\gamma)}\sim(1,0,0,\delta+ \beta\gamma)$.  Hence there are at most two cases, say $(1,0,0,0)$ and $(1,0,0,1)$. Clearly, $(1,0,0,0)\overset{P_{24}}\sim(0,1,0,0)\overset{P_{34}}\sim\langle0,1,0,0\rangle$.

The subspace $(1,0,0,1)$ is given by
$\begin{pmatrix}
  & c& a& \\
-c&  &  &b \\
-a&  &  &c\\
  &-b&-c& \\
\end{pmatrix},$
and the Pfaffian of $(1,0,0,1)$ is $c^2-ab$, by comparing the Pfaffians, we know that  $(1,0,0,0)$ is congruent to none of $\langle1,0,0,0\rangle$, $\langle0,1,0,0\rangle$, $\langle0,0,0,1\rangle$ and $\langle1,0,0,1\rangle$.

\textbf{Case 3.}\quad
Assume  $W=\left(\begin{array}{cccc}&&a&b\\&&&a\\-a&&&\\-b&-a&&\end{array}\right)$, then $V=W + \Z_3\begin{pmatrix} &\alpha& \beta& \\- \alpha & & \gamma &\\- \beta&- \gamma&& \delta\\&&- \delta&\end{pmatrix}$ for some $\alpha, \beta, \delta, \gamma$. We denote this 3-dimensional subspace by $[\alpha,\beta,\gamma,\delta]$.

By direct calculation, we have \[[\alpha,\beta,\gamma,\delta]\overset{D_{1,2}(-1)}\sim[\alpha,-\beta,-\gamma,\delta]
\overset{D_{1,3}(-1)}\sim[-\alpha,-\beta,\gamma,-\delta]\overset{D_{1,2}(-1)}\sim[-\alpha,\beta,-\gamma,-\delta],\]
where $D_{1,2}(-1)=\diag(-1,-1,1,1)$ and $D_{1,3}(-1)=\diag(-1,1,-1,1)$. Consider
$$\mathrm{pf}\begin{pmatrix}&c\alpha&a+c\beta&\\-c\alpha&&c\gamma&a\\
-a-c\beta&-c\gamma&&c\delta\\&-a&-c\delta&\end{pmatrix}
=c^2\alpha\delta-a^2-ac\beta=-(a-c\beta)^2+c^2(\beta^2+\alpha\delta).$$

If $\beta^2+\alpha\delta=0$, then we take $a=\beta,c=1$ to get a rank two element in the 2-dimensional subspace $\begin{pmatrix}&c\alpha&a+c\beta&\\-c\alpha&&c\gamma&a\\-a-c\beta&-c\gamma&&c\delta\\&-a&-c\delta&\end{pmatrix}$,
and together with $\left(\begin{array}{cccc}&&&1\\&&0&\\&0&&\\-1&&&\end{array}\right)$ it spans a 2-dimensional subspace, which is equivalent to the first two classes of 2-dimensional subspaces. Hence this case reduces to the first two cases.

If $\beta^2+\alpha\delta=1$, then by taking $a=\beta+1,c=1$ we get a rank two element in the 2-dimensional subspace $\left(\begin{array}{cccc}&c\alpha&a+c\beta&\\-c\alpha&&c\gamma&a\\-a-c\beta&-c\gamma&&c\delta\\&-a&-c\delta&\end{array}\right)$.
Together with $\left(\begin{array}{cccc}&&&1\\&&0&\\&0&&\\-1&&&\end{array}\right)$, it spans a 2-dimensional subspace, which is equivalent to the first two classes of 2-dimensional subspaces. Hence this case reduces to the first two cases.

It remains to consider the case $\beta^2+\alpha\delta=-1$. Then $\beta=0$, $\alpha\delta=-1$ or $\beta\ne0$, $\alpha\delta=1$. In either case, $\alpha\ne0$ and $\delta\ne0$.

We may assume $\gamma=0$. Otherwise if $\gamma\ne0$, then $[\alpha,\beta,\gamma,\delta]\overset{T_{42}(\delta \gamma)}\sim[\alpha,\beta,\gamma,0]$, and we will go back to the cases $\beta^2+\alpha\delta=0$ or $1$.

Now we have $[\alpha,\beta,\gamma,\delta]\sim [1,0,0,-1]$ or $[1,1,0,1]$.
We claim that $[1,1,0,1]=[1,0,0,-1]$. In fact,
\begin{align*}
[1,1,0,1]=&\begin{pmatrix}&c&a+c&b\\-c&&&a\\-a-c&&&c\\-b&-a&-c&\end{pmatrix}
\overset{T_{32}(1)}\sim
\begin{pmatrix}&c&a-c&b\\-c&&&a+c\\-a+c&&&c\\-b&-a-c&-c& \end{pmatrix}\\
\overset{T_{41}(\delta \gamma)}\sim &
\begin{pmatrix}&c&a-c&b\\-c&&&a-c\\c-a&&&-c\\-b&c-a&c&\end{pmatrix}
=\begin{pmatrix}&c&a&b\\-c&&&a\\-a&&&-c\\-b&-a&c&\end{pmatrix}=[1,0,0,-1].
\end{align*}

\textbf{Case 4.} \quad Assume $W=\begin{pmatrix}&&a&b\\&&-b&a\\-a&b&&\\-b&-a&&\end{pmatrix}$, then
$V= W+ \Z_3\begin{pmatrix}&c\alpha&c\beta&c\gamma\\-c\alpha&&&\\-c\beta&&&c\delta\\-c\gamma&&-c\delta&\end{pmatrix}$.
Consider the determinant $$\det\begin{pmatrix}&c\alpha&a+c\beta&b+c\gamma\\
-c\alpha&&-b&a\\-a-c\beta&b&&c\delta\\-b-c\gamma&-a&-c\delta&\end{pmatrix}
=(-(a-c\beta)^2-(b-c\gamma)^2+c^2(\alpha\delta+\beta^2+\gamma^2))^2.$$
Then there always exists some element of rank two element in the 3-dimensional subspace. Hence this case reduces to the first three cases.

In summary, we have the following classification.
\begin{theorem}\label{thm-3of4}
There are 6 congruence classes of 3-dimensional subspaces of $\Asym_4(\mathbb{Z}_3)$:
$$\begin{pmatrix}&c&a&\\-c&&b&\\-a&-b&&\\&&&0\end{pmatrix},
\begin{pmatrix}&&a&\\&&b&\\-a&-b&&c\\&&-c&\end{pmatrix},
\begin{pmatrix}&c&a&\\-c&&b&\\-a&-b&&c\\&&-c&\end{pmatrix},$$
$$\begin{pmatrix}&&a&c\\&&b&\\-a&-b&&\\-c&&&\end{pmatrix},
\begin{pmatrix}&c&a&\\-c&&&b\\-a&&&c\\&-b&-c&\end{pmatrix},
\begin{pmatrix}&c&a&b\\-c&&&a\\-a&&&-c\\-b&-a&c&\end{pmatrix}.$$
Moreover, the Pfaffians are $0,0,c^2,bc,c^2-ab,-c^2-a^2$ respectively, which are pairwise nonequivalent except the first two.
\end{theorem}

\subsection{n=4.}\label{sec-4of4} In this subsection, we will give a full list of congruence classes of 4-dimensional subspaces of $\Asym_4(\Z_3)$.

Let $V\subset\Asym_4(\Z_3)$ be of dimension 4, and $W$ be a subspace of $V$ dimension 3. We may assume that $W$ is of one of the six subspaces as listed in Theorem \ref{thm-3of4}.

\textbf{Case 1.}\quad Assume $W= \begin{pmatrix}&&a&c\\&&b&\\-a&-b&&\\-c&&&\end{pmatrix}$. Then $V= W + \Z_3 \begin{pmatrix}&\alpha&&\\-\alpha&&&\beta\\&&&\gamma\\&-\beta&-\gamma&\end{pmatrix}$
for some $\alpha, \beta, \gamma\in\Z_3$. We denote $V$ of this form by $\langle\alpha,\beta,\gamma\rangle$.

Applying the congruence transformation by $P_{13}P_{24}$ we obtain that $\langle\alpha,\beta,\gamma\rangle\sim\langle\gamma,-\beta,\alpha\rangle$, and applying $D_1(-1), D_2(-1)$ and $D_3(-1)$, we obtain that
$\langle\alpha,\beta,\gamma\rangle\sim\langle-\alpha,\beta,\gamma\rangle\sim
\langle\alpha,-\beta,\gamma\rangle\sim\langle\alpha,\beta,-\gamma\rangle.
$

If $\beta\ne0$, then
$\langle\alpha,\beta,\gamma\rangle\overset{T_{32}
(-\gamma\beta)}\sim\langle\alpha,\beta,0\rangle\overset{T_{14}(\alpha\beta)}\sim
\langle0,\beta,0\rangle\sim\langle0,1,0\rangle$; else if $\beta=0$, then
$\langle\alpha,\beta,\gamma\rangle\sim\langle1,0,0\rangle$ or $\langle1, 0, 1\rangle$.
The subspaces $\langle0,1,0\rangle,\langle1,0,0\rangle$ and $\langle1,0,1\rangle$ are
\[\begin{pmatrix}&&a&c\\&&b&d\\-a&-b&&\\-c&-d&&\end{pmatrix},
\begin{pmatrix}&d&a&c\\-d&&b&\\-a&-b&&\\-c&&&\end{pmatrix},
\begin{pmatrix}&d&a&c\\-d&&b&\\-a&-b&&d\\-c&&-d&\end{pmatrix}.
\]

\textbf{Case 2.}\quad
Assume $W=\left(\begin{array}{cccc}&c&a&\\-c&&&b\\-a&&&c\\&-b&-c&\end{array}\right)$, then
$V= W +\Z_3
\left(\begin{array}{cccc}&&&\alpha\\&&\beta&\\&-\beta&&\gamma\\-\alpha&&-\gamma&\end{array}\right)$, for some $\alpha, \beta$ and $\gamma$. We denote the subspace of this form by $(\alpha,\beta,\gamma)$.

It is easy to check that $(\alpha,\beta,\gamma)\overset{P_{12}P_{34}}\sim(\beta,\alpha,-\gamma)$, and \[(\alpha,\beta,\gamma)\overset{D_1(-1)D_4(-1)}\sim(\alpha,\beta,-\gamma)\overset{D_1(-1)D_2(-1)}\sim(-\alpha,-\beta,-\gamma).\]
If $\gamma=0$, then $(\alpha,\beta,\gamma)\sim (1,1,0)$, $(1,-1,0)$ or $(1,0,0)$.
If $\gamma\ne0$, then $(\alpha,\beta,\gamma)\sim (1,1,1)$, $(1,-1,1)$, $(1,0,1)$ or $(0,0,1)$.

Now we compare $(\alpha, \beta, \gamma)$'s with $\langle\alpha, \beta, \gamma\rangle$'s.

(i) $(0,0,1)\overset{P_{13}}\sim \langle0,1,0\rangle$.

(ii) $(1,1,1)\sim\langle1,0,1\rangle$ for
\begin{align*}
(1,1,1)=& \begin{pmatrix}&c&a&d\\-c&&d&b\\-a&-d&&c+d\\-d&-b&-c-d&\end{pmatrix}\overset{T_{31}(1)T_{42}(1)}\sim
\begin{pmatrix}&c&a&d+c\\-c&&d-c&b\\-a&c-d&&d-c\\-d-c&-b&c-d& \end{pmatrix}\\
\overset{T_{42}(1)} \sim&
\begin{pmatrix}&c&a&d-c\\-c&&d-c&b\\-a&c-d&&\\c-d&-b&&\end{pmatrix}
\overset{P_{34}P_{24}} \sim \begin{pmatrix}&d-c&c&a\\c-d&&-b&\\-c&b&&d-c\\-a&&c-d&\end{pmatrix}
=\langle1,0,1\rangle
\end{align*}

(iii) $(1,1,0)\sim (1,-1,1)$ for
\begin{align*}(1,1,0)=&\begin{pmatrix}&c&a&d\\-c&&d&b\\-a&-d&&c\\-d&-b&-c&\end{pmatrix}
\overset{T_{24}(-1)}\sim
\begin{pmatrix}&c-d&a&d\\-c+d&&d+c&b\\-a&-d-c&&c\\-d&-b&-c&\end{pmatrix}\\
\overset{T_{31}(1)}\sim
&\begin{pmatrix}&c-d&a&d\\-c+d&&-d&b\\-a&-d&&(c-d)-d\\-d&-b&d-(c-d)&\end{pmatrix} = (1,-1,-1)\sim (1,-1,1).
\end{align*}

(iv) $(1,0,1)\sim(0,0,1)\sim\langle0,1,0\rangle$ for
\begin{align*}(1,0,1)=&\begin{pmatrix}&c&a&d\\-c&&&b\\-a&&&c+d\\-d&-b&-c-d&\end{pmatrix}
\overset{T_{41}(1)}\sim\begin{pmatrix}&c&a&c+d\\-c&&&b\\-a&&&c+d\\-c-d&-b&-c-d&\end{pmatrix}\\
\overset{T_{13}(-1)}\sim&
\begin{pmatrix}&c&a&\\-c&&&b\\-a&&&c+d\\&-b&-c-d&\end{pmatrix}=(0,0,1)\sim \langle0,1,0\rangle.
\end{align*}

(v) $(1,-1,0)\sim\langle0,1,0\rangle$ for
\begin{align*} (1,-1,0)=&\begin{pmatrix}&c&a&d\\-c&&-d&b\\-a&d&&c\\-d&-b&-c&\end{pmatrix}
\overset{T_{42}(1)} \sim \begin{pmatrix}&c&a&d+c\\-c&&-d&b\\-a&d&&c+d\\-c-d&-b&-c-d&\end{pmatrix}\\
\overset{T_{13}(-1)} \sim
&\begin{pmatrix}&(c-d)&a&\\d-c&&-d&b\\-a&d&&(c-d)-d\\&-b&d-(c-d)&\end{pmatrix}=(0,1,1)\sim(1,0,1)
\sim\langle0,1,0\rangle.
\end{align*}

(vi) $(1,0,0) \sim\langle1,0,1\rangle$ for
$$(1,0,0)= \begin{pmatrix}&c&a&d\\-c&&&b\\-a&&&c\\-d&-b&-c&\end{pmatrix}
\overset{P_{34}}\sim
\begin{pmatrix}&c&d&a\\-c&&b&\\-d&-b&&-c\\-a&&c&\end{pmatrix}= \langle1,0,-1\rangle\sim \langle1,0,1\rangle.$$

\textbf{Case 3.}\quad
Assume $W= \left(\begin{array}{cccc}&c&a&b\\-c&&&a\\-a&&&-c\\-b&-a&c&\end{array}\right)$, and
$V= W + \Z_3 \begin{pmatrix}&&&\\&&\alpha&\beta\\&-\alpha&&\gamma\\&-\beta&-\gamma&\end{pmatrix}$, we denote the 4-dimensional subspace of this form by $[\alpha,\beta,\gamma]$.

We note that $$[\alpha,\beta,\gamma]\overset{D_1(-1)D_2(-1)}\sim [-\alpha,-\beta,\gamma]\overset{D_{1}(-1)D_3(-1)}\sim[\alpha,-\beta,-\gamma]\overset{D_1(-1)D_2(-1)}\sim [-\alpha,\beta,-\gamma].$$
If $\alpha\ne1$, then we may assume $\alpha=1$ for $[\alpha,\beta,\gamma]\sim[-\alpha,-\beta,\gamma]$, and $[1,\beta,\alpha]\sim(1,1,0)$ for
\begin{align*} [1,\beta,\gamma]=&\begin{pmatrix}&c&a&b\\-c&&d&a+d\beta\\-a&-d&&-c+d\gamma\\-b&-a-d\beta&c-d\gamma&\end{pmatrix}
\overset{T_{43}(\beta)}\sim \begin{pmatrix}&c&a&b-a\beta\\-c&&d&a\\-a&-d&&-c+d\gamma\\a\beta-b&-a&c-d\gamma&\end{pmatrix}\\
\overset{T_{42}(\gamma)}\sim & \begin{pmatrix}&c&a&b-a\beta+c\gamma\\-c&&d&a\\-a&-d&&-c\\-b+a\beta-c\gamma&-a&c&\end{pmatrix}
= \begin{pmatrix}&c&a&b\\-c&&d&a\\-a&-d&&-c\\-b&-a&c&\end{pmatrix}\\
\overset{P_{34}}\sim &
\begin{pmatrix}&c&b&a\\-c&&a&d\\-b&-a&&c\\-a&-d&-c&\end{pmatrix} = (1,1,0).
\end{align*}

If $\alpha=0$, then $[0,\beta,\gamma]\sim[0,0,1], [0,1,0]$ or $[0,1,1]$.
We can show that

(i) $[0,0,1]\overset{P_{34}P_{23}}\sim \langle1,0,1\rangle$;

(ii)  $[0,1,0]\overset{P_{34}}\sim \langle1,0,1\rangle$;

(iii) $[0,1,1]\sim \langle1,0,1\rangle$ for
\begin{align*}
[0,1,1]=&\begin{pmatrix} &c&a&b\\-c&&&a+d\\-a&&&-c+d\\-b&-a-d&c-d&\end{pmatrix}
\overset{T_{32}(-1)}\sim \begin{pmatrix} &c&a-c&b\\-c&&&a+d\\c-a&&&-c-a\\-b&-a-d&a+c&\end{pmatrix}\\
\overset{T_{23}(-1)}\sim &\begin{pmatrix} &-a-c&a-c&b\\a+c&&&-a+c+d\\c-a&&&-a-c\\-b&a-c-d&a+c&\end{pmatrix}
=\begin{pmatrix} &c&a&b\\-c&&&d\\-a&&&-c\\-b&-d&c&\end{pmatrix} \overset{P_{34}}\sim \langle1,0,1\rangle.
\end{align*}

\textbf{Case 4.}
Assume $W=\left(\begin{array}{cccc}&c&a&\\-c&&b&\\-a&-b&&c\\&&-c&\end{array}\right)$, and $V= \begin{pmatrix}&c&a&d\alpha\\-c&&b&d\beta\\-a&-b&&c+d\gamma\\-d\alpha&-d\beta&-c-d\gamma&\end{pmatrix}$ for some $\alpha, \beta, \gamma$. Then
the Pfaffian is $c^2+cd\gamma-ad\beta+bd\alpha$.

If $\alpha\ne0$, then the 3-dimensional subspace $\{a=0\}$ is equivalent to the 3-dimensional subspaces considered in the first three cases, since its Pfaffian is neither equivalent to $c^2$ nor to 0. Hence this case reduces to the first three cases.

If $\beta\ne0$, the 3-dimensional subspace $\{b=0\}$ is equivalent to the 3-dimensional subspaces considered in the first three cases, since its Pfaffian is neither equivalent to $c^2$ nor to 0. Hence this case reduces to the first three cases.

If $\alpha=\beta=0$, then $\gamma\ne0$, and hence the 3-dimensional subspace $\{a=0\}$ is equivalent to the 3-dimensional subspaces considered in the first three cases, since its Pfaffian is neither equivalent to $c^2$ nor 0. Hence this case reduces to the first three cases.

\textbf{Case 5.}
 Assume $W=\begin{pmatrix}&c&a&\\-c&&b&\\-a&-b&&\\&&&0\end{pmatrix}$, and $V= \begin{pmatrix}&c&a&d\alpha\\-c&&b&d\beta\\-a&-b&&d\gamma\\-d\alpha&-d\beta&-d\gamma&\end{pmatrix}$ for some $\alpha, \beta$ and $\gamma$. The Pfaffian is $cd\gamma-ad\beta+bd\alpha$.

 We claim that there always exists a full rank element in $V$ and hence this case reduces to the first four cases discussed above.
 In fact, if $\alpha\ne0$, then we can take $a=c=0$, $b=d=1$; if $\beta\ne0$, then we can take $b=c=0$, $a=d=1$; and if $\gamma\ne0$, then we can take $a=b=0$, $c=d=1$. It is direct to check that in each case, the corresponding matrix is of full rank.

\textbf{Case 6.} Assume $W= \begin{pmatrix}&&a&\\&&b&\\-a&-b&&c\\&&-c&\end{pmatrix}$, then
$V=\begin{pmatrix}&d\alpha&a&d\beta\\-d\alpha&&b&d\gamma\\-a&-b&&c\\-d\beta&-d\gamma&-c&\end{pmatrix}$ for some
$\alpha, \beta$ and $\gamma$,  and the Pfaffian is
$cd\alpha-ad\gamma+bd\beta$. Similar arguments as in Case 5 shows that this case reduces to the first four cases.

Now we have obtained all possible equivalence subspaces, say $\langle0,1,0\rangle$, $\langle1,0,0\rangle$, $\langle1,0,1\rangle$ and $(1,1,0)$. Moreover, they are pairwise noncongruent.

\begin{theorem}\label{thm-4of4}
There are 4 congruence classes of 4-dimensional subspaces of the space of $\Asym_4(\mathbb{Z}_3)$:
\[\begin{pmatrix}&&a&c\\&&b&d\\-a&-b&&\\-c&-d&&\end{pmatrix},
\begin{pmatrix}&d&a&c\\-d&&b&\\-a&-b&&\\-c&&&\end{pmatrix},
\begin{pmatrix}&d&a&c\\-d&&b&\\-a&-b&&d\\-c&&-d&\end{pmatrix},
\begin{pmatrix}&c&a&d\\-c&&d&b\\-a&-d&&c\\-d&-b&-c&\end{pmatrix}.
\]
Moreover, the Pfaffians are $bc-ad$, $bc$, $bc+d^2$ and $c^2-ab+d^2$ respectively which are pairwise nonequivalent.
\end{theorem}

\begin{remark}
There is a non-degenerate symmetric bilinear form $(\, ,\,)$ on $\Asym_4(\mathbb{Z}_3)$ given by $(X,Y) =\mathrm{tr}(XY)$. For any $V\subseteq\Asym_4(\Z_3)$, set $V^\perp= \{Y\in \Asym_4(\Z_3)\mid \mathrm{tr}(XY)=0, \forall X\in V\}$ to be its orthogonal complement. Clearly $\dim V + \dim V^\perp = 6$. Since $\mathrm{tr}(PXP^{-1})=\mathrm{tr}(X)$ for any $P\in \GL_4(\Z_3)$, we have $(PVP^T)^\perp=(P^{-1})^T V^\perp P^{-1}$, which implies a one-to-one correspondence between congruence classes of $k$-dimensional subspaces of $\Asym_4(\mathbb{Z}_3)$ and the $6-k$ dimensional ones under taking $()^\perp$. We may compare Theorem \ref{thm-4of4} and Section 6.3.
\end{remark}

\section{Congruence classes of 3-dimensional subspaces of $\Asym_5(\mathbb{Z}_3)$}

In this section, we will give a complete set of representatives of congruence classes of 3-dimensional subspaces of $\Asym_5(\Z_3)$.

\subsection{A technical lemma} First we have the following lemma.

\begin{lemma}\label{lem-trivialradical}
Let $V\subseteq\Asym_5(\Z_3)$ be a subspace with $\mathrm{rad}(V)=0$.
\begin{enumerate}
\item[(1)] If $V$ has a basis $X, Y, Z$ such that $\rank(aX+bY)\le 2$ for all $a, b\in \Z_3$ and $\rank(Z)=2$, then
$V$ is congruent to $\diag(\begin{pmatrix}&&a\\&&b\\-a&-b&\end{pmatrix},\begin{pmatrix}&c\\-c&\end{pmatrix})$.
\item[(2)] There exist some $X, Y\in V$ such that $\rank(X)=4$ and $\mathrm{rad}(X)\cap\mathrm{rad}(Y)=0$.
\end{enumerate}
\end{lemma}
\begin{proof}
(1) Let $W$ be the two dimensional subspace spanned by $X$ and $Y$. By Section \ref{sec-2of5}, we may assume $W = \diag(\begin{pmatrix}&&a\\&&b\\-a&-b&\end{pmatrix},0)$.  Since $\rank(Z)=2$, there exists some $P\in \GL_5(\Z_3)$ such that  $PZ P^T=\diag(0,\begin{pmatrix}&1\\-1&\end{pmatrix})$. We may write $P =\begin{pmatrix}P_1&P_2\\P_3&P_4\end{pmatrix}$, where $P_1\in M_{3\times 3}(\Z_3)$, $P_2\in M_{3\times 2}(\Z_3)$, $P_3\in M_{2\times 3}(\Z_3)$ and  $P_4\in M_{2\times 2}(\Z_3)$ are matrices.

Since $\dim\mathrm{rad}(W)=2$ and $\mathrm{rad}(W)\cap \mathrm{rad}(Z)=0$, we can show that $P_1$ is invertible. Set $Q=\begin{pmatrix}P_1^{-1}&\\-P_3P_1^{-1}& I\end{pmatrix}P$. Then $Q$ has the form $\begin{pmatrix}I_3 & Q_2 \\  &Q_4\end{pmatrix}$, and $QZQ^{T}= \diag(0,\begin{pmatrix}&1\\-1&\end{pmatrix})$. Now the conclusion follows from the fact that $W$ is invariant under the congruence transformation by $Q$.

(2) From the proof of (1) we show that there exists some $X\in V$ with $\rank(X)=4$.
 If $\mathrm{rad}(X)\cap\mathrm{rad}(Y)\ne0$ for any $Y\in V$, then $\mathrm{rad}(X)\subseteq\mathrm{rad}(Y)$ for any $Y$, which means that $\mathrm{rad}(V)=\mathrm{rad}(X_1)\ne 0$, contradicts to the assumption $\mathrm{rad}(V)=0$.
Thus there must exist some $Y\in V$ such that $\mathrm{rad}(X)\cap\mathrm{rad}(Y)=0$.
\end{proof}

In Section 9, we have obtained all congruence classes of 3 and 4-dimensional subspaces of $\Asym_4(\Z_3)$. By Lemma \ref{lem-dimincreasing}, to determine the congruence classes of subspaces of $\Asym_5(\Z_3)$, we need only to consider those ones which have trivial radicals.

Let $V\subseteq \Asym_5(\Z_3)$ be a 3-dimensional subspace with $\mathrm{rad}(V)=0$. Then by Lemma \ref{lem-trivialradical} and Section \ref{sec-2of5}, $V$ has a 2-dimensional subspace $W$ which is congruent to one of the following two subspaces:
\[W_1= \begin{pmatrix}&&&a&\\&&&b&a\\&&&&b\\-a&-b&&&\\&-a&-b&&\end{pmatrix}, \ W_2 = \begin{pmatrix}&&&a&\\&&&b&\\&&&&a\\-a&-b&&&\\&&-a&&\end{pmatrix}.\]
In the following two subsections, we will deal with these two cases separately.

\subsection{Case 1: $V$ has a two dimensional subspace congruent to $W_1$.}\

In this subsection, we assume that $V$ contains a 2-dimensional subspace $W$ such that any nonzero element has rank 4. Then we may assume $W= W_1$.

We recall a group homomorphism $\GL_2(\Z_3)\to \GL_3(\Z_3)$ induced by the natural action of $\GL_2(\Z_3)$ on the homogeneous polynomials of degree 2 in two variables. Precisely, any $X=\begin{pmatrix}a&b\\c&d\end{pmatrix}\in\GL_2(\Z_3)$ maps to
$\hat X= \begin{pmatrix}a^2&ab&b^2\\2ac&ad+bc&2bd\\c^2&cd&d^2\end{pmatrix}\in \GL_3(\Z_3)$. For instance,
\[\widehat{\begin{pmatrix}1&\\&-1\end{pmatrix}} = \begin{pmatrix}1&&\\&-1&\\&&1\end{pmatrix},
\qquad \widehat{\begin{pmatrix}&1\\1&\end{pmatrix}} = \begin{pmatrix}&&1\\&1&\\1&&\end{pmatrix},\]
\[\widehat{\begin{pmatrix}1&1\\&1\end{pmatrix}} = \begin{pmatrix}1&1&1\\&1&2\\&&1\end{pmatrix},
\qquad\quad\ \widehat{\begin{pmatrix}1&\\1&1\end{pmatrix}} = \begin{pmatrix}1&&\\2&1&\\1&1&1\end{pmatrix}.
\]

Then $\GL_2(\Z_3)$ acts on $M_{3\times2}(\Z_3)$ by $X\cdot M = \hat X M X^{-1}$ for any $X\in \GL_2(\Z_3)$ and $M\in M_{3\times2}(\Z_3)$.
Let $E=\left\{\left.\begin{pmatrix}a&\\b&a\\&b\end{pmatrix}\right|a,b\in\mathbb{Z}_3\right\}$ be a subspace of $M_{3\times2}(\Z_3)$. Then $\hat{X}E=EX$ for any $X\in \GL_2(\Z_3)$, or in other words, $E$ is a subrepresentation.
Consequently, $\diag(\hat X, (X^{-1})^T) W \diag({\hat X}^T, (X^{-1})^T) = W$.

\begin{remark} \label{rem-E-stabilizer} Clearly, let $Z$ be a square matrix of order 3. If $ZE\subseteq E$, then $Z=\lambda I_3$ for some $\lambda\in\Z_3$. Consequently, if $ZEX^{-1}\subseteq E$ for some $X\in \GL_2(\Z_3)$, then $Z=\lambda\hat X$ for some $\lambda\in \Z_3$.
\end{remark}

Define an equivalence relation ``$\sim$" on $\mathbb{Z}_3^3$ by $(\alpha,\beta,\gamma)\sim(\alpha',\beta',\gamma')$ if
$$\left(\begin{array}{ccc}0&\alpha'&\beta'\\-\alpha'&0&\gamma'\\-\beta'&-\gamma'&0\end{array}\right)=\hat{X}\left(\begin{array}{ccc}0&\alpha&\beta\\-\alpha&0&\gamma\\-\beta&-\gamma&0\end{array}\right)\hat{X}^T$$
for some $X\in\text{GL}_2(\mathbb{Z}_3)$.

By applying the action by $\begin{pmatrix}1&\\&-1\end{pmatrix}$, $\begin{pmatrix}&1\\1&\end{pmatrix}$, $\begin{pmatrix}1&1\\&1\end{pmatrix}$ and $\begin{pmatrix}1&\\1&1\end{pmatrix}$ iteratively, one shows easily that \begin{align*}(\alpha,\beta,\gamma)\sim&(-\alpha,\beta,-\gamma)\sim(-\gamma,-\beta,-\alpha)\sim(\gamma,-\beta,\alpha)\\
\sim&(\alpha-\beta+\gamma,\beta+\gamma,\gamma)\sim(\alpha+\beta+\gamma,\beta-\gamma,\gamma)\\
\sim&(\alpha,\alpha+\beta,\alpha-\beta+\gamma)\sim(\alpha,\alpha-\beta,\alpha+\beta+\gamma).
\end{align*}
If $\alpha\ne0$ or $\gamma\ne0$, then $(\alpha,\beta,\gamma)\sim(1,0,*)$, otherwise $(\alpha,\beta,\gamma)\sim(0,1,0)$ or $(0,0,0)$.
We note that $(1,0,-1)\sim(-1,0,1)\sim(0,1,1)\sim(0,-1,1)\sim(0,1,0)$.
Hence any $(\alpha, \beta, \gamma)$ is equivalent to one of $(1,0,0)$,$(1,0,1)$, $(0,1,0)$ and $(0,0,0)$.
We claim that these four classes are not congruent to each other.

Firstly, by comparing the rank, we know that $(0,0,0)$ is not equivalent to other 3 classes. By direct calculation,
\begin{align*}&\begin{pmatrix}a^2&ab&b^2\\2ac&ad+bc&2bd\\c^2&cd&d^2\end{pmatrix}
\begin{pmatrix}0&1&0\\-1&0&0\\0&0&0\end{pmatrix}
\begin{pmatrix}a^2&2ac&c^2\\ab&ad+bc&cd\\b^2&2bd&d^2\end{pmatrix}\\
=&\begin{pmatrix}0&a^2(ad-bc)&ac(ad-bc)\\ *&0&c^2(ad-bc)\\ *&*&0\end{pmatrix}
\end{align*}
cannot be equal to $\left(\begin{array}{ccc}0&0&1\\0&0&0\\-1&0&0\end{array}\right)$ or $\left(\begin{array}{ccc}0&1&0\\-1&0&1\\0&-1&0\end{array}\right)$. Similarly,
\begin{align*}
&\begin{pmatrix}a^2&ab&b^2\\2ac&ad+bc&2bd\\c^2&cd&d^2\end{pmatrix}
 \begin{pmatrix}0&0&1\\0&0&0\\-1&0&0\end{pmatrix}
 \begin{pmatrix}a^2&2ac&c^2\\ab&ad+bc&cd\\b^2&2bd&d^2\end{pmatrix}\\
=&\begin{pmatrix}0&2ab(ad-bc)&(ad+bc)(ad-bc)\\ *&0&2cd(ad-bc)\\ *&*&0\end{pmatrix}
\end{align*}
cannot be equal to $\begin{pmatrix}0&1&0\\-1&0&1\\0&-1&0\end{pmatrix}$, and the claim follows.

Therefore we may assume that $V$ is congruent to one of the following:
\begin{description}
\item[Case 1.1]
$\begin{pmatrix}&d&&a&d\alpha\\-d&&d&b&a+d\beta\\&-d&&d\gamma&b+d\delta
\\-a&-b&-d\gamma&&d\epsilon\\-d\alpha&-a-d\beta&-b-d\delta&-d\epsilon&\end{pmatrix},$
\item[Case 1.2]
$\begin{pmatrix}&d&&a&d\alpha\\-d&&&b&a+d\beta\\&&&d\gamma&b+d\delta
\\-a&-b&-d\gamma&&d\epsilon\\-d\alpha&-a-d\beta&-b-d\delta&-d\epsilon&\end{pmatrix},$
\item[Case 1.3]
$\begin{pmatrix}&&d&a&d\alpha\\&&&b&a+d\beta\\-d&&&\gamma&b+d\delta
\\-a&-b&-d\gamma&&d\epsilon\\-d\alpha&-a-d\beta&-b-d\delta&-d\epsilon&\end{pmatrix},$
\item[Case 1.4]
$\begin{pmatrix}&&&a&d\alpha\\&&&b&a+d\beta\\&&&d\gamma&b+d\delta
\\-a&-b&-d\gamma&&d\epsilon\\-d\alpha&-a-d\beta&-b-d\delta&-d\epsilon&\end{pmatrix}.$
\end{description}

We will find out all possible congruence classes case by case.

\noindent\textbf{Case 1.1}\quad
In this case, we can  assume $\alpha=\beta=\gamma=\delta=0$. In fact, after taking the congruence transformation by $T_{52}(-\alpha)T_{43}(-\alpha)$ and $T_{42}(\gamma)T_{51}(\gamma)$, we can eliminate
$\alpha$ and $\gamma$. Now we may assume $\alpha=\gamma=0$, then we can take the congruence transformation by $T_{41}(-\delta)$ to eliminate $\gamma$ and by $T_{53}(-\beta)$ to eliminate $\beta$. Thus $V$ is congruent to
$$\begin{pmatrix}&d&&a&\\-d&&d&b&a\\&-d&&&b\\-a&-b&&&d\epsilon\\&-a&-b&-d\epsilon&\end{pmatrix}$$
for some  $\epsilon\in\{0,\pm1\}$. We denote by $V_{1.1.1}, V_{1.1.2}$ and $V_{1.1.3}$ the subspaces corresponding to
the cases $\epsilon=0$, $\epsilon=1$ and $\epsilon=-1$ respectively.

\noindent\textbf{Case 1.2}\quad
In this case, we can assume $\alpha=\beta=\delta=0$. By taking the congruence transformation by $T_{52}(-\alpha)T_{43}(-\alpha)$, $T_{51}(\beta)T_{42}(\beta)$ and $T_{41}(-\delta)$ we may eliminate $\alpha$, $\beta$ and $\gamma$ respectively.

If $\gamma\ne 0$, we may assume $\gamma=1$, otherwise we may take the congruence transformation by $D_3(-1)D_5(-1)$. Then \[V\sim\begin{pmatrix}&d&&a&\\-d&&&b&a\\&&&d&b\\-a&-b&-d&&d\epsilon\\&-a&-b&-d\epsilon&\end{pmatrix}
\overset{T_{53}(\epsilon)}\sim
\begin{pmatrix}&d&&a&\\-d&&&b&a\\&&&-d&b\\-a&-b&d&&\\&-a&-b&&\end{pmatrix}
.\]
We denote this subspace by $V_{1.2.0}$.

If $\gamma=0$, then $V$ is equivalent to
$$V\sim\begin{pmatrix}&d&&a&\\-d&&&b&a\\&&&&b\\-a&-b&&&d\epsilon\\&-a&-b&-d\epsilon&\end{pmatrix}$$
for some $\epsilon\in\{0,\pm1\}$. The subspaces corresponding to $\epsilon=0$, $\epsilon= 1$ and $\epsilon=-1$ are denoted by $V_{1.2.1}$, $V_{1.2.2}$ and $V_{1.2.3}$ respectively.

\noindent\textbf{Case 1.3}\quad
In this case, we can assume $\alpha=\beta=\gamma=\delta$. In fact, we can take congruence transformation by $T_{53}(-\alpha)$, $T_{43}(\beta)T_{52}(\beta)$, $T_{41}(\gamma)$ and $T_{51}(\delta)T_{42}(\delta)$ to
eliminate $\alpha$, $\beta$, $\gamma$ and $\delta$ respectively. Thus in this case, $V$ is equivalent to \[\begin{pmatrix}&&d&a&\\&&&b&a\\-d&&&&b\\-a&-b&&&d\epsilon\\&-a&-b&-d\epsilon&\end{pmatrix}\]
for some $\epsilon\in\{0,1\}$. Note that the case $\epsilon=1$ is equivalent to the one $\epsilon=-1$ by taking congruence transformation by $D_2(-1)D_4(-1)$. We denote the subspaces corresponding to $\epsilon=0$ and $\epsilon=1$ by $V_{1.3.1}$ and $V_{1.3.2}$ respectively.

\begin{remark}\label{rem-eqeg-5case1} We observe that $V_{1.1.2}\sim V_{1.1.3}\sim V_{1.2.0}\sim V_{1.3.2}$. In fact, we have
\begin{align*}
&\begin{pmatrix}0&-1&0&1&1\\-1&1&-1&0&0\\0&-1&0&-1&-1\\0&0&1&1&-1\\1&0&1&-1&1
\end{pmatrix}
\begin{pmatrix}&d&&a&\\-d&&d&b&a\\&-d&&&b\\-a&-b&&&d\\&-a&-b&-d&
\end{pmatrix}
\begin{pmatrix}0&-1&0&0&1\\-1&1&-1&0&0\\0&-1&0&1&1\\1&0&-1&1&-1\\1&0&-1&-1&1
\end{pmatrix}\\
=&\begin{pmatrix}&&-a-b&b-a-d&\\&&&d-a+b&b-a-d\\a+b&&&&d-a+b\\a+d-b&a-b-d&&&-a-b\\&a+d-b&a-b-d&a+b&
\end{pmatrix}
=\begin{pmatrix}&&d&a&\\&&&b&a\\-d&&&&b\\-a&-b&&&d\\&-a&-b&-d&
\end{pmatrix};
\end{align*}
\begin{align*}
&\begin{pmatrix}1&0&-1&0&1\\1&0&0&0&0\\1&0&-1&0&-1\\0&1&0&-1&0\\0&1&0&1&0
\end{pmatrix}
\begin{pmatrix}&d&&a&\\-d&&d&b&a\\&-d&&&b\\-a&-b&&&-d\\&-a&-b&d&
\end{pmatrix}
\begin{pmatrix}1&1&1&0&0\\0&0&0&1&1\\-1&0&-1&0&0\\0&0&0&-1&1\\1&0&-1&0&0
\end{pmatrix}\\
=&\begin{pmatrix}&&-b&a+d&\\&&&d-a&a+d\\b&&&&d-a\\-a-d&a-d&&&-b\\&-a-d&a-d&b&
\end{pmatrix}
=\begin{pmatrix}&&d&a&\\&&&b&a\\-d&&&&b\\-a&-b&&&d\\&-a&-b&-d&
\end{pmatrix};
\end{align*}
\begin{align*}
&\begin{pmatrix}0&1&0&0&0\\0&0&0&0&1\\1&0&0&0&0\\0&0&0&1&0\\0&0&-1&0&0
\end{pmatrix}
\begin{pmatrix}&&d&a&\\&&&b&a\\-d&&&&b\\-a&-b&&&d\\&-a&-b&-d&
\end{pmatrix}
\begin{pmatrix}0&0&1&0&0\\1&0&0&0&0\\0&0&0&0&-1\\0&0&0&1&0\\0&1&0&0&0
\end{pmatrix}\\
=&\begin{pmatrix}&a&&b&\\-a&&&-d&b\\&&&a&-d\\-b&d&-a&&\\&-b&d&&
\end{pmatrix}
=\begin{pmatrix}&d&&a&\\-d&&&b&a\\&&&d&b\\-a&-b&-d&&\\&-a&-b&&
\end{pmatrix}.
\end{align*}
\end{remark}

\noindent\textbf{Case 1.4}\quad
Note that if one of $\alpha$, $\beta$, $\gamma$, and $\delta$ is nonzero, then we can assume $\epsilon=0$.
In fact, when $\alpha$ (\emph{resp.} $\beta$, $\gamma$, $\delta$) is nonzero, we can take the congruence transformation by $T_{41}(-\alpha\epsilon)$ (\emph{resp.}   $T_{42}(\beta\epsilon)T_{51}(\beta\epsilon)$, $T_{53}(\gamma\epsilon)$,  $T_{43}(-\delta\epsilon)T_{52}(-\delta\epsilon)$) to eliminate $\epsilon$. Otherwise $\alpha=\beta=\gamma=\delta=0$, then we can make $\epsilon=1$ (taking the congruence transformation by $D_2(-1)D_4(-1)$ if necessary), and $V$ is equivalent to
$$\begin{pmatrix}&&&a&\\&&&b&a\\&&&&b\\-a&-b&&&d\\&-a&-b&-d&\end{pmatrix},$$ which is denoted by $V_{1.4.1}$.

Next we consider the case $\epsilon=0$. We define an equivalence relation on $\mathbb{Z}_3^4$ by $(\alpha,\beta,\gamma,\delta)\sim(\alpha',\beta',\gamma',\delta')$ if
$$\hat{X}\begin{pmatrix}a&d\alpha\\b&a+d\beta\\d\gamma&b+d\delta\end{pmatrix}X^{-1}=
\begin{pmatrix}a'&d'\alpha'\\b'&a'+d'\beta'\\d'\gamma'&b'+d'\delta'\end{pmatrix}$$
for some $X\in\text{GL}_2(\mathbb{Z}_3)$.
By applying the matrices $\begin{pmatrix}1&\\&-1\end{pmatrix}$, $\begin{pmatrix}&1\\1&\end{pmatrix}$, $\begin{pmatrix}1&1\\&1\end{pmatrix}$ and $\begin{pmatrix}1&\\1&1\end{pmatrix}$ iteratively,
we obtain that
\begin{align*}
(\alpha,\beta,\gamma,\delta)\sim&(-\alpha,-\beta,-\gamma,-\delta)
\sim(-\alpha,\beta,\gamma,-\delta)\sim(\gamma,-\delta,\alpha,-\beta)
\sim(-\gamma,-\delta,\alpha,\beta)\\
\sim&(\alpha+\beta-\gamma+\delta,\beta-\delta,\gamma,\delta)
\sim(\alpha-\beta+\gamma+\delta,\beta+\delta,\gamma,\delta)\\
\sim&(\alpha,\beta,\gamma-\alpha-\beta-\delta,\delta-\beta)
\sim(\alpha,\beta,\gamma+\alpha-\beta+\delta,\delta+\beta).
\end{align*}

If $\beta=\delta=0$, $(\alpha,0,\gamma,0)\sim(\alpha-\gamma,0,\gamma,0)\sim(\alpha,0,\gamma-\alpha,0)$. Since $(\alpha,0,\gamma,0)\sim(\gamma,0,\alpha,0)$, we can assume $\gamma\ne0$, then $(\alpha,0,\gamma,0)\sim(0,0,\gamma,0)\sim(0,0,1,0)$.

Otherwise, one of $\beta$ and $\delta$ is nonzero. Since $(\alpha,\beta,\gamma,\delta)\sim(\gamma,-\delta,\alpha,-\beta)$, we can assume $\beta\ne0$, then $(\alpha,\beta,\gamma,\delta)\sim(\alpha,\beta,\gamma',0)$.
Since $(\alpha,\beta,\gamma,0)\sim(\alpha+\beta-\gamma,\beta,\gamma,0)$, if $\beta\ne\gamma$, then $(\alpha,\beta,\gamma,0)\sim(0,\beta,\gamma,0)\sim(0,1,0,0)$ or $(0,1,-1,0)$.
If $\beta=\gamma$, since $(\alpha,\beta,\beta,0)\sim(-\alpha,\beta,\beta,0)$, we have $(\alpha,\beta,\beta,0)\sim(0,1,1,0)$ or $(1,1,1,0)$.

Then $(\alpha,\beta,\delta,\gamma)$ is equivalent to one of $(1,1,1,0)$, $(0,0,1,0)$, $(0,1,-1,0)$, $(0,1,0,0)$, $(0,1,1,0)$, and hence $V$ is
equivalent to one of the corresponding subspaces, say
\[\begin{pmatrix}&&&a&d\\&&&b&a+d\\&&&d&b\\-a&-b&-d&&\\-d&-a-d&-b&&\end{pmatrix},
\begin{pmatrix}&&&a&\\&&&b&d\\&&&&b\\-a&-b&&&\\&-d&-b&&\end{pmatrix},
\begin{pmatrix}&&&a&\\&&&b&a\\&&&d&b\\-a&-b&-d&&\\&-a&-b&&\end{pmatrix}
\]
\[\begin{pmatrix}&&&a&\\&&&b&a+d\\&&&-d&b\\-a&-b&d&&\\&-a-d&-b&&\end{pmatrix},
\begin{pmatrix}&&&a&\\&&&b&a+d\\&&&d&b\\-a&-b&-d&&\\&-a-d&-b&&\end{pmatrix}.
\]
We denote these subspaces by $V_{1.4.2}$, $V_{1.4.3}$, $V_{1.4.4}$, $V_{1.4.5}$ and $V_{1.4.6}$ respectively.

\subsection{Case 2: any two dimensional subspace of $V$ contains rank 2 matrices}\

In this subsection, we assume $V$ does not occur in Case 1, say any two dimensional subspace of $V$ has a element of rank 2. Then by Lemma \ref{lem-trivialradical}, we may assume $V$ contains $W_2$ as a subspace, and $V=\begin{pmatrix}&d\phi&d\theta&a&d\alpha\\-d\phi&&d\psi&b&d\beta\\
-d\theta&-d\psi&&d\gamma&a+d\delta\\-a&-b&-d\gamma&&d\epsilon\\
-d\alpha&-d\beta&-a-d\delta&-d\epsilon&\end{pmatrix}$ for some $\alpha, \beta, \gamma, \delta, \epsilon, \theta, \phi $ and $\psi\in \Z_3$. We write $V=[\alpha, \beta, \gamma, \delta, \epsilon, \theta, \phi, \psi]$ for short.

We recall a well-known result which follows easily from the Cauchy-Binet formula.

\begin{lemma}\label{lem-maxprincipalminor}
Let $X\in\Asym_5(\Z_3)$ be an anti-symmetric matrix. Then $\rank(X)=4$ if and only if $X$ has a nonzero principal minor of order 4.
\end{lemma}

We may draw the following consequence.
\begin{lemma} Assume that any 2-dimensional subspace of $V = [\alpha, \beta, \gamma, \delta, \epsilon, \theta, \phi, \psi]$ contains a matrix of rank 2. Then $\alpha=\theta=0$.
\end{lemma}

\begin{proof} If $\theta\neq0$, then $V\cong [0, \beta', 0, 0, \epsilon', \theta, \phi', 0]$ for some $\beta', \epsilon'$ and $\phi'$. In fact, we can take the congruence transformation by $T_{53}(-\alpha\theta)T_{41}(\gamma\theta)T_{21}(-\psi\theta)$) to eliminate $\alpha, \gamma$ and $\theta$, and then
take the congruence transformation by $T_{43}(-\delta\theta)T_{51}(-\delta\theta)$ to eliminate $\delta$. Thus without loss of generality, we may assume $V=[0, \beta, 0, 0, \epsilon, 1, \phi, 0]$, which reads as
\[\begin{pmatrix}
&d\phi&d&a&\\-d\phi&&&b&d\beta\\-d&&&&a\\-a&-b&&&d\epsilon\\&-d\beta&-a&-d\epsilon
\end{pmatrix}.
\]
The Pfaffians of the principal minors of order $4$ are $ba,d^2\epsilon-a^2$, $d^2\phi\epsilon-ad\beta$, $d\phi a-d^2\beta$ and $bd$. Clearly, if $\epsilon\ne0$ or $\beta\ne0$, then 2-dimensional subspace $\{a=b\}$ contains no matrices of rank 2, and if $\epsilon=\beta=0$, then the subspace $\{b=d\}$ contains no matrices of rank 2.

If $\alpha\ne0$, then by taking the congruence transformation by $T_{35}(1)$ or $T_{35}(-1)$, we show that $V$ is congruent to some $[\alpha, \beta, \gamma', \delta', \epsilon, \theta', \phi, \psi']$ with $\theta'\ne 0$, hence $V$ contains some 2-dimensional subspace which contains no matrices of rank 2.
\end{proof}

Now we may assume $\alpha=\theta=0$ thanks to the lemma. By more detailed analysis, we have the following result.
\begin{proposition}\label{prop-3-asym5.1} Assume any two dimensional subspace of $V$ contains a rank two element. Then $V$ is equivalent to one of
$V_{2.1}=[0,0,0,0,0,0,1,0]$, $V_{2.2}=[0,0,0,0,0,0,0,1]$, $V_{2.3}=[0,0,0,1,0,0,0,0]$ and $V_{2.4}=[0,0,0,0,1,0,0,0]$, say
\[V_{2.1}=\begin{pmatrix} &d&&a&\\-d&&&b&\\&&&&a\\-a&-b&&&\\&&-a&&
\end{pmatrix},\quad
V_{2.2}=\begin{pmatrix} &&&a&\\&&d&b&\\&-d&&&a\\-a&-b&&&\\&&-a&&
\end{pmatrix},
\]
\[V_{2.3}=\begin{pmatrix} &&&a&\\&&&b&\\&&&&d\\-a&-b&&&\\&&-d&&
\end{pmatrix}, \
V_{2.4}=\begin{pmatrix} &&&a&\\&&&b&\\&&&&a\\-a&-b&&&d\\&&-a&-d&
\end{pmatrix}.
\]
\end{proposition}
\begin{proof}
By the above lemma, we may assume $\alpha=\theta=0$. We have the following cases.

\noindent\textbf{Case 2.1: $\phi\ne0$.}\quad
By taking congruence actions by $T_{51}(\phi\beta)T_{43}(\phi\beta)$, $T_{45}(-\phi\psi)T_{31}(\phi\psi)$ and $T_{42}(\phi\delta)$ one after another, we show that
$V$ is equivalent to some $[0, 0, \gamma', 0, \epsilon', 0, \phi, 0]$. Thus without loss of
generality, we may assume $\phi=1$, $\beta=\delta=\psi=0$, and $V$ reads as
\[\begin{pmatrix}&d&&a&\\-d&&&b&\\&&&d\gamma&a\\-a&-b&-d\gamma&&d\epsilon\\&&-a&-d\epsilon&\end{pmatrix}.\]
The Pfaffians of the principal blocks of order 4 are $ab$, $a^2$, $d^2\epsilon$, $ad$, and $d^2\gamma$.
If $\epsilon\ne 0$ or $\gamma\ne0$, then the two dimensional subspace $\{b=0\}$ has no elements of rank 2, which
contradicts the assumption on $V$. Therefore, $V$ is congruent to $V_{2.1}=[0,0,0,0,0,0,1,0]$.

Note that the Pfaffians of the principal minors of $V_{2.1}$ of order 4 are $ab$, $a^2$, 0, $ad$, and $0$. Then clearly any nonzero element in the two dimensional subspace $\{a=0\}$ has rank $2$. Thus any two dimensional subspace of $V_{2.1}$ has an element of rank 2, since it will have nonzero intersection with the space $\{a=0\}$.

\noindent\textbf{Case 2.2: $\phi=0$, $\psi\ne0$.}\quad
In this case, we can take the congruence transformation by $T_{52}(-\psi\beta)T_{42}(\psi\gamma)$ to show that $V$ is congruent to
some $[0, 0, 0, \delta', \epsilon', 0, 0, \psi]$. Thus we may assume $\psi=1$ and $\phi=\beta=\gamma=0$, and $V$ reads as
\[\begin{pmatrix} &&&a&\\&&d&b&\\&-d&&&a+d\delta\\-a&-b&&&d\epsilon\\&&-a-d\delta&-d\epsilon&
\end{pmatrix}.
\]
The Pfaffians of the principal minors of order 4 are $d^2\epsilon-b(a+d\delta)$, $a(a+d\delta)$, $0$, $0$ and $ad$.
If $\epsilon\ne0$, then any matrix in the 2-dimensional subspace $\{b=0\}$ has rank 4;
if $\epsilon=0$ and $\delta\ne0$, then any matrix in the 2-dimensional subspace $\{b=d\}$ has rank 4.
It forces that $\epsilon=\delta=0$, and hence $V$ is congruent to $V_{2.2}=[0,0,0,0,0,0,0,1]$.

Now the Pfaffians of the principal minors of $V_{2.2}$ of order 4 are $-ba$, $a^2$, $0$, $0$ and $ad$. Then any nonzero elements of the subspace $\{a=0\}$ of $V_{2.2}$ has rank 2, and hence any two dimensional subspace of $V_{2.2}$ contains an element of rank 2.

\noindent\textbf{Case 2.3: $\phi=\psi=0$.} We may assume $\beta=0$. In fact, if $\beta\ne0$, then we take the congruence transformation by $T_{25}(1)$ to make $\psi$ nonzero, and it reduces to the Case 2.2. Now $V=[0, 0, \gamma, \delta, \epsilon, 0, 0, 0]$ reads as
\[V=\begin{pmatrix}&&&a&\\&&&b&\\
&&&d\gamma&a+d\delta\\-a&-b&-d\gamma&&d\epsilon\\
&&-a-d\delta&-d\epsilon&\end{pmatrix}.\]

Now we have two subcases.

(i) $\delta\ne0$. After taking the congruence transformations by $T_{31}(-\gamma\delta)T_{45}(-\gamma\delta)$ and $T_{43}(-\epsilon\delta)T_{51}(-\epsilon\delta)$, we show that $V$ is equivalent to
$[0, 0, 0, \delta, 0, 0, 0, 0]$, and hence to $V_{2.3}=[0, 0, 0, 1, 0, 0, 0, 0]$.

The Pfaffians of the principal minors of $[0, 0, 0, 1, 0, 0, 0, 0]$ of order 4 are $b(a+d)$, $a(a+d), 0, 0, 0$. Thus any nonzero element in the subspace $\{a+d=0\}$ of $[0, 0, 0, 1, 0, 0, 0, 0]$ has rank 2, and hence any two dimensional subspace contains an element of rank 2.

(ii) $\delta=0$. Write $(\gamma, \epsilon)= [0, 0, \gamma, 0, \epsilon, 0, 0, 0]$ for short. Then it is easy to  show that $(\gamma, \epsilon)\overset{T_{35}(1)}\sim (\gamma-\epsilon, \epsilon)$ and $(\gamma, \epsilon)\overset{T_{53}(1)}\sim(\gamma, \epsilon-\gamma)$. Applying these two equivalences iteratively, we show that $V$ is equivalent to $V_{2.4}=(0, 1)= [0, 0, 0, 0, 1, 0, 0, 0]$.

The Pfaffians of the principal minors of $V_{2.4}$ of order 4 are $ba$, $a^2, 0, 0, 0$. Then any nonzero element in the subspace $\{a=0\}$ of $V_{2.4}$ has rank 2, and hence any two dimensional subspace contains an element of rank 2.
\end{proof}

We denote the subspaces in the above proposition by $V_{2.1}, V_{2.2}, V_{2,3}$ and $V_{2.4}$ respectively.

\begin{proposition}\label{prop-3-asym5.2} $V_{2.1}, V_{2.2}, V_{2,3}$ and $V_{2.4}$ are not congruent to each other.
\end{proposition}

\begin{proof} For any subspace $V\subseteq \Asym_5(\Z_3)$, we set $N(V)=\#\{X\in V\mid \rank(X)\ne 4\}$. Then $N(V)$ is invariant under congruence equivalence. By Lemma \ref{lem-maxprincipalminor}, $X$ has rank 2 if and only if all its principal minor of order 4 is 0. Easy calculation shows that
$N(V_{2.1})=N(V_{2.2})=9$, $N(V_{2.3})=11$, $N(V_{2.4})=9$. Thus $V_{2.3}$ is not congruent to other three.

For short we write $\perp_1=\perp_{V_{2.1}}$, $\perp_2=\perp_{V_{2.2}}$ and $\perp_4=\perp_{V_{2.4}}$ temporarily.
Set \[\mathbf U= (\Z_3, \Z_3, 0, 0, 0)= \{(u_1, u_2, 0, 0, 0)\in\Z_3^5\mid u_1, u_2\in \Z_3\}, \mathbf V= (\Z_3, \Z_3, 0, 0, \Z_3).\] It is easy to show that $\mathbf U^{\perp_4}=(\Z_3,\Z_3,0,0,\Z_3)$ has dimension $4$, and
$\mathbf V^{\perp_2}=(\Z_3,\Z_3,0,0,\Z_3)$ has dimension 3. Moreover, for any $\mathbf W\subseteq \Z_3^5$, if $\dim(W)=2$, then $\dim(W^{\perp_1})\le 3$ and $\dim(W^{\perp_2})\le 3$; if  $\dim(W)=3$, then $\dim(W^{\perp_1})<3$.
Then the assertion follows from Remark \ref{rem-complement-inv}.
\end{proof}

\subsection{The classification result.}
We have shown that any three dimensional subspaces of $\Asym_5(\Z_3)$ is equivalent to one as listed in Section 10.2 and 10.3. Moreover, the subspaces $V_{2.1}$, $V_{2.2}$, $V_{2.3}$ and $V_{2.4}$ as in Proposition \ref{prop-3-asym5.1} are pairwise noncongruent, and they are not congruent to the ones given in Section 10.2. Now to complete our classification, it suffices to show that the ones in Case 1.1 through 1.4 are pairwise noncongruent, except the ones as listed in Remark \ref{rem-eqeg-5case1}.

\begin{proposition}\label{prop-noneq-5case1}  The 12 subspaces $V_{1.1.1}$, $V_{1.1.2}$, $V_{1.2.1}$, $V_{1.2.2}$, $V_{1.2.3}$, $V_{1.3.1}$, $V_{1.4.1}$, $V_{1.4.2}$, $V_{1.4.3}$, $V_{1.4.4}$, $V_{1.4.5}$, $V_{1.4.6}$ are not congruent to each other.
\end{proposition}

\begin{proof} First the Pfaffians of the principal minors of order 4 and the number of elements with rank strictly less than 4 are listed as follows.
\begin{align*}
&V_{1.1.1}: -b^2, ab, -a^2, bd, ad; &N= 3.\\
&V_{1.1.2}: d^2-b^2, ab, d^2-a^2, bd, ad; &N= 1.\\
&V_{1.2.1}: -b^2, -ab, - a^2, bd, 0;  &N= 3.\\
&V_{1.2.2}: -b^2, -ab, d^2- a^2, bd, 0; &N= 5.\\
&V_{1.2.3}: -b^2, -ab, -d^2- a^2, bd, 0; &N= 1.\\
&V_{1.3.1}: -b^2, -ab, -a^2, -da, -bd; &N= 3.\\
&V_{1.4.1}: -b^2, -ab, -a^2, 0, 0; &N=3.\\
&V_{1.4.2}: d(a+d)-b^2, d^2-ad, bd-a(a+d), 0, 0; &N= 1.\\
&V_{1.4.3}: -b^2, -ab, -ad, 0, 0 ; &N=5.\\
&V_{1.4.4}: ad-b^2, -ab, -a^2, 0, 0; &N= 3.\\
&V_{1.4.5}: -d(a+d)-b^2, -ab, -a(a+d), 0, 0; &N= 3.\\
&V_{1.4.6}: d(a+d)-b^2, -ab, -a(a+d), 0, 0; &N= 7.
\end{align*}
We listed the spaces according to $N(V)$ as follows:

$N=1, V= V_{1.1.2}\sim V_{1.2.0}, V_{1.2.3}, V_{1.4.2}:$
\begin{align*}
&V_{1.1.2}=\begin{pmatrix}&d&&a&\\-d&&d&b&a\\&-d&&&b\\-a&-b&&&d\\&-a&-b&-d&\end{pmatrix}
\sim \begin{pmatrix}&d&&a&\\-d&&&b&a\\&&&d&b\\-a&-b&-d&&\\&-a&-b&&
\end{pmatrix}= V_{1.2.0}, \\
& V_{1.2.3}=  \begin{pmatrix}&d&&a&\\-d&&&b&a\\&&&&b\\-a&-b&&&-d\\&-a&-b&d&\end{pmatrix},
 V_{1.4.2} \begin{pmatrix}&&&a&d\\&&&b&a+d\\&&&d&b\\-a&-b&-d&&\\-d&-a-d&-b&&\end{pmatrix}.
\end{align*}

$N=3, V= V_{1.1.1}, V_{1.2.1}, V_{1.3.1}, V_{1.4.1}, V_{1.4.4}, V_{1.4.5}:$
\begin{align*}
&V_{1.1.1}= \begin{pmatrix}&d&&a&\\-d&&d&b&a\\&-d&&&b\\-a&-b&&&\\&-a&-b&&\end{pmatrix},
&V_{1.2.1}=\begin{pmatrix}&d&&a&\\-d&&&b&a\\&&&&b\\-a&-b&&&\\&-a&-b&&\end{pmatrix},\\
&V_{1.3.1}=\begin{pmatrix}&&d&a&\\&&&b&a\\-d&&&&b\\-a&-b&&&\\&-a&-b&&\end{pmatrix},
&V_{1.4.1}=\begin{pmatrix}&&&a&\\&&&b&a\\&&&&b\\-a&-b&&&d\\&-a&-b&-d&\end{pmatrix},\\
&V_{1.4.4}=\begin{pmatrix}&&&a&\\&&&b&a\\&&&d&b\\-a&-b&-d&&\\&-a&-b&&\end{pmatrix},
&V_{1.4.5}=\begin{pmatrix}&&&a&\\&&&b&a+d\\&&&-d&b\\-a&-b&d&&\\&-a-d&-b&&\end{pmatrix}.
\end{align*}

$N=5, V= V_{1.2.2}, V_{1.4.3}:$
\[V_{1.2.2}=\begin{pmatrix}&d&&a&\\-d&&&b&a\\&&&&b\\-a&-b&&&d\\&-a&-b&-d&\end{pmatrix},\qquad \qquad
V_{1.4.3}=\begin{pmatrix}&&&a&\\&&&b&d\\&&&&b\\-a&-b&&&\\&-d&-b&&\end{pmatrix}.
\]

$N=7, V= V_{1.4.6}:$
\[V_{1.4.6}=\begin{pmatrix}&&&a&\\&&&b&a+d\\&&&d&b\\-a&-b&-d&&\\&-a-d&-b&&\end{pmatrix}.\]

Clearly subspaces with different number of rank 4 elements will not be congruent. Thus the subspace $V_{1.4.6}$ is not congruent to others. Now we discuss on $N$ case by case.

\noindent{\bf Case N=1:}

$V_{1.4.2}$ is not congruent to $V_{1.1.2}$ and $V_{1.2.3}$, since $\Z_3^5 = (\Z_3,\Z_3,\Z_3, 0, 0)\oplus (0,0,0,\Z_3,\Z_3)$ is a direct sum of a 3-dimensional isotropy subspace and a 2-dimensional isotropy subspace with respect to $V_{1.4.2}$, while for $V_{1.1.2}$ and $V_{1.2.3}$ there is no such a decomposition.

Moreover, $\Z_3^5$ has a pair of 2-dimensional isotropy subspace with respect to $V_{1.2.0}$, say $(0,\Z_3,\Z_3, 0, 0)$ and $(0,0,0,\Z_3,\Z_3)$, whose intersection is zero, while for $V_{1.2.3}$, there are no such isotopy subspaces. Hence $V_{1.1.2}$ is not congruent to $V_{1.2.3}$.

\noindent{\bf Case N=5:}

 $\Z_3^5 = (\Z_3,\Z_3,\Z_3, 0, 0)\oplus (0,0,0,\Z_3,\Z_3)$ is a direct sum of isotropy subspaces with respect to $V_{1.4.3}$, while for $V_{1.2.2}$ there is no such a decomposition. Thus $V_{1.2.2}$ and
$V_{1.4.3}$ are not congruent.

\noindent{\bf Case N=3:}

Similar argument as above shows that $V_{1.4.4}$ and $V_{1.4.5}$ are not congruent to $V_{1.1.1}$, $V_{1.2.1}$, $V_{1.3.1}$ and $V_{1.4.1}$.  Moreover,
\[V_{1.4.5}\overset{T_{31}(-1)}\sim
\begin{pmatrix}&&&a&\\&&&b&a+d\\&&&-(a+d)&b\\-a&-b&a+d&&\\&-a-d&-b&&\end{pmatrix}
=\begin{pmatrix}&&&a&\\&&&b&d\\&&&-d&b\\-a&-b&d&&\\&-d&-b&&\end{pmatrix},\]
and the latter contains a subspace $W=\begin{pmatrix}&&&&\\&&&b&d\\&&&-d&b\\&-b&d&&\\&-d&-b&&\end{pmatrix}$, such that $\mathrm{rad}(W)\ne 0$ and every nonzero element of $W$ has rank 4. There is no difficulty to show that $V_{1.4.4}$ does not have such a two dimensional subspace, therefore $V_{1.4.4}$ is not congruent to $V_{1.4.5}$.

Note that $V_{1.4.1}$ has a three dimensional isotropy subspace $(\Z_3, \Z_3, \Z_3, 0, 0)$, while $V_{1.1.1}$, $V_{1.2.1}$ and $V_{1.3.1}$ do not have, thus $V_{1.4.1}$ is not congruent to the others.

$V_{1.2.1}$ has a two dimensional subspace $\{b=0\}$ whose radical is $(0,0,\Z_3,0,0)\ne 0$, while any two dimensional subspace of $V_{1.1.1}$ or $V_{1.3.1}$ has trivial radical, thus $V_{1.2.1}$ is neither congruent to $V_{1.1.1}$ nor to $V_{1.3.1}$.

Now we are left to prove that  $V_{1.1.1}$ and $V_{1.3.1}$ are not congruent. Otherwise, if $V_{1.3.1}$ is congruent to $V_{1.1.1}$, then there exists some $P\in \GL_5(\Z_3)$, such that $PXP^T\in V_{1.1.1}$ for any $X\in V_{1.3.1}$.
We write
 \[D=\begin{pmatrix}&&1\\&&\\-1&&\end{pmatrix}, D'=\begin{pmatrix}&1&\\-1&&1\\&-1&\end{pmatrix},
E(a,b)=\begin{pmatrix}a&\\b&a\\&b\end{pmatrix}\]
for $a, b\in \Z_3$. Then
$P\begin{pmatrix}dD &E(a,b)\\ -E(a,b)^T&\end{pmatrix}P^T =\begin{pmatrix}d'D'& E(a',b')\\ -E(a',b')^T&\end{pmatrix}$, and $(a,b,d)=(a',b',d')Q$ for some $Q\in\GL_3(\Z_3)$. Write $P$ as a block matrix, say $P=\begin{pmatrix}P_1& P_2\\ P_3&P_4 \end{pmatrix}$, where $P_1, P_2, P_3$ and $P_4$ are $3\times 3$, $3\times 2$, $2\times 3$ and $2\times 2$ matrices respectively.

By comparing the rank, $P\begin{pmatrix}D & \\ &\end{pmatrix}P^T =\begin{pmatrix}d'D'& \\  &\end{pmatrix}$ for some $d'$. Then $\begin{pmatrix}P_1\\P_3\end{pmatrix}DP_3^T=0$, hence $DP_3^T=0$, $P_3=\left(\begin{array}{ccc}0&r_1&0\\0&r_2&0\end{array}\right)$.
Let $P_4=\left(\begin{array}{cc}s_{11}&s_{12}\\s_{21}&s_{22}\end{array}\right)$, then from
$$\begin{pmatrix}P_1&P_2\\P_3&P_4\end{pmatrix}\begin{pmatrix}dD&E(a,b)\\-E(a,b)^T&\end{pmatrix}
\begin{pmatrix}P_1^T&P_3^T\\P_2^T&P_4^T\end{pmatrix}=\begin{pmatrix}d'D'&E(a',b')\\-E(a',b')^T&\end{pmatrix},$$
we get $P_3E(a,b)P_4^T=P_4E(a,b)^TP_3^T$ for any $a, b$.
Thus $r_1s_{21}=r_2s_{11}$, $r_1s_{22}=r_2s_{12}$ and
$(P_3,P_4)=\left(\begin{array}{ccccc}0&r_1&0&s_{11}&s_{12}\\0&r_2&0&s_{21}&s_{22}\end{array}\right)$.

If $(r_1,r_2)\ne(0,0)$, then $\rank(P_3,P_4)=1$, which leads to a contradiction.
Therefore $r_1=r_2=0$, $P_3=0$. Hence $P_1(dD)P_1^T=d'D'$ and $P_1E(a,b)P_4^T=E'(a',b')$. By Remark \ref{rem-E-stabilizer}, $D'$ is a multiple of some $\hat X$, and $P_1DP_1^T= \hat X D \hat X^T$, contradicts to the fact shown in Section 10.2 that $D$ and $D'$ are not equivalent under the congruence action by $\hat X$'s.
\end{proof}

In summary, we have the following classification.

\begin{theorem}\label{thm-3of5}
There are 16 congruence classes of 3-dimensional subspaces of $\Asym_5(\Z_3)$ with trivial radical:
\[
\begin{pmatrix}&d&&a&\\-d&&d&b&a\\&-d&&&b\\-a&-b&&&\\&-a&-b&&\end{pmatrix},
\begin{pmatrix}&d&&a&\\-d&&d&b&a\\&-d&&&b\\-a&-b&&&d\\&-a&-b&-d&\end{pmatrix},
\begin{pmatrix}&d&&a&\\-d&&&b&a\\&&&&b\\-a&-b&&&\\&-a&-b&&\end{pmatrix},
\]
\[
\begin{pmatrix}&d&&a&\\-d&&&b&a\\&&&&b\\-a&-b&&&d\\&-a&-b&-d&\end{pmatrix},
\begin{pmatrix}&d&&a&\\-d&&&b&a\\&&&&b\\-a&-b&&&-d\\&-a&-b&d&\end{pmatrix},
\begin{pmatrix}&&d&a&\\&&&b&a\\-d&&&&b\\-a&-b&&&\\&-a&-b&&\end{pmatrix},
\]
\[
\begin{pmatrix}&&&a&\\&&&b&a\\&&&&b\\-a&-b&&&d\\&-a&-b&-d&\end{pmatrix},
\begin{pmatrix}&&&a&d\\&&&b&a+d\\&&&d&b\\-a&-b&-d&&\\-d&-a-d&-b&&\end{pmatrix},
\begin{pmatrix}&&&a&\\&&&b&d\\&&&&b\\-a&-b&&&\\&-d&-b&&\end{pmatrix},
\]
\[
\begin{pmatrix}&&&a&\\&&&b&a\\&&&d&b\\-a&-b&-d&&\\&-a&-b&&\end{pmatrix},
\begin{pmatrix}&&&a&\\&&&b&a+d\\&&&-d&b\\-a&-b&d&&\\&-a-d&-b&&\end{pmatrix},
\begin{pmatrix}&&&a&\\&&&b&a+d\\&&&d&b\\-a&-b&-d&&\\&-a-d&-b&&\end{pmatrix},
\]
\[
\begin{pmatrix} &d&&a&\\-d&&&b&\\&&&&a\\-a&-b&&&\\&&-a&&\end{pmatrix},
\begin{pmatrix} &&&a&\\&&d&b&\\&-d&&&a\\-a&-b&&&\\&&-a&&\end{pmatrix},
\]
\[
\begin{pmatrix} &&&a&\\&&&b&\\&&&&d\\-a&-b&&&\\&&-d&&\end{pmatrix},
\begin{pmatrix} &&&a&\\&&&b&\\&&&&a\\-a&-b&&&d\\&&-a&-d&\end{pmatrix},
\]
and 6 congruence classes of  3-dimensional subspaces with nontrivial radical:
$$\begin{pmatrix}&c&a&&\\-c&&b&&\\-a&-b&&&\\&&&0&\\&&&&0\end{pmatrix},
\begin{pmatrix}&&a&&\\&&b&&\\-a&-b&&c&\\&&-c&&\\&&&&0\end{pmatrix},
\begin{pmatrix}&c&a&&\\-c&&b&&\\-a&-b&&c&\\&&-c&&\\&&&&0\end{pmatrix},$$
$$\begin{pmatrix}&&a&c&\\&&b&&\\-a&-b&&&\\-c&&&&\\&&&&0\end{pmatrix},
\begin{pmatrix}&c&a&&\\-c&&&b&\\-a&&&c&\\&-b&-c&&\\&&&&0\end{pmatrix},
\begin{pmatrix}&c&a&b&\\-c&&&a&\\-a&&&-c&\\-b&-a&c&&\\&&&&0\end{pmatrix}.$$
\end{theorem}

\vskip20pt
\subsection*{Acknowledgments}
Z. Wan is supported by the Shuimu Tsinghua Scholar Program. Y. Ye is partially supported by the Natural
Science Foundation of China (Grant No.11971449).
C. Zhang is supported by the Fundamental Research Funds for the Central Universities (No. 2020QN20).

\vskip20pt

\begin{thebibliography}{99}

\bibitem{ars}
 M. Auslander, I. Reiten, S. O. Smal{\o}, Representation theory of algebras, Cambridge Studies in Advanced Mathematics 36, Cambridge University Press, 1997.

\bibitem{br} K. Brown, Cohomology of groups, GTM 87, Springer-Verlag, New York, 1982.

\bibitem{gan}
F. R. Gantmacher, The theory of matrices, Vol. 2, Chelsea Publ., New York, 1959.

\bibitem{hs} G. Hochschild, J. P. Serre, Cohomology of group extensions, Trans. Amer. Math. Soc.,
74 (1953), 110-134.

\bibitem{hlyy}
H. L. Huang, G. X. Liu, Y. P. Yang, Y. Ye, Finite quasi-quantum groups of diagonal type, J. Reine. Angew. Math., 759 (2020), 201-243.

\bibitem{hwy}
H. L. Huang, Z. Y. Wan, Y. Ye, Explicit cocycle formulas on finite abelian groups with applications to braided linear Gr-categories and Dijkgraaf–Witten invariants, P. Royal Society Edinburgh Section A: Mathematics, 2019, doi.org/10.1017/prm.2019.15.

\bibitem{HU}
B. Huppert, Endliche Gruppen, Springer-Verlag, Berlin, 1967.

\bibitem{lyn} R. C. Lyndon, The cohomology theory of group extensions, Duke Math. J., 15 (1948),
271–292.


\bibitem{khu}
E. I. Khukhro, $p$-automorphisms of finite $p$-groups, London Mathematical Society Lecture Note Series,
vol. 246, Cambridge University Press, Cambridge, 1998.

\bibitem{new} M. F. Newman, Determination of groups of prime-power order, Group theory (Proc. Miniconf., Australian Nat. Univ., Canberra, 1975), pp. 73-84, Lecture Notes in Math., Vol. 573, Springer, Berlin, 1977.


\bibitem{ob}
E. A. O'Brien, The $p$-group generation algorithm, J. Symbolic Computation, 9(1990), 677-698.

\bibitem{rot} J. J. Rotman, An introduction to the theory of groups, Vol. 148. Springer Science \& Business Media, 2012.

\bibitem{schar}
R. Scharlau, Paare alternierender Formen, Math. Z., 147(1)(1976), 13-19.

\bibitem{lee1}
M. Vaughan-Lee, The automorphisms of class two groups of prime exponent, arXiv:1501.00678.

\bibitem{lee2}
M. Vaughan-Lee, Groups of order $p^8$ and exponent $p$, Int. J. Group Theory, 4(2015), 25-42.
\bibitem{lee3}
M. Vaughan-Lee, The class three groups of order $p^9$ with exponent $p$, arXiv:1702.04898v1.
\bibitem{lee4}
M. Vaughan-Lee, Orbits of irreducible binary forms over $\text{GF}(p)$, arXiv:1705.07418v1.

\bibitem{vish}
A. L. Vishnevetskii, A system of invariants for some groups of class 2 with commutator subgroup of rank two, Int. Ukrain. Mat. Zh., 37(1985), 294-300.

\bibitem{wei}
C. A. Weibel, An introduction to homological algebra, Cambridge Studies in Advanced Mathematics, 38. Cambridge University Press, Cambridge, 1994.

\bibitem{wil}
D. Wilkinson, The group of exponent $p$ and order $p^7$ ($p$ any prime), J. Algebra, 118(1988), 109-119.
\end{thebibliography}
\end{document}